\documentclass[11pt]{article}

\usepackage{amsmath,amssymb,amsfonts,amsthm,enumitem,pstricks,float,xy,mathrsfs,accents}
\xyoption{all}

\newtheorem*{thm*}{Theorem}

\numberwithin{equation}{section} 
\newtheorem{thm}{Theorem}[section]
\newtheorem{prp}[thm]{Proposition} 
\newtheorem{lmm}[thm]{Lemma}  
\newtheorem{crl}[thm]{Corollary}
\newtheorem*{prp*}{Proposition}
\theoremstyle{definition}

\newtheorem{eg}[thm]{Example}
\newtheorem{rmk}[thm]{Remark}

\newtheorem{cnj}[thm]{Conjecture}

\DeclareFontFamily{U}{mathx}{\hyphenchar\font45}
\DeclareFontShape{U}{mathx}{m}{n}{
      <5> <6> <7> <8> <9> <10>
      <10.95> <12> 
      mathx10
      }{}
\DeclareSymbolFont{mathx}{U}{mathx}{m}{n}
\DeclareFontSubstitution{U}{mathx}{m}{n}
\DeclareMathAccent{\widecheck}{0}{mathx}{"71}

\def\BE#1{\begin{equation}\label{#1}}
\def\EE{\end{equation}}
\def\eref#1{(\ref{#1})}

\def\lra{\longrightarrow}

\def\xra#1{\xrightarrow{\hspace*{#1cm}}}

\def\lr#1{\langle{#1}\rangle}
\def\blr#1{\big\langle{#1}\big\rangle}
\def\ov#1{\overline#1}

\def\wc#1{\widecheck{#1}}
\def\wt#1{\widetilde{#1}}
\def\wh#1{\widehat{#1}}
\def\tn#1{\textnormal{#1}}
\def\sm#1{\begin{small}#1\end{small}}
\def\ori#1{\accentset{\circ}{#1}}
\def\obu#1{\accentset{\bu}{#1}}
\def\sf#1{\textsf{#1}}

\def\al{\alpha}
\def\be{\beta}
\def\de{\delta}

\def\ga{\gamma}
\def\io{\iota}
\def\ka{\kappa}
\def\la{\lambda}

\def\om{\omega}
\def\si{\sigma}
\def\th{\theta}
\def\vph{\varphi}

\def\Ga{\Gamma}
\def\De{\Delta}

\def\Th{\Theta}

\def\C{\mathbb C}

\def\bE{\mathbb E}

\def\cF{\mathcal F}
\def\cH{\mathcal H}

\def\bK{\mathbb K}
\def\cL{\mathcal L}

\def\M{\mathcal M}
\def\fM{\mathfrak M}

\def\cN{\mathcal N}
\def\cO{\mathcal O}
\def\P{\mathbb P}

\def\cR{\mathcal R}
\def\fR{\mathfrak R}
\def\Q{\mathbb Q}
\def\R{\mathbb R}

\def\bS{\mathbb S}
\def\T{\mathbb T}

\def\fX{\mathfrak X}
\def\cZ{\mathcal Z}
\def\Z{\mathbb Z}

\def\bs{\mathbf s}

\def\bff{\mathbf f}
\def\ne{\textnormal{e}}
\def\fI{\mathfrak i}

\def\fc{\mathfrak c}
\def\ff{\mathfrak f}
\def\fs{\mathfrak s}
\def\x{\mathbf x}
\def\z{\mathbf z}

\def\0{\mathbf 0}

\def\Si{\Sigma}

\def\Aut{\tn{Aut}}

\def\Contr{\tn{Contr}}
\def\deg{\textnormal{deg}}
\def\Deck{\tn{Deck}}
\def\Eff{\tn{Eff}}
\def\ev{\textnormal{ev}}

\def\Flux{\tn{Flux}}

\def\GW{\textnormal{GW}}

\def\id{\textnormal{id}}
\def\Im{\textnormal{Im}}

\def\Obs{\tn{Obs}}
\def\ord{\textnormal{ord}}
\def\pt{\tn{pt}}
\def\PD{\textnormal{PD}}

\def\vir{\textnormal{vir}}

\def\i{\infty}

\def\eset{\emptyset}
\def\nd{\textnormal{d}}
\def\bu{\bullet}

\begin{document}

\title{On the Refined Symplectic Sum Formula\\ 
for Gromov-Witten Invariants}
\author{Mohammad F.~Tehrani and Aleksey Zinger\thanks{Partially 
supported by NSF grant 0846978}}
\date{\small\today}
\maketitle

\begin{abstract}
\noindent
We describe the extent to which Ionel-Parker's proposed refinement of 
the standard relative Gromov-Witten invariants  sharpens
the usual symplectic sum formula.
The key product operation on the target spaces for the refined invariants 
is specified in terms of abelian covers of symplectic divisors,
making it suitable for studying from a topological perspective.
We give several qualitative applications of this refinement,
which include vanishing results for  Gromov-Witten invariants.
\end{abstract}

\tableofcontents

\section{Introduction}
\label{intro_sec}

\noindent
Gromov-Witten invariants of symplectic manifolds, which include nonsingular projective varieties, 
are certain counts of pseudo-holomorphic curves that play 
prominent roles in symplectic topology, algebraic geometry, and string theory.
The decomposition formulas, known as symplectic sum formulas 
in symplectic topology and degeneration formulas in algebraic geometry, 
are one of the main tools used to compute Gromov-Witten invariants;
they relate Gromov-Witten invariants of one symplectic manifold 
to Gromov-Witten invariants of two simpler symplectic manifolds.
Unfortunately, the formulas of \cite{Jun2,LR} do not completely determine 
the former in terms of the latter
in many cases because of the so-called \sf{vanishing cycles}:
second homology classes in the first manifold which vanish when projected 
to the union of the other two manifolds; see~\eref{cRXYVprop_e}.
A~refinement to the usual relative Gromov-Witten invariants of \cite{Jun1,LR} 
is sketched in~\cite{IPrel};
the aim of this refinement is to resolve the unfortunate deficiency of 
the  formulas of \cite{Jun2,LR} in~\cite{IPsum}.
In~\cite{GWrelIP}, we formally constructed  
the refinement to relative invariants suggested in~\cite{IPrel} and 
discussed the invariance and computability aspects of the resulting curve counts.
In this paper, we 
describe the extent to which it sharpens the usual symplectic sum formula
and obtain some qualitative applications.

\subsection{Relative GW-invariants}
\label{RelGW_subs0}

\noindent
Let $(X,\om)$ be a compact symplectic manifold and $J$ be an $\om$-tame almost complex structure on~$X$.
For $g,k\!\in\!\Z^{\ge0}$ and $A\!\in\!H_2(X;\Z)$, we denote by $\ov\fM_{g,k}(X,A)$ 
the moduli space of stable $J$-holomorphic $k$-marked degree~$A$ maps from connected nodal 
curves of genus~$g$.
By \cite{LT,FO,BF}, this moduli space carries a virtual class, which is independent of~$J$
and of representative~$\om$ in a deformation equivalence class 
of symplectic forms on~$X$.
If $V\!\subset\!X$ is a compact symplectic divisor (symplectic submanifold of real codimension~2), 
$\ell\!\in\!\Z^{\ge0}$, $\bs\!\equiv\!(s_1,\ldots,s_{\ell})$
is an $\ell$-tuple of positive integers such~that 
\BE{bsumcond_e} s_1+\ldots+s_{\ell}=A\cdot V,\EE
and $J$ restricts to an almost complex structure on $V$, let 
$\ov\fM_{g,k;\bs}^V(X,A)$  denote the moduli space of 
stable $J$-holomorphic $(k\!+\!\ell)$-marked maps from connected nodal curves of genus~$g$ 
that have contact with~$V$ at the last $\ell$ marked points of orders $s_1,\ldots,s_{\ell}$.
According to \cite{LR,Jun1}, this moduli space carries a virtual class, 
which is independent of~$J$ and of representative~$\om$ in a deformation equivalence class 
of symplectic forms on~$(X,V)$.\\

\noindent
There are natural \sf{evaluation morphisms}
\begin{alignat}{1}\label{evdfn_e1}
\ev_X\!\equiv\!\ev_1\!\times\!\ldots\!\times\!\ev_k\!: 
\ov\fM_{g,k}(X,A),\ov\fM_{g,k;\bs}^V(X,A)&\lra X^k,\\
\label{evdfn_e2}
\ev_X^V\!\equiv\!\ev_{k+1}\!\times\!\ldots\!\times\!\ev_{k+\ell}:
\ov\fM_{g,k;\bs}^V(X,A)&\lra V_{\bs}\equiv V^{\ell},
\end{alignat}
sending each stable map to its values at the marked points.
The (\textsf{absolute}) \textsf{GW-invariants of~$(X,\om)$}
are obtained by pulling back elements of $H^*(X^k;\Q)$ by the morphism~\eref{evdfn_e1}
and integrating them and other natural classes on  $\ov\fM_{g,k}(X,A)$
against the virtual class of $\ov\fM_{g,k}(X,A)$.
The (\textsf{relative}) \textsf{GW-invariants of~$(X,V,\om)$}
are obtained by pulling back elements of $H^*(X^k;\Q)$ and $H^*(V_{\bs};\Q)$
by the morphisms~\eref{evdfn_e1} and~\eref{evdfn_e2}, and integrating them and 
other natural classes on  $\ov\fM_{g,k;\bs}^V(X,A)$
against the virtual class of $\ov\fM_{g,k;\bs}^V(X,A)$.\\

\noindent
As emphasized in \cite[Section~5]{IPrel}, two preimages of the same point in~$V_{\bs}$
under~\eref{evdfn_e2} determine an element~of 
\BE{cRXVdfn_e0}\cR_X^V\equiv \ker\big\{\io_{X-V*}^X\!:\,H_2(X\!-\!V;\Z)\lra H_2(X;\Z)\big\},\EE
where $\io_{X-V}^X\!:X\!-\!V\!\lra\!X$ is the inclusion; 
see \cite[Section~2.1]{GWrelIP}.
The elements of~$\cR_X^V$, called \sf{rim tori} in~\cite{IPrel},
can be represented by circle bundles over loops~$\ga$ in~$V$; 
see \cite[Section~3.1]{GWrelIP}.
By standard topological considerations,
\BE{cRXVvsH1V_e}\cR_X^V\approx H_1(V;\Z)_X\equiv 
\frac{H_1(V;\Z)}{H_X^V},
\qquad\hbox{where}\quad H_X^V\equiv \big\{A\!\cap\!V\!:\,A\!\in\!H_3(X;\Z)\big\} \,;\EE
see \cite[Corollary~3.2]{GWrelIP}.\\

\noindent
The main claim of \cite[Section~5]{IPrel} is that the above observations can be used 
to lift~\eref{evdfn_e2} over some regular (Galois), possibly disconnected (unramified) covering
\BE{IPcov_e}\pi_{X;\bs}^V\!: \wh{V}_{X;\bs}\lra V_{\bs};\EE
the topology of this cover is specified in \cite[Section~6.1]{GWrelIP}.
Its group of deck transformations~is 
\BE{DeckGr_e}\Deck\big(\pi_{X;\bs}^V\big) 
=\frac{\cR_X^V}{\cR_{X;\bs}'^{\,V}}\!\times\! \cR_{X;\bs}'^{\,V}\EE
for a certain submodule $\cR_{X;\bs}'^{\,V}$ of $\cR_{X}^V$.
For example,
$$\cR_{X;\bs}'^{\,V}=\begin{cases}\{0\},&\hbox{if}~\ell\!=\!0;\\
\gcd(\bs)\cR_X^V,&\hbox{if}~|\pi_0(V)|\!=\!1.
\end{cases}$$
As discussed in \cite[Section~1.1]{GWrelIP}, the topology of the covering~\eref{IPcov_e}
is usually very complicated.
Since
\BE{evfactor_e}\ev_X^V\!=\!\pi_{X;\bs}^V\!\circ\!\wt\ev_X^V\!:\,\ov\fM_{g,k;\bs}^V(X,A)\lra V_{\bs}\EE
for some morphism
\BE{evXVlift_e} \wt\ev_X^V\!:\ov\fM_{g,k;\bs}^V(X,A)\lra \wh{V}_{X;\bs}\,,\EE
the numbers obtained by pulling back elements of $H^*(\wh{V}_{X;\bs};\Q)$ by~\eref{evXVlift_e},
instead of elements of $H^*(V_{\bs};\Q)$ by~\eref{evdfn_e2}, and integrating them and 
other natural classes on  $\ov\fM_{g,k;\bs}^V(X,A)$ against the virtual class of 
$\ov\fM_{g,k;\bs}^V(X,A)$ refine the usual GW-invariants of~$(X,V,\om)$.
We will call these numbers the {\sf{IP-counts for~$(X,V,\om)$}.
These numbers generally depend on the choice of the lift~\eref{evXVlift_e}.\\

\noindent
The construction of the coverings~\eref{IPcov_e} is recalled in Section~\ref{RelGW_subs}.
The lifts~\eref{evXVlift_e} can be chosen systematically in a manner suitable
for use in the symplectic sum context; see Proposition~\ref{RimToriAct_prp}.
In~\cite{GWrelIP}, we deduced vanishing results for the standard GW-invariants 
of $(X,V,\om)$ from the existence of the lifts~\eref{evXVlift_e}.
We use a very basic case of these vanishing results
in Section~\ref{RES_sec} to streamline the proof of \cite[(15.4)]{IPsum},
after correcting its statement;
this formula computes the GW-invariants of the blowup~$\wh\P^2_9$ of~$\P^2$ at 9~points.

\subsection{Symplectic sum formulas}
\label{SympSum_subs0}

\noindent
Let  $(X,\om_X)$ and $(Y,\om_Y)$ be compact symplectic manifolds
with a common compact symplectic divisor $V\!\subset\!X,Y$.
If
\BE{cNVcond_e}e(\cN_XV)=-e(\cN_YV)\in H^2(V;\Z),\EE
then there exists an isomorphism
\BE{cNpair_e} \Phi\!:\cN_XV\otimes\cN_YV\approx V\!\times\!\C\EE
of complex line bundles.
A \sf{symplectic sum} of symplectic manifolds $(X,\om_X)$ and $(Y,\om_Y)$ 
with a common symplectic divisor~$V$ such that~\eref{cNVcond_e} holds
is a symplectic manifold $(Z,\om_Z)\!=\!(X\!\#_V\!Y,\om_{\#})$ obtained from~$X$ and~$Y$ 
by gluing the complements of tubular neighborhoods of~$V$ in~$X$ and~$Y$ 
along their common boundary as directed by~$\Phi$.
In fact, the symplectic sum construction of~\cite{Gf, MW} produces 
a \sf{symplectic fibration} $\pi\!:\cZ\!\lra\!\De$ 
with central fiber \hbox{$\cZ_0\!=\!X\!\cup_V\!Y$}, where
$\De\!\subset\!\C$ is a disk centered at the origin and 
$\cZ$ is a symplectic manifold with symplectic form~$\om_{\cZ}$
such~that 
\begin{enumerate}[label=$\bullet$,leftmargin=*]
\item  $\pi$ is surjective and is a submersion outside of $V\!\subset\!\cZ_0$, 
\item the restriction~$\om_{\la}$ of~$\om_{\cZ}$ to $\cZ_{\la}\equiv\pi^{-1}(\la)$ 
is nondegenerate for every $\la\!\in\!\De^*$,
\item $\om_{\cZ}|_X\!=\!\om_X$, $\om_{\cZ}|_Y\!=\!\om_Y$.
\end{enumerate} 
The symplectic manifolds $(\cZ_{\la},\om_{\la})$ with $\la\!\in\!\De^*$
are then symplectically deformation equivalent to each other and denoted~$(X\!\#_V\!Y,\om_{\#})$.
However, different homotopy classes of the isomorphisms~\eref{cNpair_e} give rise
to generally different topological manifolds; see~\cite{Gf0}.
There is also a retraction $q\!:\cZ\!\lra\!\cZ_0$ such that $q_{\la}\!\equiv\!q|_{\cZ_{\la}}$ 
restricts to a diffeomorphism 
$$\cZ_{\la}-q_{\la}^{-1}(V)\lra \cZ_0-V$$
and to an $S^1$-fiber bundle $q_{\la}^{-1}(V)\!\lra\!V$, whenever $\la\!\in\De^*$.
We denote by $q_{\#}\!:X\!\#_V\!Y\!\lra\!X\!\cup_V\!Y$ a typical \sf{collapsing map}~$q_{\la}$.\\

\noindent
The symplectic sum formulas of \cite{LR,Jun2} for GW-invariants relate 
the absolute GW-invariants of a smooth fiber $\cZ_{\la}\!=\!X\!\#_V\!Y$ 
to the GW-invariants of a singular fiber $\cZ_0\!=\!X\!\cup_V\!Y$ 
and to the relative GW-invariants of the pairs $(X,V)$ and~$(Y,V)$.
The first relation is often called a degeneration formula for GW-invariants
in symplectic topology and an  \sf{invariance property of GW-invariants} in algebraic geometry;
the formula \cite[(5.7)]{LR} and the second formula at the bottom of \cite[p201]{Jun2}
fall under this category. 
The second relation is often called a \sf{decomposition formula for GW-invariants}
in symplectic topology and a degeneration formula for GW-invariants in algebraic geometry;
the three formulas \cite[(5.4),(5.7),(5.8)]{LR} together and the first formula 
at the bottom of \cite[p201]{Jun2} fall under this category. 
In order to reduce confusion, we will call the first relation an \sf{invariance property} 
and the second a \sf{decomposition formula}; see Table~\ref{termin_tbl}.
As indicated below, a decomposition formula for GW-invariants is an immediate consequence
of an invariance property  in the basic symplectic sum settings of \cite{LR,Jun2}.
However, the difference between the two formulas
turns out to be insurmountable in the refined setting of~\cite{IPsum}
and substantial in the (unrefined) multifold degeneration settings of \cite{Brett,GS,AC}.\\

\begin{table}
\begin{center}
\begin{tabular}{||c||c|c||}
\hline\hline
relation for GWs& ST name& AG name\\
\hline
$X\!\#_V\!Y$ vs.~$X\!\cup_V\!\!Y$
&degeneration formula &\sf{invariance property}\\
\hline
$X\!\#_V\!Y$ vs.~$(X,V)$ and~$(Y,V)$
&\sf{decomposition formula} &degeneration formula\\
\hline\hline
\end{tabular}
\end{center}
\caption{ST and AG terminology describing the two types of symplectic sum formulas
for GW-invariants; our terminology is in \sf{sans-serif}.}
\label{termin_tbl}
\end{table}

\noindent
With $V\!\subset\!X,Y$ as above, let 
\begin{alignat}{1}
\label{cRXYVprop_e}
\cR_{X,Y}^V&=\ker\big\{q_{\#*}\!:H_2(X\!\#_V\!Y;\Z)\lra H_2(X\!\cup_V\!\!Y;\Z)\big\},\\
\notag
H_2(X;\Z)\!\times_V\!H_2(Y;\Z)&=
\big\{(A_X,A_Y)\!\in\!H_2(X;\Z)\!\times\!H_2(Y;\Z)\!:~A_X\!\cdot_X\!V=A_Y\!\cdot_Y\!V\big\}.
\end{alignat}
As recalled in \cite[Section~2.2]{GWrelIP}, there is a natural homomorphism
\BE{homsumdfn_e}
H_2(X;\Z)\!\times_V\!H_2(Y;\Z)\lra H_2(X\!\#_V\!Y;\Z)/\cR_{X,Y}^V, \quad
(A_X,A_Y)\lra A_X\!\#_V\!A_Y.\EE
It is obtained by representing $A_X$ and $A_Y$ by cycles in~$X$ and~$Y$ with the same contacts
with~$V$ and smoothing out the nodes of the resulting cycle into $X\!\cup_V\!\!Y\!\subset\!\cZ$.  
Let  \hbox{$\eta\!\in\!H_2(X\!\#_V\!Y;\Z)/\cR_{X,Y}^V$}
be an $\cR_{X,Y}^V$-coset of $H_2(X\!\#_V\!Y;\Z)$ and $g\!\in\!\Z^{\ge0}$.
According to the invariance formulas of \cite{LR,Jun2},
the sum of the genus~$g$ GW-invariants of $X\!\#_V\!Y$ with degrees $A\!\in\!\eta$
is the same as the sum of the genus~$g$ GW-invariants of $X\!\cup_V\!\!Y$
of degrees
$$(A_X,A_Y)\in H_2(X;\Z)\!\times_V\!H_2(Y;\Z)
\qquad\hbox{s.t.}\quad A_X\!\#_V\!A_Y=\eta$$
with the same cohomological insertions.
The allowed cohomological insertions consist of intrinsic classes on moduli spaces of
stable maps, such as $\psi$- and $\la$-classes, and pullbacks of cohomology classes on~$\cZ$
by the evaluation morphism~\eref{evdfn_e1};
we will call such insertions \sf{$\Phi$-admissible inputs}.
Two characterizations of cohomology classes on a smooth fiber $\cZ_{\la}\!=\!X\!\#_V\!Y$ that are 
restrictions of cohomology classes on~$\cZ$ are provided in \cite[Section~4.4]{GWrelIP}.
By Gromov's Compactness Theorem for $J$-holomorphic curves, 
both sums have only finitely many possibly nonzero terms for each fixed~$g$ and~$\eta$
(independently of the cohomological insertions).\\

\noindent
The GW-invariants of $X\!\cup_V\!\!Y$ count curves that lie in~$X$ and~$Y$ and
meet the divisor~$V$ at the same points of~$V$ and with the same order of contact;
see Figure~\ref{limitcurve_fig}.
The contacts of such a curve with~$V$ can be described by a tuple $\bs\!\in\!\Z_+^{\ell}$,
where $\ell\!\in\!\Z^{\ge0}$ is the number of nodes on~$V$. 
Its contribution to the corresponding GW-invariant of $X\!\cup_V\!\!Y$ is the product 
of the orders of contacts,
$$\lr\bs\equiv s_1\cdot\ldots\cdot s_{\ell}\,.$$
The combinatorial structure of such a curve is described by a bipartite graph~$\Ga$.
The vertices of~$\Ga$ specify the genus and degree of each maximal connected curve mapped
into~$X$ and~$Y$; its edges determine the meeting pattern between the components
and the degrees of contacts with~$V$.
In the terminology of~\cite{Jun2}, $\Ga$ corresponds to a triple $(\Ga_1,\Ga_2,I)$,
with $\Ga_1$ and $\Ga_2$ specifying the components mapped into~$X$ and~$Y$, respectively, and
$I$ determining the distribution of the marked points between the components.
Each graph~$\Ga$ and an ordering of $\ell\!=\!\ell(\Ga)$ relative contact points  
determine moduli spaces $\ov\fM_{\Ga}^V(X)$ and $\ov\fM_{\Ga}^V(Y)$
of relative stable maps into~$(X,V)$ and~$(Y,V)$ from disconnected domains
with the same relative contact vector $\bs\!=\!\bs(\Ga)$.
These moduli spaces are quotients of products of moduli spaces of stable maps 
into~$(X,V)$ and~$(Y,V)$ from connected domains by $\Aut(\Ga_1,I|_{\Ga_1})$ 
and $\Aut(\Ga_2,I|_{\Ga_2})$, respectively.\\

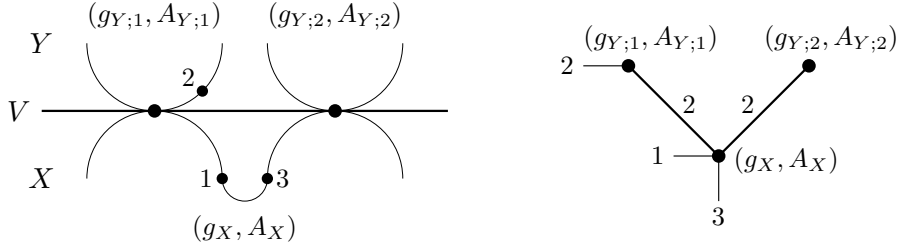
\begin{figure}
\begin{pspicture}(-1,-2.5)(20,1)
\psset{unit=.3cm}
\psline[linewidth=.1](5,-2)(23,-2)
\rput(4,-2){$V$}\rput(5,-5){$X$}\rput(5,1){$Y$}
\psarc[linewidth=.04](10,-5){3}{0}{180}\pscircle*(10,-2){.3}
\psarc[linewidth=.04](14,-5){1}{180}{0}
\psarc[linewidth=.04](18,-5){3}{0}{180}\pscircle*(18,-2){.3}
\psarc[linewidth=.04](10,1){3}{180}{0}\psarc[linewidth=.04](18,1){3}{180}{0}
\pscircle*(13,-5){.25}\pscircle*(15,-5){.25}\pscircle*(12.12,-1.12){.25}
\rput(12.3,-5){\sm{1}}\rput(11.5,-.5){\sm{2}}\rput(15.7,-5){\sm{3}}
\rput(14,-7.2){\sm{$(g_X,A_X)$}}\rput(10,2.2){\sm{$(g_{Y;1},A_{Y;1})$}}
\rput(18,2.2){\sm{$(g_{Y;2},A_{Y;2})$}}
\pscircle*(35,-4){.3}\pscircle*(31,0){.3}\pscircle*(39,0){.3}
\psline[linewidth=.1](35,-4)(31,0)\psline[linewidth=.1](35,-4)(39,0)
\psline[linewidth=.05](35,-4)(33,-4)\psline[linewidth=.05](35,-4)(35,-6)
\psline[linewidth=.05](31,0)(29,0)
\rput(32.3,-4){\sm{1}}\rput(28.3,0){\sm{2}}\rput(35,-6.7){\sm{3}}
\rput(38,-4.2){\sm{$(g_X,A_X)$}}\rput(32,1.2){\sm{$(g_{Y;1},A_{Y;1})$}}
\rput(40,1.2){\sm{$(g_{Y;2},A_{Y;2})$}}
\rput(33.7,-1.7){\sm{2}}\rput(36.3,-1.7){\sm{2}}
\end{pspicture}
\caption{A curve in $X\!\cup_V\!Y$ possibly contributing to the genus $g_X\!+\!g_{Y;1}\!+\!g_{Y;2}$
degree $(A_X,A_{Y;1}\!+\!A_{Y;2})$ GW-invariant of $X\!\cup_V\!\!Y$ and the corresponding bipartite
graph.}
\label{limitcurve_fig}
\end{figure}

\noindent
The morphisms~\eref{evdfn_e1} and~\eref{evdfn_e2} induce morphisms on  
$\ov\fM_{\Ga}^V(X)$ and $\ov\fM_{\Ga}^V(Y)$.
Let 
$$\ov\fM_{\Ga}^V(X\!\cup_V\!\!Y)\subset \ov\fM_{\Ga}^V(X)\times\ov\fM_{\Ga}^V(Y)$$
denote the preimage of the diagonal $\De_{\bs}^V\!\subset\!V_{\bs}^2$ under
the product evaluation morphism
\BE{ProdEval_e}\ev_X^V\!\times\!\ev_Y^V\!:\, 
\ov\fM_{\Ga}^V(X)\!\times\!\ov\fM_{\Ga}^V(Y)\lra V_{\bs}^2\,.\EE
The moduli space of curves into $X\!\cup_V\!\!Y$ of type $\Ga$ is the quotient of 
$\ov\fM_{\Ga}^V(X\!\cup_V\!\!Y)$ by the action of the symmetric group~$\bS_{\ell}$
permuting the relative marked points.
The number of such maps is computed by pulling back the Poincare dual 
$\PD_{\bs}^V\De$ of $\De_{\bs}^V$ by~\eref{ProdEval_e} 
and integrating it over the product of the moduli spaces
along with the original cohomology insertions.  
By the Kunneth formula for cohomology \cite[Theorem~60.6]{Mu2},
\BE{KunnDecomp_e} \PD_{\bs}^V\De=\sum_{i=1}^N\ka_{X;i}\!\otimes\!\ka_{Y;i}\in H^{(n-1)\ell}(V_{\bs}^2;\Q)\EE
for some $\ka_{X;i},\ka_{Y;i}\!\in\!H^*(V_{\bs};\Q)$.
Thus, the contribution to the GW-invariant of $X\!\cup_V\!\!Y$ from maps of type~$\Ga$
is the sum of $N$ products of relative invariants of $(X,V)$ and~$(Y,V)$ 
with relative insertions~$\ka_{X;i}$ and~$\ka_{Y;i}$, times $\lr{\bs(\Ga)}/\ell(\Ga)!$.
So, in this case, a decomposition formula for GW-invariants is a direct consequence
of an invariance property for GW-invariants.\\

\noindent
An obvious deficiency of the decomposition formulas of \cite{LR,Jun2} is that 
they express sums of GW-invariants of $X\!\#_V\!Y$ over degrees differing 
by elements of~$\cR_{X,Y}^V$ in terms of relative GW-invariants of $(X,V)$ and $(Y,V)$;
it would of course be preferable to express individual GW-invariants of $X\!\#_V\!Y$
in terms of relative GW-invariants of $(X,V)$ and~$(Y,V)$.
The rim tori refinement of relative invariants is introduced in~\cite{IPrel}
with the aim of resolving this deficiency in~\cite{IPsum}.\\

\noindent
For $\ell\!\in\!\Z^{\ge0}$ and $\bs\!\in\!\Z_+^{\ell}$, let
\BE{cHXYV_e0}\wh{V}_{X,Y;\bs} =  \wh{V}_{X;\bs}\!\times_{V_{\bs}}\!\wh{V}_{Y;\bs}
\equiv\big\{\pi_{X;\bs}^V\!\times\!\pi_{Y;\bs}^V\big\}^{-1}
\big(\De_{\bs}^V\big).\EE
The idea of~\cite{IPsum} is that there is a continuous map
\BE{IPdegmap_e} g_{A_X,A_Y}\!:\wh{V}_{X,Y;\bs}\lra A_X\!\#_V\!A_Y\subset H_2(X\!\#_V\!Y;\Z)\EE
such that its composition with the restriction of
\BE{LfProdEv_e}\wt\ev_X^V\!\times\!\wt\ev_Y^V\!: 
\ov\fM_{\Ga}^V(X)\!\times\!\ov\fM_{\Ga}^V(Y)   \lra  \wh{V}_{X;\bs}\!\times\! \wh{V}_{Y;\bs}\EE
to $\ov\fM_{\Ga}^V(X\!\cup_V\!Y)$
is the homology degree of the glued map into~$X\!\#_V\!Y$;
see Proposition~\ref{RimToriAct_prp2} and Figure~\ref{RimToriAct_fig}.
Thus, the space of maps into $X\!\cup_V\!\!Y$ contributing to
the GW-invariant of $X\!\#_V\!Y$ of a degree $A\!\in\!A_X\!\#_V\!A_Y$
is the preimage of 
\BE{cHXYVC_e}\wh{V}_{X,Y;\bs}^A\equiv g_{A_X,A_Y}^{-1}(A)\EE
under the morphism~\eref{LfProdEv_e}.\\

\noindent
Each $\wh{V}_{X,Y;\bs}^A$ determines an intersection homomorphism and thus a class 
\BE{PDcXY_e} \PD_{X,Y;\bs}^{V,A}\De\in 
H^*\big(\wh{V}_{X;\bs}\!\times\!\wh{V}_{Y;\bs};\Q\big),\EE 
as suggested by \cite[Definition~10.2]{IPsum}.
The contribution to the IP-count of $X\!\cup_V\!\!Y$ from maps of type~$\Ga$
is computed by pulling back the cohomology class $\PD_{X,Y;\bs}^{V,A}\De$
by the morphism~\eref{LfProdEv_e}.
Thus, the approach of~\cite{IPsum} expresses GW-invariants of $X\!\#_V\!Y$
of {\it each} degree $A\!\in\!H_2(X\!\#_V\!Y;\Z)$ in terms of IP-counts of $X\!\cup_V\!\!Y$,
i.e.~provides a refined  {\it invariance property} for GW-invariants.
It can be summarized as follows.

\begin{thm*}[Refined Invariance Property for GW-Invariants]
Let $(X,\om_X)$ and $(Y,\om_Y)$ be compact symplectic manifolds
with a common compact symplectic divisor $V\!\subset\!X,Y$
and $\Phi$ be an isomorphism of complex line bundles as in~\eref{cNpair_e}.
For all $g\!\in\!\Z^{\ge0}$, $A\!\in\!H_2(X\!\#_V\!Y;\Z)$, and 
$\Phi$-admissible insertions~$\ka$,
\BE{GWsumIP_e} \GW_{g,A}^{X\#_VY}(\ka)=
\sum_{\begin{subarray}{c} (A_X,A_Y)\in H_2(X;\Z)\times_V H_2(Y;\Z)\\ 
A\in A_X\#_V A_Y\end{subarray}}
\sum_{\begin{subarray}{c}\bs\in\Z_+^{\ell}\\ \ell\ge0 \end{subarray}}
\wt\GW_{g,(A_X,A_Y);\bs}^{X\cup_VY}\big(\ka;\PD_{X,Y;\bs}^{V,A}\De\big),\EE
where $\wt\GW$ denotes an IP-count for $X\!\cup_V\!Y$.
\end{thm*}

\noindent
If the $\Q$-homology of either $\wh{V}_{X;\bs}$ or $\wh{V}_{Y;\bs}$ is finitely generated,
then
\BE{KunnDecomp_e2}
\PD_{X,Y;\bs}^{V,A}\De=\sum_{i=1}^N\wt\ka_{X;i}\!\otimes\!\wt\ka_{Y;i}\in 
H^{(n-1)\ell}(\wh{V}_{X;\bs}\!\times\!\wh{V}_{Y;\bs};\Q)\EE
for some $\wt\ka_{X;i}\!\in\!H^*(\wh{V}_{X;\bs};\Q)$ and $\wt\ka_{Y;i}\!\in\!H^*(\wh{V}_{Y;\bs};\Q)$.
This is also the case if the submodule~$\cR_{X,Y}^V$ of $H_1(X\!\#_V\!Y;\Z)$ is finite;
see Corollary~\ref{Kunn_crl}.
In such cases, the approach of~\cite{IPsum} provides a refined 
{\it decomposition formula} for GW-invariants of $X\!\#_V\!Y$ in terms of IP-counts
for $(X,V)$ and~$(Y,V)$.
However, in general  the homologies of $\wh{V}_{X;\bs}$ and $\wh{V}_{Y;\bs}$ are not finitely generated 
and a Kunneth decomposition~\eref{KunnDecomp_e2} need not exist; see Example~\ref{noKunn_eg}.
In these cases, the approach of~\cite{IPsum} does not provide a decomposition formula 
for GW-invariants of $X\!\#_V\!Y$ in terms of any kind of invariants of~$(X,V)$ and~$(Y,V)$.

\subsection{Vanishing applications}
\label{VanishAppl_subs}

\noindent
Even in cases when a Kunneth decomposition~\eref{KunnDecomp_e2} exists
(and $\cR_{X;Y}^V\!\neq\!\{0\}$, $|\bs|\!\neq\!0$),
the use of the refined decomposition formula of~\cite{IPsum} for quantitative applications
does not appear practical outside of rare cases, in part because of the dependence
of the IP-counts for $(X,V)$ and $(Y,V)$ on the lifts~\eref{evXVlift_e}. 
However, we are able to extract some qualitative applications from the approach 
of \cite{IPrel,IPsum}.\\

\noindent
Let $(X,\om)$ be a symplectic manifold and $V\!\subset\!X$ be a common symplectic divisor
with connected components~$V_1,\ldots,V_N$.
Denote~by 
$$\wh{V}_X\lra V_1\!\times\!\ldots\!\times\!V_N$$
the covering projection corresponding to the preimage of $H_X^V$ under the natural homomorphism
$$\pi_1(V_1\!\times\!\ldots\!\times\!V_N)\lra H_1(V_1\!\times\!\ldots\!\times\!V_N;\Z)
=\bigoplus_{r=1}^N H_1(V_r;\Z)=H_1(V;\Z),$$
i.e.~$\wh{V}_{X;(1)^N}$ in the notation of Section~\ref{RelGW_subs}.
We will call $V\!\subset\!X$ \sf{virtually connected} if  the cokernel
of the composition homomorphism
\BE{VirConn_e}H_1(V_r;\Z)\lra H_1(V;\Z)\lra H_1(V;\Z)_X\EE
is finite for every component $V_r\!\subset\!V$.
For example, this is the case if $V$ is connected or 
$X\!=\!\P^1\!\times\!F$ and $V\!=\!\{0,\i\}\!\times\!F$ for some connected symplectic manifold~$F$;
the homomorphism~\eref{VirConn_e} is surjective in both cases.
We will call a class $A\!\in\!H_2(X\!-\!V;\Z)$ \sf{$\om$-effective} 
if for every $\om$-tame almost complex structure~$J$ on~$(X,V)$ there
exists a $J$-holomorphic map $u\!:\Si\!\lra\!X\!-\!V$ from a compact Riemann surface
such~that
$$u_*[\Si]=A\in H_2(X\!-\!V;\Z)\,.$$
We denote by  $\Eff_{\om}(X,V)\!\subset\!H_2(X\!-\!V;\Z)$ the subset of effective classes.

\begin{thm}\label{AbsGW_thm}
Let $(X,\om_X)$ and $(Y,\om_Y)$ be compact symplectic manifolds
with a common compact symplectic divisor $V\!\subset\!X,Y$
and $\Phi$ be an isomorphism of complex line bundles as in~\eref{cNpair_e}.
Suppose $\cR_{X,Y}^V$ is infinite,  $V\!\subset\!X$ is virtually connected,
and $H_*(\wh{V}_X;\Q)$ is finitely generated.
If 
\BE{AbsGW_e2}A\in H_2(X\!\#_V\!Y;\Z) - \io_{X-V*}^{X\#_VY}\big(\Eff_{\om_X}(X,V)\big)
- \io_{Y-V*}^{X\#_VY}\big(\Eff_{\om_Y}(Y,V)\big),\EE
then all degree~$A$ GW-invariants of $X\!\#_V\!Y$ with $\Phi$-admissible inputs
vanish.
\end{thm}

\noindent
By \cite[Corollary~4.2(1)]{GWrelIP}, $\cR_{X,Y}^V$ is infinite if
either of the homomorphisms
$$\io_{V*}^X\!: H_2(V;\Z)\lra H_2(X;\Z) \qquad\hbox{or}\qquad 
\io_{V*}^Y\!: H_2(V;\Z)\lra H_2(Y;\Z)$$
is injective and 
$$ H_1(V;\Q)\neq \big\{B\!\cap_X\!V\!:\,B\!\in\!H_3(X;\Q)\big\}
+\big\{B\!\cap_Y\!V\!:\,B\!\in\!H_3(Y;\Q)\big\}\,.$$
In light of \cite[Assertion~6]{Mi68} in the case $H_1(V;\Z)$ is of rank~1 and 
its extension~\cite{Mi14}, the finite generation condition can be satisfied 
only if~$\chi(V)\!=\!0$ (assuming $\cR_{X,Y}^V$ is infinite).
All covers of the fibrations in \cite[Theorem~1.1]{GWrelIP} have finitely generated homology.
A characterization of abelian covers with finitely generated homology for an arbitrary
compact base is provided
by \cite[Theorem~1]{DF}, though it appears difficult to use in practice.\\

\noindent
In a typical situation, most classes $A\!\in\!H_2(X\!\#_V\!Y;\Z)$ with potentially 
nonzero GW-invariants satisfy~\eref{AbsGW_e2}. 
For example, if $X$ and $Y$ are Kahler, $V\!\subset\!X$ is ample, and $V\!\subset\!Y$ is
anti-ample, then $X\!-\!V$ and $Y\!-\!V$ contain no curves and the restriction~\eref{AbsGW_e2}
is no longer necessary.
The same is the case if $\PD_XV$ and $\PD_YV$ are nonzero multiples of 
the cohomology classes represented by~$\om_X$ and~$\om_Y$, respectively.
Remark~\ref{AbsGW_rmk} describes a generalization of Theorem~\ref{AbsGW_thm}.\\

\noindent
If $X\!\lra\!\Si_1$ and $Y\!\lra\!\Si_2$ are possibly singular fibrations over curves
with the same smooth fiber~$F$ and $V\!\subset\!X,Y$ is a union of finitely many fibers, 
all elements~$A$ in the image of $H_2(F;\Z)$ in $H_2(X\!\#_V\!Y;\Z)$ fail the condition~\eref{AbsGW_e2}.
For example, $\T^{2n}$ with $n\!>\!1$ is the symplectic sum of 
$\P^1\!\times\!\T^{2n-2}$ with itself along $V\!=\!\{0,\i\}\!\times\!\T^{2n-2}$
with respect to the canonical isomorphism~\eref{cNpair_e}.
By Theorem~\ref{AbsGW_thm}, the GW-invariants of~$\T^{2n}$ in classes~$A$ not contained 
in a fiber vanish.
By changing the projection, we recover the vanishing of all GW-invariants of~$\T^{2n}$ with $n\!>\!1$,
except in degree~0.\\

\noindent
The $K3$ surface~$\bK_3$ is the symplectic sum of the blowup~$\wh\P_9^2$ of~$\P^2$ at 9 points
with itself along a smooth fiber $V\!=\!F$ of the fibration $\wh\P_9^2\!\lra\!\P^1$
with respect to the canonical isomorphism~\eref{cNpair_e}.
In this case, Theorem~\ref{AbsGW_thm} recovers the vanishing of the GW-invariants 
of~$\bK_3$ except in degrees that are multiples of a fiber of
the fibration $\bK_3\!\lra\!\P^1$ (all GW-invariants of~$\bK_3$ are known to vanish).\\

\noindent
The symplectic sum of $\P^1\!\times\!F$ with itself along $V\!=\!\{0,\i\}\!\times\!F$ 
 with respect to the canonical isomorphism~\eref{cNpair_e} is $\T^2\!\times\!F$.
Applying Theorem~\ref{AbsGW_thm} in this case to the genus~1 GW-invariants 
in the section class $A\!=\![\T^2\!\times\!\pt]$, we obtain the following statement
about the maximal abelian cover $\wh{F}\!\lra\!F$,
i.e.~the covering projection corresponding to the commutator subgroup of~$\pi_1(F)$.

\begin{crl}\label{AbsGW_crl}
Let $(F,\om)$ be a compact connected symplectic manifold.
If $H_1(F;\Q)\!\neq\!\{0\}$ and $\chi(F)\!\neq\!0$, then 
$H_*(\wh{F};\Q)$ is not finitely generated over~$\Q$.
\end{crl}

\noindent
By \cite[Assertion~6]{Mi68} in the case $H_1(F;\Z)$ is of rank~1 and its extension~\cite{Mi14}, 
the conclusion of this corollary holds for all finite simplicial complexes~$F$ 
(not just compact symplectic manifolds) and for all infinite abelian covers 
(not just the maximal one).
Thus, its conclusion is not surprising.
What perhaps is surprising is that this purely topological property is detected by
the refined invariance property for GW-invariants arising from~\cite{IPsum}.

\subsection{GW-invariants and the flux group}
\label{FluxAppl_subs}

\noindent
We next describe qualitative applications of the refined decomposition formula 
of~\cite{IPsum} that equate GW-invariants of the symplectic sum in classes
differing by some rim tori.
In certain cases, these observations suffice to express individual GW-invariants
of the symplectic sum in terms of the usual relative GW-invariants of the two pieces.\\

\noindent
If $V$ is any topological space, a loop of homeomorphisms 
$$\Psi_t\!: V\lra V, \qquad t\in[0,1],~~\Psi_0=\Psi_1,$$
and a point $x\!\in\!V$ determines a loop $t\!\lra\!\Psi_t(x)$ in~$V$
and thus an element of $H_1(V;\Z)$.
The latter is independent of the choice of $x\!\in\!V$.
We denote the set of all elements of $H_1(V;\Z)$ obtained in this way by~$\Flux(V)$. 
It is a subgroup of $H_1(V;\Z)$, usually called \sf{the flux subgroup} 
(or \sf{group}).
If in addition $V\!\subset\!X$ is a compact oriented submanifold of a compact oriented manifold
and $H_1(V;\Z)_X$ is as in~\eref{cRXVvsH1V_e}, let
$$\Flux(V)_X\subset H_1(V;\Z)_X$$
denote the image of $\Flux(V)$ under the quotient projection.\\

\noindent
Let $V\!\subset\!X$ be a symplectic divisor with topological
components $V_1,\ldots,V_N$.
For each \hbox{$r\!=\!1,\ldots,N$}, denote by $\ff_{X,V_r}\!\in\!H_1(S_XV_r;\Z)$
the homology class of a fiber of the circle bundle in $\cN_XV_r\!\lra\!V_r$.
If $Y$ is another symplectic manifold containing~$V$ and  $\Phi$ is an isomorphism of complex line bundles 
as in~\eref{cNpair_e}, let
$$\de_{\Phi}\!: H_2(X\!\#_V\!Y;\Z)\lra H_1(S_XV;\Z)$$
be the connecting homomorphism of the Mayer-Vietoris sequence for 
$X\!\#_V\!Y\!=\!(X\!-\!V)\!\cup\!(Y\!-\!V)$,
i.e.~as in the first exact sequence 
in the proof of \cite[Lemma~4.1]{GWrelIP} with $m\!=\!\fc\!=\!2$.
For each $A\!\in\!A_X\!\#_V\!A_Y$,
\BE{AVrdfn_e}\de_{\Phi}(A)=\sum_{r=1}^N|A|_{V_r}\ff_{X,V_r}\in H_1(S_XV;\Z)
, \qquad\hbox{where}\quad
|A|_{V_r}\!\equiv A_X\cdot_X\!V_r\in \Z\,.\EE
If $\cN_XV_r\!\approx\!V_r\!\times\!\C$,  $\ff_{X,V_r}\!\neq\!0$ 
and so the number $|A|_{V_r}$ depends only on $A_X\!\#_V\!A_Y\!\subset\!H_2(X\!\#_V\!Y;\Z)$.

\begin{prp}\label{EqGWs_prp1}
Let $(X,\om_X)$ and $(Y,\om_Y)$ be compact symplectic manifolds,
$V\!\subset\!X,Y$ be a common  compact connected symplectic divisor
such that $\cN_XV\!\approx\!V\!\times\!\C$ and $\Flux(V)_X\!=\!H_1(V;\Z)_X$, 
and $\Phi$ be an isomorphism of complex line bundles as in~\eref{cNpair_e}.
If 
\BE{EqGWsprp1_e}A_1,A_2\!\in\!H_2(X\!\#_V\!Y;\Z)  \qquad\hbox{and}\qquad
A_1\!-\!A_2\in  |A_1|_V\cR_{X,Y}^V,\EE
then  the GW-invariants of $X\!\#_V\!Y$ of degrees~$A_1$ and~$A_2$ with 
$\Phi$-admissible inputs are the same.
\end{prp}

\begin{prp}\label{EqGWs_prp2}
Let $(X,\om_X)$ and $(Y,\om_Y)$ be compact symplectic manifolds,
$V\!\subset\!X,Y$ be a common compact symplectic divisor with topological
components $V_1,\ldots,V_N$ such that $\cN_XV\!\approx\!V\!\times\!\C$ and $\Flux(V)_X\!=\!H_1(V;\Z)_X$, 
and $\Phi$ be an isomorphism of complex line bundles as in~\eref{cNpair_e}.
If 
\BE{EqGWsprp2_e} A_1,A_2 \in H_2(X\!\#_V\!Y;\Z), \qquad A_1\!-\!A_2\in\cR_{X,Y}^V,\EE
and $|A_1|_{V_r}$ is prime relative~to $|\cR_{X,Y}^V|$ for every~$r$,
then the GW-invariants of $X\!\#_V\!Y$ of degrees~$A_1$ and~$A_2$ with $\Phi$-admissible 
inputs are the same.
\end{prp}

\noindent
If the vanishing cycles module $\cR_{X,Y}^V$ is infinite, 
we call the numbers $|A_1|_{V_r}$ and $|\cR_{X,Y}^V|$ \sf{relatively prime} if $|A_1|_{V_r}\!=\!\pm1$.
Propositions~\ref{EqGWs_prp1} and~\ref{EqGWs_prp2} are special cases of Theorem~\ref{EqGWs_thm}.
The latter also leads to a vanishing result for the GW-invariants of $X\!\#_V\!Y$
in the spirit of Theorem~\ref{AbsGW_thm}, but with different assumptions;
see Corollary~\ref{EqGWs_crl}.\\

\noindent
The flux condition in the above propositions is automatically satisfied if $V$ is a union of
the tori~$\T^{2n-2}$.
The relevant covers $\wh{V}_{X;\bs}$ and $\wh{V}_{Y;\bs}$
are then of the form $\R^m\!\times\!\T^{m'}$;
the groups of deck transformations act trivially on~them
and thus on the Poincare duals~\eref{PDcXY_e} of the diagonals components~\eref{cHXYVC_e}.
This approach provides a direct justification of Propositions~\ref{EqGWs_prp1}
and~\ref{EqGWs_prp2} from the refined invariance property for GW-invariance arising from~\cite{IPsum}
when  $V$ is a union of~tori.

\subsection{Miscellaneous considerations}
\label{misc_subs}

\noindent
The algebraic approach of~\cite{Jun2} considers only Kahler fibrations 
$\pi\!:\cZ\!\lra\!\De$  which come with an ample line bundle $\cL\!\lra\!\cZ$.
Since every element of $\cR_{X,Y}^V$ in~$\cZ_{\la}$ can be represented by a totally
real submanifold, its homology intersection with every complex hyperplane in~$\cZ_{\la}$
is~zero.
Thus, by the Lefschetz Theorem on $(1,1)$-classes \cite[p163]{GH},
an  element of $\cR_{X,Y}^V$ in~$\cZ_{\la}$ determines a class
in $H^{n-2,n}(\cZ_{\la})\!\oplus\!H^{n,n-2}(\cZ_{\la})$,
where $n$ is the complex dimension of~$\cZ_{\la}$, $X$, and~$Y$.
In particular, curve classes in $H_2(\cZ_{\la};\Z)$  
differing by an element of~$\cR_{X,Y}^V$ differ by a torsion class.
This observation and Theorem~\ref{AbsGW_thm} suggest the following conjecture
about non-Kahler symplectic sums.

\begin{cnj}\label{RimTori_cnj}
Let $(X,\om_X)$ and $(Y,\om_Y)$ be compact symplectic manifolds
with a common compact symplectic divisor $V\!\subset\!X,Y$
and $\Phi$ be an isomorphism of complex line bundles as in~\eref{cNpair_e}.
If $A_1,A_2\!\in\!H_2(X\!\#_V\!Y;\Z)$, $A_1\!-\!A_2\!\in\!\cR_{X,Y}^V$,
and some GW-invariants of $X\!\#_V\!Y$ of degrees~$A_1$ and~$A_2$ are nonzero,
then $A_1\!-\!A_2$ is a torsion class.
\end{cnj}

\begin{rmk}\label{nonKahler_rmk}
The reasoning above Corollary~\ref{RimTori_cnj} does not apply outside of the Kahler setting.
For example, let $X\!=\!\P^1\!\times\!\T^2$, $V\!=\!\{0,\i\}\!\times\!\T^2$,
$$f_1,f_2\!:\T^2\lra X\!-\!V, \qquad 
f_1(\ne^{\fI\th_1},\ne^{\fI\th_2})=(2,\ne^{\fI\th_1},\ne^{\fI\th_2}), \quad
f_2(\ne^{\fI\th_1},\ne^{\fI\th_2})=(\ne^{\fI\th_1},\ne^{\fI\th_1},\ne^{\fI\th_2}).$$
Since the images of the embeddings~$f_1$ and~$f_2$ are disjoint symplectic submanifolds of~$X$,
we can choose an almost complex structure~$J_X$ on~$X$
which is standard around $V$ and makes these images $J_X$-holomorphic.
The two maps~$f_1$ and~$f_2$ differ by a rim torus.
Since both maps miss~$V$, they induce $J$-holomorphic maps into 
$Z\!=\!X\!\#_V\!X$, which differ by a non-trivial element of $\cR_{X,X}^V\!\approx\!\Z^2$.
\end{rmk}

\subsection{Outline of the paper}
\label{outline_subs}

\noindent
We review the notation for abelian covers from~\cite{GWrelIP} in Section~\ref{AbCov_subs}
and the definition of the coverings~\eref{IPcov_e} in Section~\ref{RelGW_subs}.
In Section~\ref{GlDeg_subs}, we define the map~\eref{IPdegmap_e} as
a special case of the map~\eref{wtPsi_e} for arbitrary abelian covers constructed 
in Section~\ref{DiagonSplit_subs}.
In Section~\ref{XYVdiagprop_subs},
the map~\eref{wtPsi_e} is used to construct and study 
the diagonal components~\eref{whVdiag_e} and their cohomology classes~\eref{GenPD_e},
which specialize to~\eref{cHXYVC_e} an~\eref{PDcXY_e}, respectively, 
in the symplectic sum context.
Theorem~\ref{AbsGW_thm} and Corollary~\ref{AbsGW_crl} are established in Section~\ref{SympSumPf_subs}.
A rim tori refinement of~\eref{IPdegmap_e}  is defined in Section~\ref{RTconv_subs}
as a special case of the convolution product on abelian covers defined in Section~\ref{AbConv_subs};
see~\eref{XiXYVdfn_e} and~\eref{Xidfn_e0}.
Section~\ref{RES_sec} streamlines the computation of some GW-invariants of~$\wh\P^2_9$ in~\cite{IPsum} 
by making use of a vanishing result for relative GW-invariants, which is an immediate consequence
of the existence of IP-counts, and corrects some statements in~\cite{IPsum} concerning 
these counts.\\

\noindent
The purpose of this paper is to investigate the topological aspects of the rim tori
refinement to the standard symplectic sum formula.
We pre-suppose that the latter has been established 
and describe the necessary steps to implement the suggestion of \cite{IPsum} 
as an enhancement on an existing analytic proof.
We deduce some qualitative applications arising from this refinement and 
discuss its usability for quantitative purposes.
A significant number of examples are included in Section~\ref{SympSum_sec} 
for illustrative purposes;
some of them are also used in Section~\ref{RES_sec}.\\

\noindent
The authors would like to thank  E.~Ionel, D.~McDuff, M.~McLean, J.~Milnor,
and J.~Starr for enlightening discussions.

\section{Review of preliminaries}
\label{review_sec}

\noindent
The coverings~\eref{IPcov_e} are special cases of the abelian covers described 
in Section~\ref{AbCov_subs}.
The former are constructed in \cite[Section~6.1]{GWrelIP} and 
reviewed in Section~\ref{RelGW_subs}.
We also recall a crucial statement concerning choices of the lifts~\eref{evXVlift_e} 
established in~\cite{GWrelIP}; see Proposition~\ref{RimToriAct_prp}.

\subsection{Abelian covers of topological spaces}
\label{AbCov_subs}

\noindent
Let $\Z_{\pm}\!\subset\!\Z$ denote the nonzero integers.
For a tuple $\bs\!=\!(s_1,\ldots,s_{\ell})\in\Z_{\pm}^{\,\ell}$ with $\ell\!\in\!\Z^{\ge0}$,
we denote by $\gcd(\bs)$ the greatest common divisor of~$s_1,\ldots,s_{\ell}$;
if $\ell\!=\!0$, we set $\gcd(\bs)\!=\!0$.\\

\noindent
Let $V$ be a topological space. 
For any submodule $H\!\subset\!H_1(V;\Z)$, let
\BE{cRH} q_H\!: H_1(V;\Z)\lra \cR_H\equiv \frac{H_1(V;\Z)}{H}\EE
be the projection to the corresponding quotient module.
If $V_1,\ldots,V_N$ are the topological components of~$V$,
\hbox{$\ell_1,\ldots,\ell_N\!\in\!\Z^{\ge0}$}, and
$\bs_1\!\in\!\Z_{\pm}^{\,\ell_1},\ldots,\bs_N\!\in\!\Z_{\pm}^{\,\ell_N}$, then the topological space
$$V_{\bs_1\ldots\bs_N}\equiv V_1^{\ell_1}\!\times\!\ldots\!\times\!V_N^{\ell_N}$$
is connected.\\

\noindent
With $V$ and $\bs_1,\ldots,\bs_N$ as above, define
\begin{gather}\label{PhiVbs_e}
\Phi_{V;\bs_1\ldots\bs_N}\!: 
H_1\big(V_{\bs_1\ldots\bs_N};\Z\big)=\bigoplus_{r=1}^N\! H_1(V_r;\Z)^{\oplus\ell_r} 
\lra H_1(V;\Z),\\ 
\notag
\Phi_{V;\bs_1\ldots\bs_N}
\big((\ga_{r;i})_{i\le\ell_r,r\le N}\big)= \sum_{r=1}^N\sum_{i=1}^{\ell_r}s_{r;i}\ga_{r;i}\,.
\end{gather}
For any submodule $H\!\subset\!H_1(X;\Z)$, let
\begin{gather}
\label{Hbsdfn_e}  H_{\bs_1\ldots\bs_N}
=\Phi_{V;\bs_1\ldots\bs_N}^{-1}(H)\subset H_1\big(V_{\bs_1\ldots\bs_N};\Z\big),\\
\label{cRbsdfn_e}
\cR_{H;\bs_1\ldots\bs_N}'=\Im\big\{q_H\!\circ\!\Phi_{V;\bs_1\ldots\bs_N}\big\}
\subset\cR_H, \qquad
\cR_{H;\bs_1\ldots\bs_N}=\frac{\cR_H}{\cR_{H;\bs_1\ldots\bs_N}'}
\!\times\!\cR_{H;\bs_1\ldots\bs_N}'\,.
\end{gather}
If $\gcd(\bs_r)\!=\!1$ for every $r\!=\!1,\ldots,N$, then 
$$\cR_{H;\bs_1\ldots\bs_N}'=\cR_{H;\bs_1\ldots\bs_N}=\cR_H\,.$$
If $V$ is connected, then \hbox{$\cR_{H;\bs}'\!=\!\gcd(\bs)\cR_H$}
for any $\bs\!\in\!\Z_{\pm}^{\ell}$ and $H_{(1)}\!=\!H$.\\

\noindent
For each $r\!=\!1,\ldots,N$, let $\wh{V}_r\!\lra\!V_r$ be the maximal abelian cover of~$V_r$,
i.e.~the covering projection corresponding to the commutator subgroup of~$\pi_1(V)$.
The group of deck transformations of this regular covering is $H_1(V_r;\Z)$.
The maximal abelian cover of~$V_{\bs_1\ldots\bs_N}$ is given~by
\BE{whVbs_e}
\wh{V}_{\bs_1\ldots\bs_N}\equiv\prod_{r=1}^N\wh{V}_r^{\ell_r}\lra V_{\bs_1\ldots\bs_N}\,;\EE
there is a natural action of $H_1(V_{\bs};\Z)$ on this space.
For any submodule $H\!\subset\!H_1(V;\Z)$, let
\BE{HbsCov_e}\begin{split}
\pi_{H;\bs_1\ldots\bs_N}'\!:\wh{V}_{H;\bs_1\ldots\bs_N}'&
\equiv\wh{V}_{\bs_1\ldots\bs_N}\big/H_{\bs_1\ldots\bs_N} \lra V_{\bs_1\ldots\bs_N}\,,\\
\pi_{H;\bs_1\ldots\bs_N}\!:\wh{V}_{H;\bs_1\ldots\bs_N}&
\equiv \frac{\cR_H}{\cR_{H;\bs_1\ldots\bs_N}'}\!\times\!\wh{V}_{H;\bs_1\ldots\bs_N}' 
\lra V_{\bs_1\ldots\bs_N}\,.
\end{split}\EE
The groups of deck transformations of these regular coverings are
\BE{HbsDeck_e}
\Deck\big(\pi_{H;\bs_1\ldots\bs_N}'\big)=\cR_{H;\bs_1\ldots\bs_N}' 
\qquad\hbox{and}\qquad
\Deck\big(\pi_{H;\bs_1\ldots\bs_N}\big)=\cR_{H;\bs_1\ldots\bs_N},\EE 
respectively.
We will write elements of the second covering in~\eref{HbsCov_e} as
\BE{whVelem_e}\big([\ga]_{H;\bs_1\ldots\bs_N},[\wh{x}]_H\big)
\in  \frac{\cR_H}{\cR_{H;\bs_1\ldots\bs_N}'}\!\times\!\wh{V}_{H;\bs_1\ldots\bs_N}' ,\EE
with the first component denoting the image of $\ga\!\in\!H_1(V_{\bs_1\ldots\bs_N};\Z)$
under the homomorphism
$$H_1(V;\Z)\lra \cR_H\lra \frac{\cR_H}{\cR_{H;\bs_1\ldots\bs_N}'}$$
and the second component denoting the image of $\wh{x}\!\in\!\wh{V}_{\bs_1\ldots\bs_N}$.\\

\noindent
A collection $\{\ga_j\}\!\subset\!H_1(V;\Z)$ of representatives for the elements of 
$\cR_H/\cR_{H;\bs_1\ldots\bs_N}'$ induces a homomorphism
$$H_1(V;\Z)\lra \Deck(\pi_{H;\bs_1\ldots\bs_N}\big), \qquad 
\eta\lra\Th_{\eta}\,,$$
as follows.
For every $\eta\!\in\!H_1(V;\Z)$ and a coset representative~$\ga_j$,
let $\ga_j(\eta)$ be the unique coset representative from the chosen collection
such~that 
\BE{gajeta_e} \ga_j+\eta-\ga_j(\eta)-\Phi_{V;\bs_1\ldots\bs_N}(\eta_j)\in H\EE
for some $\eta_j\!\in\!H_1(V_{\bs_1\ldots\bs_N};\Z)$.
Define 
\BE{Thetadfn_e}\Th_{\eta}\!:\wh{V}_{H;\bs_1\ldots\bs_N}\lra \wh{V}_{H;\bs_1\ldots\bs_N}, ~~
\Th_{\eta}\big([\ga_j]_{H;\bs_1\ldots\bs_N},[\wh{x}]_H\big)
=\big([\ga_j(\eta)]_{H;\bs_1\ldots\bs_N},[\eta_j\!\cdot\!\wh{x}]_H\big).\EE
Since \eref{gajeta_e} determines $\eta_j$ up to an element of $H_{\bs_1\ldots\bs_N}$,
the last component of~$\Th_{\eta}$ is well-defined.\\

\noindent
The coverings $V_{\bs_1\ldots\bs_N}$ often do not have finitely generated homology groups.
Some cases when their homology groups are finitely generated are described
in \cite[Section~5.2]{GWrelIP}.\\

\noindent
The next example underlines most examples in Section~\ref{SympSum_sec}.

\begin{eg}[{\cite[Example~5.1]{GWrelIP}}]\label{Tcov_eg}
If $V\!=\!\T^2$, $\ell\!\in\!\Z^+$, and $H\!=\!\{0\}$,
then
$$\wh{V}_{H;\bs}=\C\!\times\!\T_{\bs}^{2(\ell-1)}, \quad\hbox{where}\quad
\T_{\bs}^{2(\ell-1)}=\big\{(z_i)_{i\le\ell}\!\in\!\C^{\ell}\!:
\sum_{i=1}^{\ell}\!s_iz_i\in\Z\!\oplus\!\fI\Z\big\} \big/\Z^{2\ell}
\subset \T^{2\ell}\!=\!V_{\bs}.$$ 
The second covering in~\eref{HbsCov_e} can be written~as
\BE{EScover_e}\C\!\times\!\T_{\bs}^{2(\ell-1)}\lra \T^{2\ell}, \qquad
\big(z,[z_i]_{i\le\ell}\big) \lra  \bigg[z_i\!-\!\frac{z}{s_i}\bigg]_{i\le\ell}.\EE
Under the standard identification of $H_1(\T^2;\Z)$ with $\Z\!\oplus\!\fI\Z$,
the action of $H_1(\T^2;\Z)^{\oplus\ell}$ on this cover is given~by
\BE{T2act_e1} (\ga_{i'})_{i'\le\ell}\cdot\big(z,[z_i]_{i\le\ell}\big) 
=\bigg(z\!+\!\frac1\ell\sum_{i'=1}^{\ell}s_{i'}\ga_{i'},
\bigg[z_i\!+\!\frac{1}{\ell s_i}\sum_{i'=1}^{\ell}s_{i'}\ga_{i'}\bigg]_{i\le\ell}\bigg)\,.\EE
The group of deck transformations of this cover is $\Z_{\gcd(\bs)}^{\,2}\!\oplus\!\gcd(\bs)\Z^2$.
The action of the second component is induced by the action of $H_1(\T^2;\Z)^{\oplus\ell}$ via
the surjective homomorphism 
\BE{T2act_e2} H_1(\T^2;\Z)^{\oplus\ell}\lra \gcd(\bs)H_1(\T^2;\Z), \qquad
(\ga_{i'})_{i'\le\ell}\lra \sum_{i'=1}^{\ell}s_{i'}\ga_{i'}\,.\EE
\end{eg}

\subsection{The rim tori covers and lifts}
\label{RelGW_subs}

\noindent
The coverings~\eref{IPcov_e} are defined below as special cases
of the abelian covers~\eref{HbsCov_e}.
The relative evaluation morphisms~\eref{evdfn_e2} lift
over these coverings,  though not uniquely.
Fixing lifts of these morphisms (as is essentially done in~\cite{IPrel}) corresponds
to choosing base points in certain spaces; see \cite[Remark~6.7]{GWrelIP}.
These choices can be made in a systematic manner; see Proposition~\ref{RimToriAct_prp}.\\

\noindent
Let $X$ be a manifold and $V\!\subset\!X$ be a closed submanifold of codimension~$\fc$.
Denote by $S_XV\!\subset\!\cN_XV$ the sphere subbundle of the normal bundle of~$V$ in~$X$,
which we will view as a hypersurface in~$X$, and~let
\BE{cRXVdfn_e}\cR_X^V\equiv \ker\big\{\io^X_{X-V*}\!:
H_{\fc}(X\!-\!V;\Z)\!\lra\!H_{\fc}(X;\Z)\big\}.\EE
If in addition $X$ and $V$ are compact and oriented,
we~define
\begin{alignat*}{2}
\cap V\!: H_*(X;\Z)&\lra H_{*-\fc}(V;\Z), &\qquad A\cap V&=\PD_V\big((\PD_XA)|_V\big),\\
\De_X^V\!: H_m(V;\Z)&\lra H_{m+\fc-1}(S_XV;\Z), &\qquad
\De_X^V(\ga)&=\PD_{S_XV}\big(q_X^{V\,^*}(\PD_V\ga)\big),
\end{alignat*}
where $q_X^V\!:S_XV\!\lra\!V$ is the projection map.
Let
\BE{H1VXdfn_e} H^V_X=
\big\{A\!\cap\!V\!: A\!\in\!H_{\fc+1}(X;\Z)\big\}\subset H_1(V;\Z), \quad
H_1(V;\Z)_X=\frac{H_1(V;\Z)}{H^V_X}.\EE
By \cite[Corollary~3.2]{GWrelIP}, the homomorphism
\BE{RTisom_e}\io_{S_XV*}^{X-V}\!\circ\!\De_X^V\!: H_1(V;\Z)_X\lra 
\cR_X^V \subset H_{\fc}(X\!-\!V;\Z)\EE
is well-defined and is an isomorphism.\\

\noindent
With notation as in~\eref{cRH},
$$\cR_{H_X^V}\equiv H_1(V;\Z)_X\approx\cR_X^V\,.$$
Let $V_1,\ldots,V_N$ be the topological components of~$V$.
For $\ell_1,\ldots,\ell_N\!\in\!\Z^{\ge0}$ and
$\bs_1\!\in\!\Z_{\pm}^{\,\ell_1},\ldots,\bs_N\!\in\!\Z_{\pm}^{\,\ell_N}$, define
\begin{gather*}
H_{X;\bs_1\ldots\bs_N}^V=\big(H_X^V\big)_{\bs_1\ldots\bs_N}\subset H_1(V_{\bs_1\ldots\bs_N};\Z),\\
\cR_{X;\bs_1\ldots\bs_N}'^{\,V}=\cR_{H_X^V;\bs_1\ldots\bs_N}'\subset H_1(V;\Z)_X\,,
\qquad  \cR_{X;\bs_1\ldots\bs_N}^V=\cR_{H_X^V;\bs_1\ldots\bs_N}\,.
\end{gather*}
The rim tori covers~\eref{IPcov_e} are the abelian covers
\BE{IPcov_e2}\begin{split}
\pi_{X;\bs_1\ldots\bs_N}'^V\!\equiv\!\pi_{H_X^V;\bs_1\ldots\bs_N}'^V\!:
\wh{V}_{X;\bs_1\ldots\bs_N}'&\equiv\wh{V}_{H_X^V;\bs_1\ldots\bs_N}'
\lra V_{\bs_1\ldots\bs_N},\\
\pi_{X;\bs_1\ldots\bs_N}^V\!\equiv\!\pi_{H_X^V;\bs_1\ldots\bs_N}^V\!: 
\wh{V}_{X;\bs_1\ldots\bs_N}&\equiv\wh{V}_{H_X^V;\bs_1\ldots\bs_N}
\lra V_{\bs_1\ldots\bs_N}.
\end{split}\EE
By~\eref{HbsDeck_e}, the groups of deck transformations of these regular coverings are
\BE{IPdeck_e}\Deck\big(\pi_{X;\bs_1\ldots\bs_N}'^V\big)=\cR_{X;\bs_1\ldots\bs_N}'^{\,V}
\qquad\hbox{and}\qquad
\Deck\big(\pi_{X;\bs_1\ldots\bs_N}^V\big) =\cR_{X;\bs_1\ldots\bs_N}^V\,,\EE
respectively.
We will write elements of the second covering in~\eref{IPcov_e2} as
\BE{whVelem_e2}\big([\ga]_{X;\bs_1\ldots\bs_N},[\wh{x}]_X\big)
\in  \frac{\cR_X^V}{\cR_{X;\bs_1\ldots\bs_N}'^V}\!\times\!\wh{V}_{X;\bs_1\ldots\bs_N}'\,,\EE
with notation as in~\eref{whVelem_e} for $H\!=\!H_X^V$.\\

\noindent
If $\Si$ is a  compact oriented $m$-dimensional manifold, $A\!\in\!H_m(X;\Z)$,
$k\!\in\!\Z^{\ge0}$, and \hbox{$p\!>\!m$}, let $\fX_{\Si,k}(X,A)$ be the space of tuples
$(z_1,\ldots,z_k,f)$
such~that $f\!\in\!L^p_1(\Si;X)$,
$f_*[\Si]\!=\!A$, and  \hbox{$z_1,\ldots,z_k\in\Si$} are distinct points. 
If $m\!=\!\fc$ and $r\!=\!1,\ldots,N$, each isolated point $z\!\in\!f^{-1}(V_r)$
has well-defined order of contact with~$V_r$, $\ord_z^{V_r}\!f\!\in\!\Z$; see the beginning of 
\cite[Section~2.1]{GWrelIP}. 
If in addition $\bs_1,\ldots,\bs_N$ are as before and \eref{bsumcond_e} holds for each
$(V,\bs)\!=\!(V_r,\bs_r)$, let 
$$\fX_{\Si,k;\bs_1\ldots\bs_N}^{V_1,\ldots,V_N}(X,A)\subset 
\fX_{\Si,k+\ell_1+\ldots+\ell_N}(X,A)$$
be the subspace of tuples $(z_1,\ldots,z_{k+\ell_1+\ldots+\ell_N},f)$ such~that 
\begin{alignat*}{2}
&f^{-1}(V_r)=\big\{z_{k+\ell_1+\ldots+\ell_{r-1}+1},\ldots,
z_{k+\ell_1+\ldots+\ell_r}\big\}
&\qquad &\forall~r=1,\ldots,N,\\
&\ord_{z_{k+\ell_1+\ldots+\ell_{r-1}+i}}^{V_r}f=s_{r;i}
&\qquad &\forall~i\!=\!1,2,\ldots,\ell_r,~r\!=\!1,\ldots,N.
\end{alignat*}
We denote~by 
\BE{evXVdfn_e}
\ev_X^V\!=\!\ev_{k+1}\!\times\!\ldots\!\times\!\ev_{k+\ell_1+\ldots+\ell_N}\!:
\fX_{\Si,k;\bs_1\ldots\bs_N}^{V_1,\ldots,V_N}(X,A)
\lra V_{\bs_1\ldots\bs_N}\EE
the total relative evaluation morphism.
Any pair $\bff,\bff'\!\in\!\fX_{\Si,k;\bs_1\ldots\bs_N}^{V_1,\ldots,V_N}(X,A)$ with
$$\ev_X^V\big(\bff)=\ev_X^V(\bff')\in V_{\bs}$$
determines an element $[f\!\#\!(-f')]$ of $\cR_X^V\!\subset\!H_{\fc}(X\!-\!V;\Z)$;
see \cite[Section~2.1]{GWrelIP}.\\

\noindent
By \cite[Lemma~6.3]{GWrelIP}, the morphism~\eref{evXVdfn_e} lifts over 
the coverings~\eref{IPcov_e2}.
Such lifts can be chosen so that they are compatible with the morphisms between
the configuration spaces of maps obtained by dropping some of the components
of~$V$ and the corresponding relative contact points;
see \cite[Figure~1]{GWrelIP}.
They can also be chosen compatibly with the isomorphism~\eref{RTisom_e}, in the sense 
described~below.

\begin{prp}[{\cite[Theorem~6.5]{GWrelIP}}]\label{RimToriAct_prp}
Suppose $X$ is a compact oriented manifold, $V\!\subset\!X$ is a 
compact oriented submanifold of codimension~$\fc$ with connected components
$V_1,\ldots,V_N$, $A\!\in\!H_{\fc}(X;\Z)$, and 
$\bs_r\!\in\!\Z_{\pm}^{\,\ell_r}$ for \hbox{$r\!=\!1,\ldots,N$}.
Let $\{\ga_j\}\!\subset\!H_1(V;\Z)$ be a collection of representatives
for the elements of $\cR_X^V/\cR_{X;\bs_1\ldots\bs_N}'^{\,V}$.
If $\Si$ is a  compact oriented $\fc$-dimensional manifold and $k\!\in\!\Z^{\ge0}$,
there exists a~lift 
\BE{EvLift_e}\wt\ev_X^V\!:
\fX_{\Si,k;\bs_1\ldots\bs_N}^{V_1,\ldots,V_N}(X,A) \lra 
\wh{V}_{X;\bs_1\ldots\bs_N}\EE
of the morphism~$\ev_X^V$ in~\eref{evXVdfn_e} 
over the covering $\pi_{X;\bs_1\ldots\bs_N}^V$ in~\eref{IPcov_e2} with the following property.
For any $\bff,\bff'\!\in\!\fX_{\Si,k;\bs_1\ldots\bs_N}^{V_1,\ldots,V_N}(X,A)$ with
\BE{RTmatch_e}\wt\ev_X^V(\bff)=\big([\ga_j]_{X;\bs_1\ldots\bs_N},[\ga\!\cdot\!\wh{x}]_X\big)
\quad\hbox{and}\quad 
\wt\ev_X^V(\bff')=\big([\ga_{j'}]_{X;\bs_1\ldots\bs_N},[\wh{x}]_X\big)\EE
for some $\wh{x}\!\in\!\wh{V}_{\bs_1\ldots\bs_N}$, 
$\ga\!\in\!H_1(V_{\bs_1\ldots\bs_N};\Z)$, and $j,j'$ indexing the coset representatives,
the map components of~$\bff$ and~$\bff'$ satisfy
\BE{RimToriAct_e}\big[f\!\#\!(-f')\big]=
\io_{S_XV*}^{X-V}\big(\De_X^V\big(\Phi_{V;\bs_1\ldots\bs_N}(\ga)\!+\!\ga_j\!-\!\ga_{j'}\big)\big)
\in H_{\fc}(X\!-\!V;\Z).\EE
Furthermore, $\wt\ev_X^V(\bff')$ is the unique point in 
$\pi_{X;\bs_1\ldots\bs_N}^{V~-1}(\ev_X^V(\bff'))$ so that \eref{RimToriAct_e} holds
for a given value of~$\wt\ev_X^V(\bff)$.
\end{prp}

\section{Diagonal components for abelian covers}
\label{DiagComp_sec}

\noindent
In Section~\ref{DiagonSplit_subs}, 
we define diagonal components for arbitrary abelian covers that 
specialize to~\eref{cHXYVC_e} in the symplectic sum setting.
Each \sf{diagonal component} $\wh{V}_{H_1,H_2;\bs_1\ldots\bs_N}^{\eta}$ 
in~\eref{whVdiag_e} determines an intersection homomorphism~\eref{cHinter_e} 
on the product of two abelian covers of~$V_{\bs_1\ldots\bs_N}$ and thus
a cohomology class~\eref{GenPD_e}, as suggested by \cite[Definition~10.2]{IPsum};
in the symplectic sum context, this class specializes to~\eref{PDcXY_e}. 
In Section~\ref{DiagonProp_subs}, we describe cases when these classes split into 
products of cohomology classes from the two factors and 
when these classes are the same for different choices of~$\eta$;
see Lemmas~\ref{Kunn_lmm} and~\ref{DiagFlux_lmm}.

\subsection{Notation and examples}
\label{DiagonSplit_subs}

\noindent
We continue with the notation for the abelian covers of a topological space~$V$
with connected components $V_1,\ldots,V_N$ introduced in Section~\ref{AbCov_subs}.
For $\bs_1\!\in\!\Z_{\pm}^{\ell_1},\ldots,\bs_N\!\in\!\Z_{\pm}^{\ell_N}$, denote by
$$\De_{\bs_1\ldots\bs_N}^V\subset V_{\bs_1\ldots\bs_N}^2\!\equiv\! 
V_{\bs_1\ldots\bs_N}\!\times\!V_{\bs_1\ldots\bs_N}$$
the diagonal. 
For any submodules $H_1,H_2\!\subset\!H_1(V;\Z)$, define
\BE{whVdiag_e0}\begin{split}
\wh{V}_{H_1,H_2;\bs_1\ldots\bs_N}' &=
 \wh{V}_{H_1;\bs_1\ldots\bs_N}'\!\times_{V_{\bs_1\ldots\bs_N}}\!\wh{V}_{H_2;\bs_1\ldots\bs_N}'
\equiv
\big\{\pi_{H_1;\bs_1\ldots\bs_N}'\!\times\!\pi_{H_2;\bs_1\ldots\bs_N}'\big\}^{-1}
\big(\De_{\bs_1\ldots\bs_N}^V\big),\\
\wh{V}_{H_1,H_2;\bs_1\ldots\bs_N} &=
 \wh{V}_{H_1;\bs_1\ldots\bs_N}\!\times_{V_{\bs_1\ldots\bs_N}}\!\wh{V}_{H_2;\bs_1\ldots\bs_N}
\equiv
\big\{\pi_{H_1;\bs_1\ldots\bs_N}\!\times\!\pi_{H_2;\bs_1\ldots\bs_N}\big\}^{-1}
\big(\De_{\bs_1\ldots\bs_N}^V\big).
\end{split}\EE
Thus,
\begin{equation*}\begin{split}
\wh{V}_{H_1,H_2;\bs_1\ldots\bs_N}' &=\bigcup_{\ga\in H_1(V_{\bs_1\ldots\bs_N};\Z)}\hspace{-.35in}
\big\{\big([\ga\!\cdot\!\wh{x}]_{H_1},[\wh{x}]_{H_2}\big)\!:\,\wh{x}\!\in\!\wh{V}_{\bs_1\ldots\bs_N}\big\},\\
\wh{V}_{H_1,H_2;\bs_1\ldots\bs_N}
&=\frac{\cR_{H_1}}{\cR_{H_1;\bs_1\ldots\bs_N}'}\!\times\!
\frac{\cR_{H_2}}{\cR_{H_2;\bs_1\ldots\bs_N}'}\!\times\!\wh{V}_{H_1,H_2;\bs_1\ldots\bs_N}'\,.
\end{split}\end{equation*}\\

\noindent
If $H_{12}\!\subset\!H_1(V;\Z)$ is a module containing $H_1$ and $H_2$, define 
$$\Psi_{H_1,H_2}'^{H_{12}}\!\!: \wh{V}_{H_1,H_2;\bs_1\ldots\bs_N}'\lra
\cR_{H_{12};\bs_1\ldots\bs_N}'\subset\cR_{H_{12}}\,, ~~
\Psi_{H_1,H_2}'^{H_{12}}\big([\ga\!\cdot\!\wh{x}]_{H_1},[\wh{x}]_{H_2}\big)
=\big[\Phi_{V;\bs_1\ldots\bs_N}(\ga)\big]_{H_{12}}\,,$$
where $[\ga']_{H_{12}}$ denotes the coset of $\ga'\!\in\!H_1(V;\Z)$ modulo~$H_{12}$.
Since~$\ga$ above is well-defined up to an element~of
$$(H_1)_{\bs_1\ldots\bs_N}+(H_2)_{\bs_1\ldots\bs_N}\subset (H_{12})_{\bs_1\ldots\bs_N}\,,$$
the map $\Psi_{H_1,H_2}'^{H_{12}}$ is well-defined.
Combining $\Psi_{H_1,H_2}'^{H_{12}}$ with the homomorphism,
\BE{cRmap_e}\begin{split}
\frac{\cR_{H_1}}{\cR_{H_1;\bs_1\ldots\bs_N}'}\!\times\!\frac{\cR_{H_2}}{\cR_{H_2;\bs_1\ldots\bs_N}'}
&\lra \frac{\cR_{H_{12}}}{\cR_{H_{12};\bs_1\ldots\bs_N}'}, \\
\big([\ga_1]_{H_1;\bs_1\ldots\bs_N},[\ga_2]_{H_2;\bs_1\ldots\bs_N}\big) &\lra
[\ga_1\!-\!\ga_2]_{H_{12};\bs_1\ldots\bs_N},
\end{split}\EE
we obtain a continuous~map
$$\Psi_{H_1,H_2}^{H_{12}}\!\!: \wh{V}_{H_1,H_2;\bs_1\ldots\bs_N}\lra
\cR_{H_{12};\bs_1\ldots\bs_N}\,;$$
its target is a discrete set.\\

\noindent
The map~\eref{cRmap_e} can be lifted to a map (not a homomorphism) to~$\cR_{H_{12}}$ 
by choosing collections 
\BE{gaj1j2_e} \{\ga_{1;j_1}\},\{\ga_{2;j_2}\}\subset H_1(V;\Z)\EE
of representatives
for the elements of $\cR_{H_1}/\cR_{H_1;\bs_1\ldots\bs_N}'$ and $\cR_{H_2}/\cR_{H_2;\bs_1\ldots\bs_N}'$,
respectively:
$$\big([\ga_{1;j_1}]_{H_1;\bs_1\ldots\bs_N},[\ga_{2;j_2}]_{H_2;\bs_1\ldots\bs_N}\big) \lra
[\ga_{1;j_1}\!-\!\ga_{2;j_2}]_{H_{12}}.$$
We can then ``lift'' $\Psi_{H_1,H_2}^{H_{12}}$ to a continuous map
\begin{gather}\label{wtPsi_e}
\wt\Psi_{H_1,H_2}^{H_{12}}\!\!: \wh{V}_{H_1,H_2;\bs_1\ldots\bs_N}\lra \cR_{H_{12}}\,,\\
\notag
\wt\Psi_{H_1,H_2}^{H_{12}}\big( 
\big([\ga_{1;j_1}]_{H_1;\bs_1\ldots\bs_N},[\ga\!\cdot\!\wh{x}]_{H_1}\big),
\big([\ga_{2;j_2}]_{H_2;\bs_1\ldots\bs_N},[\wh{x}]_{H_2} \big)\big)
=\big[\Phi_{V;\bs_1\ldots\bs_N}(\ga)\!+\!\ga_{1;j_1}\!-\!\ga_{2;j_2}\big]_{H_{12}}\,.
\end{gather}
The latter gives rise to the refined gluing degree map~\eref{IPdegmap_e}
in the symplectic sum context; see~\eref{gAdfn_e}. 
This map is compatible with the deck transformations~\eref{Thetadfn_e}
associated with the collections~\eref{gaj1j2_e}, i.e.
\BE{PsiThDeg_e}\begin{split} 
\wt\Psi_{H_1,H_2}^{H_{12}}\big(\Th_{\eta}(\wt{x}_1),\wt{x}_2\big)&= 
\wt\Psi_{H_1,H_2}^{H_{12}}\big(\wt{x}_1,\wt{x}_2\big)+[\eta]_{H_{12}},\\
\wt\Psi_{H_1,H_2}^{H_{12}}\big(\wt{x}_1,\Th_{\eta}(\wt{x}_2)\big)&= 
\wt\Psi_{H_1,H_2}^{H_{12}}\big(\wt{x}_1,\wt{x}_2\big)-[\eta]_{H_{12}}
\end{split}\EE
for all $(\wt{x}_1,\wt{x}_2)\!\in\!\wh{V}_{H_1,H_2;\bs_1\ldots\bs_N}$ and  $\eta\!\in\!H_1(V;\Z)$.\\

\noindent
For $\eta\!\in\!\cR_{H_{12};\bs_1\ldots\bs_N}$
and $\wt\eta\!\in\!\cR_{H_{12}}$, let
\BE{whVdiag_e} \wh{V}_{H_1,H_2;\bs_1\ldots\bs_N}^{\eta}=
\big\{\Psi_{H_1,H_2}^{H_{12}}\big\}^{-1}(\eta), \qquad
\wh{V}_{H_1,H_2;\bs_1\ldots\bs_N}^{\wt\eta}=
\big\{\wt\Psi_{H_1,H_2}^{H_{12}}\big\}^{-1}(\wt\eta).\EE
These subsets of $\wh{V}_{H_1;\bs_1\ldots\bs_N}\!\times\!\wh{V}_{H_2;\bs_1\ldots\bs_N}$
are closed.\\

\noindent
If in addition $V$ is an oriented manifold, then so are the subsets~\eref{whVdiag_e}.
Thus, they define intersection homomorphisms
\BE{cHinter_e}
H_*\big(\wh{V}_{H_1;\bs_1\ldots\bs_N}\!\times\!\wh{V}_{H_2;\bs_1\ldots\bs_N};\Q\big)\lra \Q, 
\quad
h\lra h\cdot \wh{V}_{H_1,H_2;\bs_1\ldots\bs_N}^{\eta}, 
h\cdot \wh{V}_{H_1,H_2;\bs_1\ldots\bs_N}^{\wt\eta},\EE
as follows.
On the homology of dimension different from half the dimension,  
we take these homomorphisms to be~zero.
A (rational multiple of a) homology class of half the dimension can be represented by
a smooth map
\BE{hZdfn_e} h\!:Z\!\lra\!\wh{V}_{H_1;\bs_1\ldots\bs_N}\!\times\!\wh{V}_{H_2;\bs_1\ldots\bs_N}\EE
from a compact oriented manifold which intersects the submanifolds~\eref{whVdiag_e} transversely.
Thus, their preimages in~$Z$ are closed zero-dimensional submanifolds of~$Z$.
We take the intersection numbers in~\eref{cHinter_e} to be the signed cardinalities of 
these finite sets (divided by the appropriate multiple if necessary). 
A cobordism between two representatives for the same homology class 
provides cobordisms between the preimages of the submanifolds~\eref{whVdiag_e}
with respect to the two representatives;
thus, the intersection homomorphisms~\eref{cHinter_e} are well-defined.
By the Universal Coefficient Theorem \cite[Theorem~53.5]{Mu2},
these homomorphisms correspond to some elements
\BE{GenPD_e}
\PD_{H_1,H_2;\bs_1\ldots\bs_N}^{\eta}\De,\PD_{H_1,H_2;\bs_1\ldots\bs_N}^{\wt\eta}\De\in 
H^*\big(\wh{V}_{H_1;\bs_1\ldots\bs_N}\!\times\!\wh{V}_{H_2;\bs_1\ldots\bs_N};\Q\big),\EE
respectively.
Since the homologies of $\wh{V}_{H_1;\bs_1\ldots\bs_N}$ and $\wh{V}_{H_2;\bs_1\ldots\bs_N}$
may not be finitely generated, these classes need not admit finite Kunneth decompositions 
as in~\eref{KunnDecomp_e2}.
By Examples~\ref{ESsum_eg0}-\ref{S2T2_eg0} and Lemma~\ref{Kunn_lmm} below, 
they do admit such decompositions in some interesting cases.

\begin{eg}\label{ESsum_eg0}
Let $V\!=\!\T^2$, $\ell\!\in\!\Z^+$, and $\bs\!\in\!\Z_{\pm}^{\ell}$.
We identify $H_1(V;\Z)$ with $\Z\!\oplus\!\fI\Z$ and denote by $0\!\subset\!H_1(V;\Z)$  
the zero submodule.
By Example~\ref{Tcov_eg},
$$\wh{V}_{0,0;\bs}=\big\{\big(x,[z_i\!+\!s_i^{-1}x]_{i\le\ell},
y,[z_i\!+\!s_i^{-1}y]_{i\le\ell}\big)\!\in\!
\C\!\times\!\T_{\bs}^{2(\ell-1)}\!\times\!\C\!\times\!\T_{\bs}^{2(\ell-1)}\big\}.$$
Let $\{\ga_j\}\!\subset\!\Z\!\oplus\!\fI\Z$ be a collection of representatives for 
the elements of $\Z_{\gcd(\bs)}\!\oplus\!\fI\Z_{\gcd(\bs)}$.
It is convenient to identify base points for the components of~$\wh{V}_{0;\bs}$~as
$$\big([\ga_j]_{0;\bs},[\wh{x}_{\bs}]_0\big)=
\big(\ell^{-1}\ga_j,[\ell^{-1}s_i^{-1}\ga_j]_{i\le\ell}\big)\in
\wh{V}_{0;\bs}= \C\!\times\!\T_{\bs}^{2(\ell-1)}.$$
The map~\eref{wtPsi_e} with $H_1,H_2,H_{12}\!=\!0$ is then given~by
$$\wt\Psi_{0,0;\bs}^0\!: \wh{V}_{0,0;\bs}
 \lra \Z\!\oplus\!\fI\Z,\qquad
\wt\Psi_{0,0;\bs}^0\big(x,[\z],y,[\z']\big)= \ell\,(x\!-\!y);$$
see~\eref{T2act_e1} and~\eref{T2act_e2}.
The diagonal components $\wh{V}_{0,0;\bs}^{\wt\eta}$ are naturally indexed~by $\Z\!\oplus\!\fI\Z$
and
\BE{ESsum_e0a}
\wh{V}_{0,0;\bs}^{\wt\eta}=\De_{\C}^{\wt\eta}\!\times\!\De_{\T^{2\ell}}^{\wt\eta}
\cap \C\!\times\!\T_{\bs}^{2(\ell-1)}\!\times\!\C\!\times\!\T_{\bs}^{2(\ell-1)}
\subset \C\!\times\!\T^{2\ell}\!\times\!\C\!\times\!\T^{2\ell}\,, \EE
where
$$\De_{\C}^{\wt\eta}=\{ (x\!+\!\wt\eta,x)\!\in\!\C^2\big\}
\quad\hbox{and}\quad
\De_{\T^{2\ell}}^{\wt\eta}=\big\{\big([z_i\!+\!s_i^{-1}\wt\eta]_{i\le\ell},[z_i]_{i\le\ell}\big)
\!\in\!(\T_{\bs}^{2(\ell-1)})^2\big\}$$
are translates of the diagonals.
In particular,
$$h\cdot\big(\De_{\C}^{\wt\eta}\!\times\!\De_{\T^{2\ell}}^{\wt\eta}\Big)
=0 \qquad\forall~
h\!\in\!H_*\big(\C\!\times\!\T^{2\ell}\!\times\!\C\!\times\!\T^{2\ell};\Q\big),~
\wt\eta\!\in\!\Z\!\oplus\!\fI\Z.$$
Combining the last observation with~\eref{ESsum_e0a},
\BE{ESsum_e} \PD_{0,0;\bs}^{\wt\eta}\De=0 \in 
H^*\big(\C\!\times\!\T_{\bs}^{2(\ell-1)}\!\times\!\C\!\times\!\T_{\bs}^{2(\ell-1)};\Q\big)
 \qquad\forall~\wt\eta\!\in\!\Z\!\oplus\!\fI\Z.\EE
This is consistent with the vanishing of the GW-invariants of $\bK_3$; 
see Example~\ref{ESsum_eg}.
\end{eg}

\begin{eg}\label{ESsum_eg2a}
For $V\!=\!\T^2$, $\ell\!\in\!\Z^+$, $\bs\!\in\!\Z_{\pm}^{\ell}$, and $H\!=\!H_1(V;\Z)$,
\begin{gather*}
\wt\Psi_{0,H;\bs}^H\!:\wh{V}_{0,H;\bs}= 
\big\{\big(z,[z_i]_{i\le\ell},[z_i\!-\!s_i^{-1}z]_{i\le\ell}\big)\!\in\!
\C\!\times\!\T_{\bs}^{2(\ell-1)}\!\times\!\T^{2\ell}\big\}
\lra \cR_H\!=\!\{0\}.
\end{gather*}
The preimage of $\wh{V}_{0,H;\bs}\!=\!\wh{V}_{0,H;\bs}^0$ under the automorphism
$$\Th_{\bs}\!: \C\!\times\!\T_{\bs}^{2(\ell-1)}\!\times\!\T^{2\ell}
\lra \C\!\times\!\T_{\bs}^{2(\ell-1)}\!\times\!\T^{2\ell},\quad
\Th_{\bs}\big(z,[z_i]_{i\le\ell},[z_i']_{i\le\ell}\big)=
\big(z,[z_i]_{i\le\ell},[z_i'\!-\!s_i^{-1}z]_{i\le\ell}\big),$$
is the intersection 
$$\C\!\times\!\De_{\T^{2\ell}}\cap \C\!\times\!\T_{\bs}^{2(\ell-1)}\!\times\!\T^{2\ell}
\subset \C\!\times\!\T^{2\ell}\!\times\!\T^{2\ell}\,.$$
Since $\Th_{\bs}$ induces the identity on the cohomology, it follows~that 
\BE{ESsum_e2}
\PD_{0,H;\bs}^0\De=
1\times\big(\PD_{(\T^{2\ell})^2}\De_{\T^{2\ell}}\big)\big|_{\T_{\bs}^{2(\ell-1)}\times\T^{2\ell}}
\in H^*\big(\C\!\times\!\T_{\bs}^{2(\ell-1)}\!\times\!\T^{2\ell};\Q\big).\EE
In the $\ell\!=\!1$ case, $\T_{\bs}^{2(\ell-1)}$ consists of $|\bs|^2$ points and 
\eref{ESsum_e2} reduces~to
\BE{ESsum_e2b}
\PD_{0,H;(s)}^0\De=1\!\times\!\PD_{\T^2}(\pt)\in  
H^*\big(\{1,\ldots,s^2\}\!\times\!\C\!\times\!\T^2;\Q\big).\EE
In the symplectic sum decomposition for $\wh\P^2_9\!=\!\wh\P^2_9\#_V\!(\P^1\!\times\!V)$, 
\eref{ESsum_e2b} corresponds 
to putting the whole fiber on the $X$-side and a point on the $Y$-side;
see Example~\ref{ESsum_eg2}.
\end{eg}

\begin{eg}\label{S2T2_eg0}
Let $F$ be a compact connected oriented manifold, $V\!=\!\{0,\i\}\!\times\!F$,
$\ell_1,\ell_2\!\in\!\Z^{\ge0}$,
$\bs_1\!\in\!\Z_{\pm}^{\ell_1}$, and $\bs_2\!\in\!\Z_{\pm}^{\ell_2}$.
We~take
$$H_1=H_2=H_{12}= H_{\De}\subset H_1(V;\Z)=H_1(F;\Z)\oplus H_1(F;\Z)$$
to be the diagonal subgroup.
By \cite[Example~6.2]{GWrelIP}, $\wh{V}_{H_{\De};\bs_1\bs_2}\!=\!\wh{F}_{0;\bs}$,
where $\bs$ is the merged tuple of $\bs_1$ and~$-\bs_2$.
With the identifications of \cite[Examples~3.6,4.7]{GWrelIP},
the map~\eref{wtPsi_e} becomes
\begin{gather*}
\frac{H_1(F;\Z)}{\gcd(\bs)H_1(F;\Z)}\times 
\frac{H_1(F;\Z)}{\gcd(\bs)H_1(F;\Z)}
\times \wh{F}_{0;\bs}'\!\times_{F_{\bs}}\!\wh{F}_{0;\bs}' \lra H_1(F;\Z),\\
\wt\Psi_{H_{\De},H_{\De};\bs_1\bs_2}^{H_{\De}}
\big([\ga_{j_1}]_{0;\bs},[\ga_{j_2}]_{0;\bs},[\ga\!\cdot\!\wh{x}]_0,[\wh{x}]_0\big) 
=\Phi_{F;\bs}(\ga)\!+\!\ga_{j_1}\!-\!\ga_{j_2}\,.
\end{gather*}
If $F\!=\!\T^2$ and $\ell\!\equiv\!\ell_1\!+\!\ell_2\!>\!0$, we associate representatives $\ga_j$ for 
the elements of $\Z_{\gcd(\bs)}\!\oplus\!\fI\Z_{\gcd(\bs)}$ with base points for 
the connected components of $\C\!\times\!\T_{\bs}^{2(\ell-1)}$
as in Example~\ref{ESsum_eg0}.
Then,
\begin{gather*}
\begin{split}
&\wh{V}_{H_{\De},H_{\De};\bs_1\bs_2}=
\big\{\big(x,\big([z_{1;i}\!+\!s_{1;i}^{-1}x]_{i\le\ell_1},[z_{2;i}\!-\!s_{2;i}^{-1}x]_{i\le\ell_2}\big),
y,\big([z_{1;i}\!+\!s_{1;i}^{-1}y]_{i\le\ell_1},[z_{2;i}\!-\!s_{2;i}^{-1}y]_{i\le\ell_2}\big)\big)\!:\\
&\qquad x,y\!\in\!\C,
\big([z_{1;i}\!+\!s_{1;i}^{-1}x]_{i\le\ell_1},[z_{2;i}\!-\!s_{2;i}^{-1}x]_{i\le\ell_2}\big),
\big([z_{1;i}\!+\!s_{1;i}^{-1}y]_{i\le\ell_1},[z_{2;i}\!-\!s_{2;i}^{-1}y]_{i\le\ell_2}\big)
\in\T_{\bs}^{2(\ell-1)}\big\},
\end{split}\\
\wt\Psi_{H_{\De},H_{\De};\bs_1\bs_2}^{H_{\De}}\!: 
\wh{V}_{H_{\De},H_{\De};\bs_1\bs_2} \lra \Z\!\oplus\!\fI\Z,\qquad
\wt\Psi_{H_{\De},H_{\De};\bs_1\bs_2}^{H_{\De}}\big(x,[\z],y,[\z']\big)= \ell\,(x\!-\!y).
\end{gather*}
The diagonal components $\wh{V}_{H_{\De},H_{\De};\bs_1\bs_2}^{\wt\eta}$ are again indexed by 
$\Z\!\oplus\!\fI\Z$ and have the same structure as in Example~\ref{ESsum_eg0}.
Thus, 
\BE{S2T2sum_e} \PD_{H_{\De},H_{\De};\bs_1\bs_2}^{\wt\eta}\De=0 \in 
H^*\big(\wh{V}_{H_{\De};\bs_1\bs_2}\!\times\!\wh{V}_{H_{\De};\bs_1\bs_2};\Q\big)
 \qquad\forall~\wt\eta\!\in\!\Z\!\oplus\!\fI\Z.\EE
This is consistent with the vanishing of the GW-invariants of~$\T^2\!\times\!\T^2$;
see Example~\ref{S2T2sum_eg}.
\end{eg}

\subsection{Some properties}
\label{DiagonProp_subs}

\noindent
We begin this section by describing cases when the diagonal classes~\eref{GenPD_e} 
split into products of cohomology classes from the two factors. 
We then show in certain cases these classes are the same for different choices of~$\eta$.

\begin{lmm}\label{Kunn_lmm}
Suppose $V$ is a compact oriented manifold with topological components $V_1,\ldots,V_N$, 
$$H_1,H_2\subset H_{12}\subset H_1(V;\Z)$$ 
are submodules, $\bs_r\!\in\!\Z_{\pm}^{\,\ell_r}$,  
$\eta\!\in\!\cR_{H_{12};\bs_1\ldots\bs_N}$, and $\wt\eta\!\in\!\cR_{H_{12}}$.
The classes~\eref{GenPD_e} admit finite Kunneth decompositions as
in~\eref{KunnDecomp_e2} if either
$\cR_{H_{12}}$ is finite or
$V$ is connected and $H_*(\wh{V}_{H_{12}};\Q)$ is finitely generated.
\end{lmm}

\begin{proof}
The inclusions $H_i\!\subset\!H_{12}$ induce projections
$$\Th_{H_{12},H_i}'\!: \wh{V}_{H_i;\bs_1\ldots\bs_N}'\!\equiv\!
\wh{V}_{\bs_1\ldots\bs_N}\big/
(H_i)_{\bs_1\ldots\bs_N} \lra 
\wh{V}_{H_{12};\bs_1\ldots\bs_N}'\!\equiv\!\wh{V}_{\bs_1\ldots\bs_N}\big/
(H_{12})_{\bs_1\ldots\bs_N}\,.$$
Combining them with the restrictions of the homomorphism~\eref{cRmap_e} to each component
of the domain, we obtain a covering projection
$$\Th\!: \wh{V}_{H_1;\bs_1\ldots\bs_N}\!\times\!\wh{V}_{H_2;\bs_1\ldots\bs_N}
\lra \wh{V}_{H_{12};\bs_1\ldots\bs_N}\!\times\!\wh{V}_{H_{12};\bs_1\ldots\bs_N}$$
such that 
\begin{gather}
\notag
\pi_{H_1;\bs_1\ldots\bs_N}\!\times\!\pi_{H_2;\bs_1\ldots\bs_N}
=\pi_{H_{12};\bs_1\ldots\bs_N}\!\times\!\pi_{H_{12};\bs_1\ldots\bs_N}\circ\Th, \\
\label{PsiTh_e}
\Psi_{H_1,H_2}^{H_{12}}=\Psi_{H_{12},H_{12}}^{H_{12}}\circ\Th\big|_{ \wh{V}_{H_1,H_2;\bs_1\ldots\bs_N}}.
\end{gather}
Thus,
\BE{PDpullback_e}
\PD_{H_1,H_2;\bs_1\ldots\bs_N}^{\eta}\De=
\Th^*\,\big(\PD_{H_{12},H_{12};\bs_1\ldots\bs_N}^{\eta}\De\big)\EE
for every $\eta\!\in\!\cR_{H_{12};\bs_1\ldots\bs_N}$.\\

\noindent
If $\cR_{H_{12}}$ is finite, then $\wh{V}_{H_{12};\bs_1\ldots\bs_N}$ is 
a finite cover of a compact manifold and so its homology is finitely generated.
By the Kunneth formula for cohomology \cite[Corollary~60.7]{Mu2},
$$H^*\big(\wh{V}_{H_{12};\bs_1\ldots\bs_N}\!\times\!\wh{V}_{H_{12};\bs_1\ldots\bs_N};\Q\big)
\approx H^*\big(\wh{V}_{H_{12};\bs_1\ldots\bs_N};\Q\big)\otimes
 H^*\big(\wh{V}_{H_{12};\bs_1\ldots\bs_N};\Q\big).$$
Thus, $\PD_{H_{12},H_{12};\bs_1\ldots\bs_N}^{\eta}$ admits a Kunneth decomposition 
in this case.
By~\eref{PDpullback_e}, so does the class $\PD_{H_1,H_2;\bs_1\ldots\bs_N}^{\eta}\De$.
In light of \cite[Lemma~5.2]{GWrelIP}, 
the same reasoning applies if $V$ is connected and $H_*(\wh{V}_{H_{12}};\Q)$ is finitely generated.\\

\noindent
The identity \eref{PDpullback_e} holds with $\Psi$ and $\eta$ replaced by 
$\wt\Psi$ and $\wt\eta$ if the relevant collections
$$\{\ga_{1;j_1}\},\{\ga_{2;j_2}\},\{\ga_{12;j}\}\!\subset\!H_1(V;\Z)$$
are compatible, i.e.
$$\big\{[\ga_{1;j_1}]_{H_{12}}\big\}, \big\{[\ga_{2;j_2}]_{H_{12}}\big\}
=\big\{[\ga_{12;j}]_{H_{12}}\big\}. $$
Changing the first or second collection would change a Kunneth decomposition of 
the cohomology class
$\PD_{H_1,H_2;\bs_1\ldots\bs_N}^{\wt\eta}\De$ by pulling back 
the cohomology classes involved  by diffeomorphisms of the topological components
of the factors of the domain.
Thus, the existence of a  Kunneth decomposition for this  class
is independent of the three collections.
We can thus assume that these collections are compatible and~\eref{PDpullback_e} holds.
The reasoning in the previous paragraph then also applies to 
 $\PD_{H_1,H_2;\bs_1\ldots\bs_N}^{\wt\eta}\De$.
\end{proof}

\begin{rmk}\label{Kunn_rmk}
The statement of Lemma~\ref{Kunn_lmm} for a connected $V$ and its proof can be adapted 
to a disconnected~$V$. 
For each $r\!=\!1,\ldots,N$, let 
$$\cR_{H;r}=q_H\big(H_1(V_r;\Z)\big)  \subset\cR_H\,;$$
these modules span~$\cR_H$.
The first factor in the definition of~$\wh{V}_{H;\bs_1\ldots\bs_N}$ in~\eref{HbsCov_e}
is finite if and only~if the submodule
$$\wc\cR_{H;\bs_1\ldots\bs_N}\equiv 
\sum_{\begin{subarray}{c}1\le r\le N\\ \ell_r\neq0\end{subarray}}\!\!\! \cR_{H;r}
\subset\cR_H$$
has finite index. 
This index is finite if $\ell_r\!\neq\!0$ whenever $H_1(V_r;\Q)\!\neq\!\{0\}$
or if $V\!=\!\{0,\i\}\!\times\!F$ for some connected~$F$ and 
$$H\!=\!H_{\De}\subset H_1(V;\Z)=H_1(F;\Z)\oplus H_1(F;\Z)$$
is the diagonal.
If this index is finite for $H\!=\!H_{12}$, the argument in the proof of Lemma~\ref{Kunn_lmm} 
concerning connected~$V$ applies via \cite[Remark~5.3]{GWrelIP}, which extends
\cite[Lemma~5.2]{GWrelIP} to this setting.
\end{rmk}

\begin{rmk}\label{Kunn_rmk2}
By \cite[Theorem~1]{DF}, $H_*(\wh{V}_{H_{12}};\Q)$ is finitely generated if
$H_*(\wh{V}_H;\Q)$ is finitely generated for some submodule $H\!\subset\!H_{12}$.
Thus, the last condition in Lemma~\ref{Kunn_lmm} could instead require that 
$H_*(\wh{V}_H;\Q)$ be finitely generated for some submodule $H\!\subset\!H_{12}$.
The approach in the proof applies directly, without use of  \cite[Theorem~1]{DF},
if in addition $H$ contains either~$H_1$ or~$H_2$.
\end{rmk}

\noindent
The next example indicates that the cohomology classes~\eref{GenPD_e} of 
the diagonal components~\eref{whVdiag_e} generally do not admit a Kunneth decomposition
as in~\eref{KunnDecomp_e2} if ($\cR_{H_{12}}$ is infinite and) 
the homology of $\wh{V}_{H_{12}}$  is not finitely generated.
By \cite[Corollary]{Taubes},
the manifold~$V$ appearing in Example~\ref{noKunn_eg} is not symplectic.
We use this particular~$V$ for the sake of simplicity.
The complex blowup of $\T^4$ at a point could be used instead, but the notation would become
a bit more involved.

\begin{eg}\label{noKunn_eg}
The maximal abelian cover~$\wh{V}$ of $V\!\equiv\!(S^1\!\times\!S^3)\#\P^2$
is $\R\!\times\!S^3$ with copies of~$\P^2$ connected at $(i,\pt)\!\in\!\R\!\times\!S^3$, 
with $i\!\in\!\Z$.
If $\P^1_i$ is a $\P^1$ in the $i$-th $\P^2$, 
$\{\P^1_i\!:i\!\in\!\Z\}$ is a basis for  $H_2(\wh{V};\Q)$.
The topological components of $\wh{V}_{0,0;(1)}\!\equiv\!\wh{V}\!\times_V\!\wh{V}$ are described~by
$$\wh{V}_{0,0;(1)}^{\ga}=\big\{(\ga\!\cdot\!\wh{x},\wh{x})\!:\,\wh{x}\!\in\!\wh{V}\big\},
\qquad \ga\!\in\!\Z.$$
For example, $\wh{V}_{0,0;(1)}^0\!=\!\De_{\wh{V}}$ is the diagonal and
\BE{noKunn_e}\P^1_i\!\times\!\P^1_j \cdot \wh{V}_{0,0;(1)}^0 
=\begin{cases}1,&\hbox{if}~i\!=\!j;\\
0,&\hbox{if}~i\!\neq\!j.
\end{cases}\EE
If $e_1,\ldots,e_N$ is a basis for $\Q^N$, then the element
$$e_1\!\otimes\!e_1+\ldots+e_N\!\otimes\!e_N\in \Q^N\!\otimes_{\Q}\Q^N$$
cannot be written as a sum of fewer than $N$ products $\al_j\!\otimes\!\be_j$.
Along with~\eref{noKunn_e}, this implies that the cohomology class on $\wh{V}\!\times\!\wh{V}$ 
determined by $\wh{V}_{0,0;(1)}^0$ does not admit a finite Kunneth decomposition.
\end{eg}

\noindent
Let $V$, $V_1,\ldots,V_N$, and $H_{12}$ be as in Lemma~\ref{Kunn_lmm}.
For each $r\!=\!1,\ldots,N$, denote~by 
$$ \Flux(V_r)_{H_{12}}\subset\cR_{H_{12}}$$
the image of $\Flux(V_r)$ under the homomorphism
$$H_1(V_r;\Z)\lra H_1(V;\Z)\lra \frac{H_1(V;\Z)}{H_{12}}=\cR_{H_{12}}\,.$$

\begin{lmm}\label{DiagFlux_lmm}
Suppose $V$, $V_1,\ldots,V_N$, $H_1,H_2,H_{12}$, and $\bs_r$ are as in Lemma~\ref{Kunn_lmm}.
If 
$$\wt\eta_1,\wt\eta_2\!\in\!\cR_{H_{12}} \quad\hbox{and}\quad
\wt\eta_1\!-\!\wt\eta_2 \in \bigoplus_{r=1}^N\gcd(\bs_r)\Flux(V_r)_{H_{12}},$$
then
\BE{DiagFlux_e}\PD_{H_1,H_2;\bs_1\ldots\bs_N}^{\wt\eta_1}\De=
\PD_{H_1,H_2;\bs_1\ldots\bs_N}^{\wt\eta_2}\De\in 
H^*\big(\wh{V}_{H_1;\bs_1\ldots\bs_N}\!\times\!\wh{V}_{H_2;\bs_1\ldots\bs_N};\Q\big)\,.\EE
\end{lmm}

\begin{proof}
Let 
$$\ga\equiv (\ga_{r;i})_{i\le\ell_r,r\le N}\in \bigoplus_{r=1}^N \Flux(V_r)^{\oplus\ell_r}$$
be such that 
\BE{DiagFlux_e3}\wt\eta_1-\wt\eta_2=\big[\Phi_{V;\bs_1\ldots\bs_N}(\ga)\big]_{H_{12}}\in \cR_{H_{12}}\,.\EE
For each $r\!=\!1,\ldots,N$ and $i\!=\!1,\ldots,\ell_r$, 
let $\Psi_{r;i;t}\!:V_r\!\lra\!V_r$ be a loop of diffeomorphisms generating~$\ga_{r;i}$
such that $\Psi_{r;i;0}\!=\!\id$.
These loops lift to paths of diffeomorphisms
\begin{alignat*}{3}
\wh\Psi_{r;i;t}\!: \wh{V}_r&\lra \wh{V}_r, &\quad t&\in[0,1], 
&\quad &\wh\Psi_{r;i;0}=\id_{\wh{V}_r},~~\wh\Psi_{r;i;1}(\wh{x}_{r;i})=\ga_{r;i}\cdot\wh{x}_{r;i},\\
\wt\Psi_t\!: \wh{V}_{H_1;\bs_1\ldots\bs_N}&\lra \wh{V}_{H_1;\bs_1\ldots\bs_N}, &\quad t&\in[0,1],  
&\quad &\wt\Psi_t\big(\big[(\wh{x}_{r;i})_{i\le\ell_r,r\le N}\big]_{H_{12}}\big)=
\big[\big(\wh\Psi_{r;i;t}(\wh{x}_{r;i})\big)_{i\le\ell_r,r\le N}\big]_{H_{12}}.
\end{alignat*}
In particular, $\wt\Psi_1\!=\!\Th_{\Phi_{V;\bs_1\ldots\bs_N}(\ga)}$.
By~\eref{PsiThDeg_e} and~\eref{DiagFlux_e3},
$$\wh{V}_{H_1,H_2;\bs_1\ldots\bs_N}^{\wt\eta_1}
=\big\{\wt\Psi_1\!\times\!\id_{\wh{V}_{H_2;\bs_1\ldots\bs_N}}\big\}
\big(\wh{V}_{H_1,H_2;\bs_1\ldots\bs_N}^{\wt\eta_2}\big).$$ 
Thus, the smooth proper map 
$$[0,1]\!\times\!\wh{V}_{H_1,H_2;\bs_1\ldots\bs_N}^{\wt\eta_2}
\lra \wh{V}_{H_1;\bs_1\ldots\bs_N}\!\times\!\wh{V}_{H_2;\bs_1\ldots\bs_N}, \qquad
(t,\wt{x}_1,\wt{x}_2)\lra \big(\wt\Psi_t(\wt{x}_1),\wt{x}_2\big),$$
is a cobordism between $\wh{V}_{H_1,H_2;\bs_1\ldots\bs_N}^{\wt\eta_1}$
and $\wh{V}_{H_1,H_2;\bs_1\ldots\bs_N}^{\wt\eta_2}$.
Its intersection with a smooth map~$h$ as in~\eref{hZdfn_e} is a compact cobordism
between the intersections of~$h$ with $\wh{V}_{H_1,H_2;\bs_1\ldots\bs_N}^{\wt\eta_1}$
and with $\wh{V}_{H_1,H_2;\bs_1\ldots\bs_N}^{\wt\eta_2}$.
Thus, the intersection homomorphisms on the homology induced by the two diagonal components
are the same; this establishes~\eref{DiagFlux_e}. 
\end{proof}

\section{The refined invariance property}
\label{SympSum_sec}

\noindent
We now provide the details needed to refine the usual symplectic sum formula,
as suggested in \cite[Sections~3,10]{IPsum} and outlined in Section~\ref{SympSum_subs0}.
Once the lifts~\eref{evXVlift_e} of~\eref{evdfn_e2} are chosen systematically 
as in Proposition~\ref{RimToriAct_prp}, 
the fiber products~\eref{cHXYV_e0} of the coverings~\eref{IPcov_e} for $(X,V)$ and~$(Y,V)$ 
over the diagonal in~$V_{\bs}\!\times\!V_{\bs}$
can be split into unions of topological components~\eref{cHXYVC_e}, 
as suggested by \cite[(3.10)]{IPsum}.
We define the crucial refined degree-gluing map~\eref{IPdegmap_e} in Section~\ref{GlDeg_subs}
as a special case of the map~\eref{wtPsi_e} for arbitrary abelian covers.
It is used in Section~\ref{XYVdiagprop_subs} to define the diagonal components~\eref{cHXYVC_e}.
Corollary~\ref{Kunn_crl} describes cases when the Poincare duals of
these components split as in~\eref{KunnDecomp_e2}.
The approach of \cite{IPrel,IPsum} can then be used to distinguish
the GW-invariants of $X\!\#_V\!Y$ in degrees differing by elements of~$\cR_{X,Y}^V$ 
in terms of the IP-counts of~$(X,V)$ and~$(Y,V)$.
Otherwise, they can be distinguished only in terms of the IP-counts of 
the singular fiber $X\!\cup_V\!Y$.
The assertions made in Sections~\ref{VanishAppl_subs} and~\ref{FluxAppl_subs} 
are established in Section~\ref{SympSumPf_subs}.

\subsection{The refined gluing degree map}
\label{GlDeg_subs}

\noindent
We begin this section by defining the \sf{refined gluing degree map}~\eref{IPdegmap_e}
as a special case of the continuous map~\eref{wtPsi_e} on the diagonal fiber product of 
two abelian covers of the divisor~$V$.
We then show that its composition with the lifted evaluation morphisms~\eref{EvLift_e}
gives the degree of the glued map, provided the coset representatives in
the constructions of~\eref{wtPsi_e} and~\eref{EvLift_e} are chosen in the same way;
see Figure~\ref{RimToriAct_fig} and Proposition~\ref{RimToriAct_prp2}.\\

\noindent
Throughout this section, $X$ and $Y$ denote compact oriented manifolds, $V\!\subset\!X,Y$
is a compact oriented submanifold of codimension~$\fc$, and
\hbox{$\vph\!:S_XV\!\lra\!S_YV$} is an orientation-reversing
diffeomorphism commuting with the projections to~$V$.
With
$$q_V\!:\,S_XV\!=\!S_YV\lra V$$
denoting the bundle projection map, define
\BE{Hfcdfn_e}H_{\fc}(SV;\Z)_{X,Y}=
\big\{A_{SV}\!\in\!H_{\fc}(SV;\Z)\!:\,q_{V*}(A_{SV})\!\in\!\ker\io^X_{V*}\cap\ker\io^Y_{V*}\big\}.\EE
Let
\BE{HfcXTdfn_e}\begin{split}
H_{\fc}(X;\Z)\!\times_V\!H_{\fc}(Y;\Z)&=
\big\{(A_X,A_Y)\!\in\!H_{\fc}(X;\Z)\!\times\!H_{\fc}(Y;\Z)\!:\\
&\hspace{1in}
A_X\!\cdot_X\!V_r=A_Y\!\cdot_Y\!V_r~~\forall\,r\!=\!1,\ldots,N\big\},
\end{split}\EE
where $\cdot_X$ and $\cdot_Y$ are the homology intersection pairings in~$X$ and~$Y$,
respectively.\\

\noindent
Let $X\!\#_{\vph}\!Y$ be the manifold
obtained by gluing the complements of tubular neighborhoods of~$V$ in~$X$ and~$Y$
by~$\vph$ along their common boundary.
We denote~by 
$$q_{\vph}\!:X\!\#_{\vph}\!Y\lra X\!\cup_V\!Y$$
a continuous map which restricts to the identity outside of a tubular neighborhood
of $S_XV\!=_{\vph}\!S_YV$, is a diffeomorphism on the complement of $q_{\vph}^{-1}(V)$,
and restricts to the bundle projection $S_XV\!\lra\!V$.
Let 
\BE{cRXYVdfn_e}\begin{split}
\cR_{X,Y}^V\equiv\big\{\io^{X\#_{\vph}Y}_{X-V*}(A_{X-V})\!+\!\io^{X\#_{\vph}Y}_{Y-V*}(A_{Y-V})\!:~
&(A_{X-V},A_{Y-V})\in \cR_X^V\!\oplus\!\cR_Y^V\big\}.
\end{split}\EE 
By \cite[Lemma~4.1]{GWrelIP}, this definition of $\cR_{X,Y}^V$ agrees with 
\eref{cRXYVprop_e} in the $\fc\!=\!2$ case.\\

\noindent
Define 
$$\De_{X,Y}^V\!: 
H_1(V;\Z)_X \oplus H_1(V;\Z)_Y \lra\frac{H_1(V;\Z)}{H_X^V\!+\!H_Y^V}\,,\quad
\big([\ga_X]_{H_X^V},[\ga_Y]_{H_Y^V}\big)\lra \big[\ga_X\!-\!\ga_Y\big]_{H_X^V+H_Y^V}.$$
Denote by $\ov{H}_{X,Y}^V$ the image of the composition
$$H_{\fc}(SV;\Z)_{X,Y}
\stackrel{(\io_{SV*}^{X-V},-\io_{SV*}^{Y-V})}{\xra{1.8}}
\cR_X^V\!\oplus\!\cR_Y^V \stackrel{\approx}{\lra}
H_1(V;\Z)_X \oplus H_1(V;\Z)_Y \stackrel{\De_{X,Y}^V}{\xra{.8}}\frac{H_1(V;\Z)}{H_X^V\!+\!H_Y^V},$$
with the second arrow above given by the isomorphisms in~\eref{RTisom_e}.
Let $H_{X,Y}^V\!\subset\!H_1(V;\Z)$ be the preimage of $\ov{H}_{X,Y}^V$
under the quotient projection
$$H_1(V;\Z)\lra \frac{H_1(V;\Z)}{H_X^V\!+\!H_Y^V}\,.$$
By \cite[Corollary~4.4(1)]{GWrelIP}, there is a commutative diagram
\BE{cRH1diag_e}\begin{split}
\xymatrix{ \frac{H_1(V;\Z)}{H_X^V}\oplus  \frac{H_1(V;\Z)}{H_Y^V}
\ar[d]|{\io_{S_XV*}^{X-V}\!\circ\De_X^V\oplus\io_{S_YV*}^{Y-V}\!\circ\De_Y^V}
\ar[rrrr]^{\ov\De_{X,Y}^V}  &&&&    \frac{H_1(V;\Z)}{H_{X,Y}^V}  \ar[d]^{\fR_{X,Y}^V}_{\approx} \\
 \cR_X^V\oplus\cR_Y^V   \ar[rrrr]^{\io_{X-V*}^{X\#_{\vph}Y}+\io_{Y-V*}^{X\#_{\vph}Y}} &&&& 
  \cR_{X,Y}^V}
\end{split}\EE
of homomorphisms, with both vertical arrows being isomorphisms.\\
 
\noindent
Let $V_1,\ldots,V_N$ be the topological components of~$V$, 
$\bs_1\!\in\!\Z_{\pm}^{\ell_1},\ldots,\bs_N\!\in\!\Z_{\pm}^{\ell_N}$,  and
$$\wh{V}_{X,Y;\bs_1\ldots\bs_N}' = \wh{V}_{H_X^V,H_Y^V;\bs_1\ldots\bs_N}'\,,\quad
\wh{V}_{X,Y;\bs_1\ldots\bs_N}= \wh{V}_{H_X^V,H_Y^V;\bs_1\ldots\bs_N};$$
see~\eref{whVdiag_e0}.
Choose collections 
\BE{gaXY_e}\big\{\ga_{X;j}\big\},\big\{\ga_{Y;j}\big\}\subset H_1(V;\Z)\EE
of coset representatives for the elements of $\cR_X^V/\cR_{X;\bs_1\ldots\bs_N}'^{\,V}$
and $\cR_Y^V/\cR_{Y;\bs_1\ldots\bs_N}'^{\,V}$.
As described above~\eref{wtPsi_e}, these collections determine a smooth~map
$$\wt\Psi_{X,Y}^V\!: 
\wh{V}_{X,Y;\bs_1\ldots\bs_N}\stackrel{\wt\Psi_{H_X^V,H_Y^V}^{H_{X,Y}^V}}{\xra{1}} \cR_{H_{X,Y}^V}
\underset{\approx}{\stackrel{\fR_{X,Y}^V}{\xra{.8}}} \cR_{X,Y}^V  
\subset H_{\fc}(X\!\#_{\vph}\!Y;\Z),$$
with $\fR_{X,Y}^V$ as in~\eref{cRH1diag_e}.\\

\noindent
Let $\Si_X,\Si_Y$ be compact oriented $\fc$-dimensional manifolds,
$k_X,k_Y\!\in\!\Z^{\ge0}$, and $(A_X,A_Y)$ be an element of 
$H_{\fc}(X;\Z)\!\times_V\!H_{\fc}(Y;\Z)$.
Denote~by
$$\wt\ev_X^V\!:\fX_{\Si_X,k_X;\bs_1\ldots\bs_N}^{V_1,\ldots,V_N}(X,A_X)\lra \wh{V}_{X;\bs_1\ldots\bs_N}
\quad\hbox{and}\quad
\wt\ev_Y^V\!:\fX_{\Si_Y,k_Y;\bs_1\ldots\bs_N}^{V_1,\ldots,V_N}(Y,A_Y) \lra 
\wh{V}_{Y;\bs_1\ldots\bs_N}$$
the lifted relative evaluation morphisms of Proposition~\ref{RimToriAct_prp} for $(X,V)$ and~$(Y,V)$
compatible with the coset representatives~\eref{gaXY_e}.
Let
\BE{Xdiagdfn_e}\begin{split}
\fX_{(\Si_X,\Si_Y),(k_X,k_Y);\bs_1\ldots\bs_N}^{V_1,\ldots,V_N}\big((X,Y),(A_X,A_Y)\big)
&=\big\{\wt\ev_X^V\!\times\!\wt\ev_Y^V\big\}^{-1}
(\wh{V}_{X,Y;\bs_1\ldots\bs_N})\\
&=\big\{\ev_X^V\!\times\!\ev_Y^V\big\}^{-1}
(\De_{\bs_1\ldots\bs_N}^V)\,.
\end{split}\EE
Each element $(\bff_X,\bff_Y)$ of this space gives rise to a marked map
\BE{bffXY_e}\bff_X\!\#_{\vph}\bff_Y\in  \bigsqcup_{A_{\#}\in A_X\#_{\vph}A_Y}\hspace{-.3in}
\fX_{\Si_X\#\Si_Y,k_X+k_Y}(X\!\#_{\vph}\!Y,A_{\#});\EE
see \cite[Section~2.2]{GWrelIP}.\\

\noindent
Fix a base  point 
$x_{\bs_1\ldots\bs_N}\!\in\!V_{\bs_1\ldots\bs_N}$.
Suppose 
$$\fX_{\Si_X,k_X;\bs_1\ldots\bs_N}^{V_1,\ldots,V_N}(X,A_X),
\fX_{\Si_Y,k_Y;\bs_1\ldots\bs_N}^{V_1,\ldots,V_N}(Y,A_Y)\neq\eset.$$
Choose
\BE{bffXbffY_e}\begin{split}
\bff_X^{\bu}&\equiv (z_{X;1}^{\bu},\ldots,z_{X;k_X+\ell_1+\ldots+\ell_N}^{\bu},f_X^{\bu})\in 
\fX_{\Si_X,k_X;\bs_1\ldots\bs_N}^{V_1,\ldots,V_N}(X,A_X),\\
\bff_Y^{\bu}&\equiv (z_{Y;1}^{\bu},\ldots,z_{Y;k_Y+\ell_1+\ldots+\ell_N}^{\bu},f_Y^{\bu})\in 
\fX_{\Si_Y,k_Y;\bs_1\ldots\bs_N}^{V_1,\ldots,V_N}(Y,A_Y)
\end{split}\EE
such that 
$$\ev_X^V\big(\bff_X^{\bu}),\ev_Y^V\big(\bff_Y^{\bu})=x_{\bs_1\ldots\bs_N}\in V_{\bs_1\ldots\bs_N}\,.$$
These two marked maps correspond to the initially chosen base points, denoted by~$\bff_0$, 
in the proof of \cite[Theorem~6.5]{GWrelIP}. 
Let
\BE{AXYdfn_e}
A_{X,Y}=\big[f_X^{\bu}\!\#_{\vph}\!f_Y^{\bu}\big]\in H_{\fc}(X\!\#_{\vph}\!Y;\Z),\EE
and define
\BE{gAdfn_e}\begin{split}
&g_{A_X,A_Y}\!:
\wh{V}_{X,Y;\bs_1\ldots\bs_N}\lra H_{\fc}(X\!\#_{\vph}\!Y;\Z) \qquad\hbox{by}\\ 
&g_{A_X,A_Y}(\wt{x},\wt{y})=A_{X,Y}-
\wt\Psi_{X,Y}^V\big(\wt\ev_X^V\big(\bff_X^{\bu}),\wt\ev_Y^V\big(\bff_Y^{\bu})\big)
+\wt\Psi_{X,Y}^V\big(\wt{x},\wt{y}\big).
\end{split}\EE
The last map is the intended refined gluing degree map of \cite[(3.10)]{IPsum}.
By Proposition~\ref{RimToriAct_prp2} below, it gives 
the degree $A_{\#}$ of each glued map~\eref{bffXY_e}.
The statement of  Proposition~\ref{RimToriAct_prp2} is illustrated in Figure~\ref{RimToriAct_fig},
where we abbreviate $\bs_1\ldots\bs_N$ as $\bs$ and ignore the absolute marked points.

\begin{rmk}\label{RimToriAct_rmk2}
The map component $f_X\!\#_{\vph}f_Y$ of $\bff_X\!\#_{\vph}\bff_Y$  depends 
on the choices made in the construction of \cite[Section~2.2]{GWrelIP}.
However, the degree (homology class) of $f_X\!\#_{\vph}f_Y$ does not depend
on these choices.
\end{rmk}

\begin{prp}\label{RimToriAct_prp2}
Suppose $X$ and $Y$ are oriented manifolds, $V\!\subset\!X,Y$ is 
a compact oriented submanifold of codimension~$\fc$ with connected components $V_1,\ldots,V_N$,
\hbox{$\vph\!:S_XV\!\lra\!S_YV$} is an orientation-reversing
diffeomorphism commuting with the projections to~$V$, and
$$(A_X,A_Y)\in H_{\fc}(X;\Z)\!\times_V\!H_{\fc}(Y;\Z), \qquad
\bs_r\in\Z_{\pm}^{\,\ell_r}~~\hbox{with}~~ r\!=\!1,\ldots,N.$$ 
Then, 
\BE{RimToriActPrp2_e}
\big[\bff_X\!\#_{\vph}\bff_Y\big]=  g_{A_X,A_Y}\big(\wt\ev_X^V(\bff_X),\wt\ev_Y^V(\bff_Y)\big)\EE
for all pairs $(\bff_X,\bff_Y)$
in $\fX_{(\Si_X,\Si_Y),(k_X,k_Y);\bs_1\ldots\bs_N}^{V_1,\ldots,V_N}\!\big((X,Y),(A_X,A_Y)\big)$.
\end{prp}

\begin{figure}
$$\xymatrix{
\fX_{(\Si_X,\Si_Y);\bs}^V\big((X,Y),(A_X,A_Y)\big)
\ar[rr]^>>>>>>>>>>>>>>>>{\wt\ev_X^V\times\wt\ev_Y^V} \ar[dd]_{\#_{\vph}} 
\ar[dr]|-{\ev_X^V\times\ev_Y^V}
&& \wh{V}_{X,Y;\bs} \ar[dd]^{g_{A_X,A_Y}} \ar[dl]|-{\pi_{X;\bs}^V\times\pi_{Y;\bs}^V}\\
& \De^V_{\bs}\\
\bigsqcup\limits_{A_{\#}\in A_X\#_{\vph}A_Y}\!\!\!\!\!\!\!\!\fX_{\Si_X\#\Si_Y}(X\!\#_{\vph}\!Y,A_{\#}) 
\ar[rr]^>>>>>>>>>>>>>>{\deg}&&  H_{\fc}(X\!\#_{\vph}\!Y;\Z)}$$
\caption{The evaluation and gluing maps of Proposition~\ref{RimToriAct_prp2}.}
\label{RimToriAct_fig}
\end{figure}

\begin{proof} 
Let $\wh{x}\!\in\!\wh{V}_{\bs_1\ldots\bs_N}$, $\ga^{\bu}\!\in\!H_1(V_{\bs_1\ldots\bs_N};\Z)$,
and $\ga_X^{\bu},\ga_Y^{\bu}\!\in\!H_1(V;\Z)$ be elements
of the collections in~\eref{gaXY_e} such~that 
\BE{RimToriAct_e1}\wt\ev_X^V(\bff_X^{\bu})=\big([\ga_X^{\bu}]_{X;\bs_1\ldots\bs_N},
[\ga^{\bu}\!\cdot\!\wh{x}]_X\big)
\quad\hbox{and}\quad 
\wt\ev_Y^V(\bff_Y^{\bu})=\big([\ga_Y^{\bu}]_{Y;\bs_1\ldots\bs_N},[\wh{x}]_X\big).\EE
Thus,
\BE{RimToriAct_e2}
\wt\Psi_{X,Y}^V\big(\wt\ev_X^V\big(\bff_X^{\bu}),\wt\ev_Y^V\big(\bff_Y^{\bu})\big)
=\fR_{X,Y}^V\big(\big[\Phi_{V;\bs_1\ldots\bs_N}(\ga^{\bu})
\!+\!\ga_X^{\bu}\!-\!\ga_Y^{\bu}\big]_{H_{X,Y}^V} \big)\,.\EE\\

\noindent
Since the two sides of~\eref{RimToriActPrp2_e} take discrete values and are
continuous in $(\bff_X,\bff_Y)$, it is sufficient to verify~\eref{RimToriActPrp2_e} 
under the assumption~that 
$$\ev_X^V\big(\bff_X),\ev_Y^V\big(\bff_Y)=x_{\bs_1\ldots\bs_N}\in V_{\bs_1\ldots\bs_N}\,.$$
Let  $\ga,\ga'\!\in\!H_1(V_{\bs_1\ldots\bs_N};\Z)$,
and $\ga_X,\ga_Y\!\in\!H_1(V;\Z)$ be elements
of the collections in~\eref{gaXY_e} such~that 
\BE{RimToriAct_e5c}\wt\ev_X^V(\bff_X)=\big([\ga_X]_{X;\bs_1\ldots\bs_N},
[\ga\ga'\!\cdot\!\wh{x}]_X\big)
\quad\hbox{and}\quad 
\wt\ev_Y^V(\bff_Y)=\big([\ga_Y]_{Y;\bs_1\ldots\bs_N},[\ga'\!\cdot\!\wh{x}]_X\big).\EE
Thus,
\BE{RimToriAct_e6}
\wt\Psi_{X,Y}^V\big(\wt\ev_X^V\big(\bff_X),\wt\ev_Y^V\big(\bff_Y)\big)
=\fR_{X,Y}^V\big(\big[\Phi_{V;\bs_1\ldots\bs_N}(\ga)
\!+\!\ga_X\!-\!\ga_Y\big]_{H_{X,Y}^V} \big)\,.\EE\\

\noindent
By~\eref{RimToriAct_e1}, \eref{RimToriAct_e5c}, and~\eref{RimToriAct_e},
\begin{equation*}\begin{split}
\big[f_X\!\#(-f_X^{\bu})\big]&=
\io_{S_XV*}^{X-V}\big(\De_X^V
\big(\Phi_{V;\bs_1\ldots\bs_N}(\ga\!+\!\ga'\!-\!\ga^{\bu})\!+\!\ga_X\!-\!\ga_X^{\bu}\big)\big)  
\in H_{\fc}(X\!-\!V;\Z),\\
\big[f_Y\!\#(-f_Y^{\bu})\big]&=
\io_{S_YV*}^{Y-V}\big(\De_Y^V
\big(\Phi_{V;\bs_1\ldots\bs_N}(\ga')\!+\!\ga_Y\!-\!\ga_Y^{\bu}\big)\big)  \in H_{\fc}(Y\!-\!V;\Z).
\end{split}\end{equation*}
Thus, 
\BE{RimToriAct_e5}\begin{split} 
&\big[f_X\!\#_{\vph}\!f_Y\big]-\big[f_X^{\bu}\!\#_{\vph}\!f_Y^{\bu}\big]
=\io_{X-V*}^{X\#_{\vph}Y}\big(\big[f_X\!\#(-f_X^{\bu})\big]\big)+
\io_{Y-V*}^{X\#_{\vph}Y}\big(\big[f_Y\!\#\!(-f_Y^{\bu})\!\big]\big)\\
&\qquad\qquad =\io_{S_XV*}^{X\#_{\vph}Y}\big(\De_X^V\big(
\Phi_{V;\bs_1\ldots\bs_N}(\ga\!-\!\ga^{\bu})\!+\!\ga_X\!-\!\ga_Y
-\ga_X^{\bu}\!+\!\ga_Y^{\bu}\big)\big).
\end{split}\EE
From~\eref{AXYdfn_e} and~\eref{RimToriAct_e5},  
we find~that 
\begin{equation*}\begin{split}
\big[f_X\!\#_{\vph}\!f_Y\big] =A_{X,Y}
&-\fR_{X,Y}^V\big(\big[\Phi_{V;\bs_1\ldots\bs_N}(\ga^{\bu})
\!+\!\ga_X^{\bu}\!-\!\ga_Y^{\bu}\big]_{H_{X,Y}^V} \big)\\
&+\fR_{X,Y}^V\big(\big[\Phi_{V;\bs_1\ldots\bs_N}(\ga)
\!+\!\ga_X\!-\!\ga_Y\big]_{H_{X,Y}^V} \big).
\end{split}\end{equation*}
Comparing this with~\eref{gAdfn_e}, \eref{RimToriAct_e2} and~\eref{RimToriAct_e6},
we obtain the claim. 
\end{proof}

\subsection{Diagonal components for rim tori covers}
\label{XYVdiagprop_subs}

\noindent
We now use the refined gluing degree map~\eref{gAdfn_e} to split fiber products
of rim tori covers into diagonal components and describe some of their properties.
This completes the details  needed to refine the usual symplectic sum formula,
as suggested in \cite[Sections~3,10]{IPsum} and outlined in Section~\ref{SympSum_subs0}. 
We include three examples of applying the refined invariance property for GW-invariants
in simple cases when the diagonal splits and its Kunneth decomposition can be easily determined.\\

\noindent
Similarly to~\eref{whVdiag_e}, we define
\BE{cHXYVsplit_e}
\wh{V}_{X,Y;\bs_1\ldots\bs_N}^A \equiv g_{A_X,A_Y}^{-1}(A)
=\wh{V}_{H_X^V,H_Y^V;\bs_1\ldots\bs_N}^{A-A_{X,Y}} 
\qquad \forall~A\in A_X\!\#_{\vph}\!A_Y\subset H_{\fc}(X\!\#_{\vph}\!Y;\Z) .\EE
Similarly to~\eref{GenPD_e}, let 
\BE{cHclass_e}\PD_{X,Y;\bs_1\ldots\bs_N}^{V,A}\De
\equiv \PD_{H_X^V,H_Y^V;\bs_1\ldots\bs_N}^{A-A_{X,Y}}\De
\in H^*\big(\wh{V}_{X;\bs_1\ldots\bs_N}\!\times\!\wh{V}_{Y;\bs_1\ldots\bs_N};\Q\big)\EE
denote the cohomology class determined by the closed submanifold~\eref{cHXYVsplit_e}.
Lemma~\ref{Kunn_lmm} gives the following description of some cases when this class
admits a finite Kunneth decomposition. 

\begin{crl}\label{Kunn_crl}
Suppose $X$, $Y$, $V$, $\vph$, 
$(A_X,A_Y)$, and $\bs_r\!\in\!\Z_{\pm}^{\,\ell_r}$ are as in Proposition~\ref{RimToriAct_prp2}
and $A\!\in\!A_X\!\#_{\vph}\!A_Y$.
The class~\eref{cHclass_e} admits a finite Kunneth decomposition as
in~\eref{KunnDecomp_e2} if either $\cR_{X,Y}^V$ is finite or
$V$ is connected and $H_*(\wh{V}_{H_{X,Y}^V};\Q)$ is finitely generated.
\end{crl}

\begin{eg}\label{ESsum_eg}
We take $X,Y\!=\!\wh\P^2_9$ to be the blowup of $\P^2$ at 9 points 
and $V\!=\!F$ to be a smooth fiber of the fibration $\wh\P^2_9\!\lra\!\P^1$;
see the beginning of Section~\ref{RES_sec}.
By \cite[Examples~3.5,4.6]{GWrelIP}, 
$$H_X^V,H_Y^V,H_{X,Y}^V=\{0\}\subset H_1(V;\Z)\approx\Z^2\,.$$
Let $|A|_V\!\in\!\Z$ be as in~\eref{AVrdfn_e}.
By~\eref{ESsum_e}, 
\BE{ESsum_e5}\PD_{X,Y;\bs}^{V,A}\De=0
\in H^*\big(\wh{V}_{X;\bs}\!\times\!\wh{V}_{Y;\bs};\Q\big)\EE
for all $\bs\!\in\!\Z^{\ell}$ with $\ell\!>\!0$ and $A\!\in\!H_2(X\!\#_{\vph}\!Y;\Z)$
with $|A|_V\!=\!|\bs|$.
Since the contribution to the GW-invariant of $X\!\#_{\vph}\!Y$ in a degree
$A\!\in\!A_X\!\#_{\vph}\!A_Y$ from its splitting as $A_X\!\#_{\vph}\!A_Y$ 
is obtained by pulling back the classes~\eref{ESsum_e5} with $|\bs|\!=\!|A|_V$ 
by $\wt\ev_X^V\!\times\!\wt\ev_Y^V$, this contribution vanishes if $|A|_V\!\neq\!0$.
If $C\!\subset\!\wh\P^2_9$ is a holomorphic curve disjoint from~$V$, its projection to~$\P^1$
misses a point and thus $C$ is a union of fibers.
Therefore, the fiber class of $\wh\P^2_9$ and its multiples are 
the only classes that have zero intersection with~$V$ and a priori may have nonzero GW-invariants.
By~\eref{ESsum_e5}, all GW-invariants of $X\!\#_{\vph}\!Y$ in non-fiber classes vanish.
This is a direct illustration of the vanishing statement of Theorem~\ref{AbsGW_thm}.
For the standard identification~$\vph$, $X\!\#_{\vph}\!Y\!=\!\bK_3$.
By \cite[Theorem~2.4]{Junho}, $\bK_3$ admits an almost complex structure~$J$
with no $J$-holomorphic curves.
Thus, \eref{ESsum_e5}~is consistent with the vanishing of all GW-invariants of~$\bK_3$.
\end{eg}

\begin{eg}\label{S2T2sum_eg}
We take $X,Y\!=\!\P^1\!\times\!\T^2$ and $V\!=\!\{0,\i\}\!\times\!\T^2$.
By \cite[Examples~3.6,4.7]{GWrelIP}, 
$$H_X^V,H_Y^V,H_{X,Y}^V=H_{\De}\subset H_1(V;\Z)=H_1(\T^2;\Z)\oplus H_1(\T^2;\Z)
\approx\Z^2\oplus\Z^2\,,$$
where $H_{\De}$ is the diagonal subgroup.
By~\eref{S2T2sum_e},
\BE{S2T2sum_e5}\PD_{X,Y;\bs_1\bs_2}^{V,A}\De=0
\in H^*\big(\wh{V}_{X;\bs_1\bs_2}\!\times\!\wh{V}_{Y;\bs_1\bs_2};\Q\big)\EE
for all $\bs_1\!\in\!\Z^{\ell_1},\bs_2\!\in\!\Z^{\ell_2}$ with $\ell_1\!+\!\ell_2\!>\!0$
and $A\!\in\!H_2(X\!\#_{\vph}\!Y;\Z)$ with $|A|_V\!=\!|\bs_1|\!+\!|\bs_2|$. 
Similarly to Example~\ref{ESsum_eg}, this implies that  
all GW-invariants of $X\!\#_{\vph}\!Y$ in non-fiber classes vanish
and provides another direct illustration of the vanishing statement of Theorem~\ref{AbsGW_thm}.
For the standard identification~$\vph$, $X\!\#_{\vph}\!Y\!=\!\T^2\!\times\!\T^2$.
By \cite[Theorem~2.4]{Junho}, $\T^2\!\times\!\T^2$ admits an almost complex structure~$J$
with no $J$-holomorphic curves.
Thus, \eref{S2T2sum_e5}~is consistent with the vanishing of all GW-invariants of~$\T^2\!\times\!\T^2$.
\end{eg}

\begin{eg}\label{ESsum_eg2}
We now take $X\!=\!\wh\P^2_9$, $Y\!=\!\P^1\!\times\!\T^2$, and $V\!=\!F\!\subset\!X,Y$ 
to be a smooth fiber.
By \cite[Examples~3.5,3.6]{GWrelIP},
$$H_X^V=\{0\},~H_Y^V,H_{X,Y}^V=H_1(V;\Z)\approx\Z^2.$$
Since $\cR_{X,Y}^V\!=\!H_1(V;\Z)/H_{X,Y}^V$, there are no rim tori in $X\!\#_{\vph}\!Y$
in this case.
By~\eref{ESsum_e2b},
\BE{ESsum_e9}  \PD_{X,Y;(1)}^{V,A}\De= 1\!\times\!\PD_{\T^2}(\pt)
\in H^*\big(\wh{V}_{X;(1)}\!\times\!\wh{V}_{Y;(1)};\Q\big)
=  H^*\big(\C\!\times\!\T^2;\Q\big) \EE
for all $A\!\in\!H_2(X\!\#_{\vph}\!Y;\Z)$ with $|A|_V\!=\!1$. 
Let $\fs_1,\ldots,\fs_9,\ff\!\in\!H_2(X;\Z)$ denote the homology classes
of the 9~sections corresponding to the exceptional divisors and of the fiber class
and $\fs,\ff\!\in\!H_2(Y;\Z)$ denote the homology classes of the section class
and of the fiber class.
Let 
$$A_{i;d}=(\fs_i\!+\!d\ff)\!\#_{\vph}\!\fs\in H_2(X\!\#_{\vph}\!Y;\Z)\,.$$
The only decompositions $A_{i;d}\!=\!A_X\!\#_{\vph}\!A_Y$ into classes $A_X,A_Y$
with possibly nonzero GW-invariants are of the~form
$$ A_{i;d}= (\fs_i\!+\!d_1\ff)\!\#_{\vph}\!(\fs\!+\!d_2\ff)
\qquad\hbox{with}\quad d_1,d_2\!\in\!\Z^{\ge0},~d_1\!+\!d_2\!=\!d.$$
Thus, the refined invariance property for GW-invariants of~\cite{IPsum},
the decomposition~\eref{ESsum_e9}, and dimensional considerations give
\BE{ESsum_e15}
\GW_{g,A_{i;d}}^{X\#_{\vph}Y}\big(\pt^g\big)
=\sum_{\begin{subarray}{c}d_1,d_2\in\Z^{\ge0}\\ d_1+d_2=d\end{subarray}}\!\!\!\!\!
\wt\GW_{g-1,\fs_i+d_1\ff;(1)}^{\wh\P^2_9,F}\big(\pt^{g-1};1\big)
\wt\GW_{1,\fs+d_2\ff;(1)}^{\P^1\times\T^2,F}\big(\pt;\PD_{\T^2}(\pt)\big),\EE
where $\pt^g$ denotes $g$~point constraints (pullback of $\PD_{(X\#_{\vph}Y)^g}(\pt)$
by the evaluation map at $g$ absolute points) and 
$\wt\GW$ denotes the IP-counts with the relative constraints 
$$1\in H^0\big(\wh{V}_{X;(1)};\Q)=H^0\big(\C;\Q) \qquad\hbox{and}\qquad
\PD_{\T^2}(\pt) \in H^2\big(\wh{V}_{Y;(1)};\Q)=H^2\big(\T^2;\Q).$$
Since the first class above is the pullback of the cohomology class~1 on~$\T^2$ 
by the covering map, \eref{ESsum_e15} reduces~to
\BE{ESsum_e15b}
\GW_{g,A_{i;d}}^{X\#_{\vph}Y}\big(\pt^g\big)
=\sum_{\begin{subarray}{c}d_1,d_2\in\Z^{\ge0}\\ d_1+d_2=d\end{subarray}}\!\!\!\!\!
\GW_{g-1,\fs_i+d_1\ff;(1)}^{\wh\P^2_9,F}\big(\pt^{g-1};1\big)
\GW_{1,\fs+d_2\ff;(1)}^{\P^1\times\T^2,F}\big(\pt;\PD_{\T^2}(\pt)\big),\EE
with the standard relative GW-invariants on the right-hand side above.
Since $\cR_{X,Y}^V\!=\!\{0\}$, the standard symplectic sum formulas of \cite{LR,Jun2} 
leave nothing to be refined in this case.
In the proof of Lemma~\ref{RES_lmm3}, \eref{ESsum_e15b} is deduced from~\eref{RESsplit_e2}
and~\eref{RelGWvan_e}.
The last statement is a consequence of~\eref{evfactor_e} and $\wh{V}_{\wh\P^2_9;(1)}\!\approx\!\C$.
The same conclusions are obtained in the proof of \cite[Lemma~15.2]{IPsum}
through a simultaneous triple induction with three separate applications of the standard
symplectic sum~formula.
\end{eg}

\noindent
We next turn to connections with the flux group defined in Section~\ref{FluxAppl_subs}.
Let $X$, $Y$,  $V$, $V_1,\ldots,V_N$, $\fc$, and~$\vph$ be as in Proposition~\ref{RimToriAct_prp2}.
For each $r\!=\!1,\ldots,N$, define
\BE{FluxXYVdfn_e}\Flux(V_r)_{X,Y}=\fR_{X,Y}^V\big(\Flux(V_r)_{H_{X,Y}^V}\big)
\subset \cR_{X,Y}^V \subset H_{\fc}(X\!\#_{\vph}Y;\Z).\EE
In particular,
$$\Flux(V_r)_{X,Y}=\io_{X-V*}^{X\#_VY}\big(\De_X^V\big(\Flux(V_r)\big)\big)
=\io_{Y-V*}^{X\#_VY}\big(\De_Y^V\big(\Flux(V_r)\big)\big).$$
Since $\fR_{X,Y}^V$ is an isomorphism, the next statement follows immediately 
from~\eref{gAdfn_e} and Lemma~\ref{DiagFlux_lmm}.

\begin{crl}\label{DiagFlux_crl}
Suppose $X$, $Y$, $V$, $\vph$, 
$(A_X,A_Y)$, and $\bs_r\!\in\!\Z_{\pm}^{\,\ell_r}$ are as in Proposition~\ref{RimToriAct_prp2}.
If 
\BE{DiagFluxcrl_e}
A_1,A_2 \in A_X\!\#_{\vph}\!A_Y \quad\hbox{and}\quad
A_1\!-\!A_2 \in \bigoplus_{r=1}^N\gcd(\bs_r)\Flux(V_r)_{X,Y}\subset H_{\fc}(X\!\#_{\vph}Y;\Z),\EE
then
$$\PD_{X,Y;\bs_1\ldots\bs_N}^{V,A_1}\De=\PD_{X,Y;\bs_1\ldots\bs_N}^{V,A_2}\De
\in 
H^*\big(\wh{V}_{X;\bs_1\ldots\bs_N}\!\times\!\wh{V}_{Y;\bs_1\ldots\bs_N};\Q\big)\,.$$
\end{crl}

\subsection{Proofs of qualitative implications}
\label{SympSumPf_subs}

\noindent
In this section, we establish the assertions made in Sections~\ref{VanishAppl_subs} and~\ref{FluxAppl_subs}.
We first obtain Theorem~\ref{AbsGW_thm} and deduce 
Corollary~\ref{AbsGW_crl} from~it.
We then state and prove Theorem~\ref{EqGWs_thm}, which includes Propositions~\ref{EqGWs_prp1}
and~\ref{EqGWs_prp2} as special cases, and conclude with Corollary~\ref{EqGWs_crl}.

\begin{proof}[{\bf{\emph{Proof of Theorem~\ref{AbsGW_thm}}}}]
Let $V_1,\ldots,V_N$ be the connected components of~$V$.
Suppose $g\!\in\!\Z^{\ge0}$, $A$ is as in~\eref{AbsGW_e2}, and $\ka_{\#}$ is
a $\Phi$-admissible input for the GW-invariants of $X\!\#_V\!Y$ such that 
the corresponding GW-invariant is nonzero,
\BE{AbsGWthm_e1} \GW_{g,A}^{X\#_VY}(\ka_{\#})\neq0.\EE
In light of~\eref{GWsumIP_e}, \eref{AbsGWthm_e1} implies
that there exist an element $(A_X,A_Y)$ of~\eref{HfcXTdfn_e} and 
relative contact vectors
\hbox{$\bs_1\!\in\Z^{\ell_1},\ldots,\bs_N\!\in\Z^{\ell_N}$} such~that 
\BE{AbsGWthm_e3} A\in A_X\!\#_V\!A_Y, \qquad
\wt\GW_{g,(A_X,A_Y);\bs_1\ldots\bs_N}^{X\cup_VY}\big(\ka_{\#};\PD_{X,Y;\bs_1\ldots\bs_N}^{V,A}\De\big)\neq0,\EE
where $\wt\GW$ is an IP-count for $X\!\cup_V\!Y$ as described in Section~\ref{SympSum_subs0}.\\

\noindent
By~\eref{AbsGWthm_e3} and~\eref{AbsGW_e2}, $\ell_r\!\neq\!0$ for some $r\!=\!1,\ldots,N$.
Reordering the components of $V_1,\ldots,V_N$, we can assume that 
$$\ell_r\neq0~~~\forall~r\!=\!1,\ldots,N' \qquad\hbox{and}\qquad 
\ell_r=0~~~\forall~r\!=\!N'\!+\!1,\ldots,N,$$
for some $N'\!=\!1,\ldots,N$.
Let $W$ denote the union of the first $N'$ components of~$V$.
By \eref{IPcov_e2} and~\eref{HbsCov_e}, 
\BE{whVsplit_e}\wh{V}_{X;\bs_1\ldots\bs_N}\!\times\!\wh{V}_{Y;\bs_1\ldots\bs_N}
=\frac{\cR_X^V}{\cR_{X;\bs_1\ldots\bs_N}'^V}\times
\frac{\cR_Y}{\cR_{Y;\bs_1\ldots\bs_N}'^V}\times
\wh{W}_{X;\bs_1\ldots\bs_{N'}}'\!\times\!\wh{W}_{Y;\bs_1\ldots\bs_{N'}}'\,.\EE
Since $H_1(V_1;\Z)$ is finitely generated, the index of $\gcd(\bs_1)H_1(V_1;\Z)$
in $H_1(V_1;\Z)$ is finite.
If $V\!\subset\!X$ is virtually connected,
the cokernel of the homomorphism~\eref{VirConn_e} with $r\!=\!1$ is finite
and so the image of $\gcd(\bs_1)H_1(V_1;\Z)$ under this homomorphism is of a finite index.
Thus, the first quotient on the right-hand side of~\eref{whVsplit_e} is finite if
 $V\!\subset\!X$ is virtually connected.\\ 

\noindent
For each element $[\ga]$ in the product of the two quotients on the right-hand side of~\eref{whVsplit_e},
let
$$\PD_{X,Y;\bs_1\ldots\bs_N}^{V,A;\ga}\De\in 
H^*\big(\big\{[\ga]\big\}\!\times\!
\wh{W}_{X;\bs_1\ldots\bs_{N'}}'\!\times\!\wh{W}_{Y;\bs_1\ldots\bs_{N'}}';\Q\big)
\subset H^*\big(\wh{V}_{X;\bs_1\ldots\bs_N}\!\times\!\wh{V}_{Y;\bs_1\ldots\bs_N};\Q\big)$$
be the cohomology class obtained by restricting 
$\PD_{X,Y;\bs_1\ldots\bs_N}^{V,A}\De$ to the topological component
$$\big\{[\ga]\big\}\!\times\!\wh{W}_{X;\bs_1\ldots\bs_{N'}}'\!\times\!\wh{W}_{Y;\bs_1\ldots\bs_{N'}}'
\subset \wh{V}_{X;\bs_1\ldots\bs_N}\!\times\!\wh{V}_{Y;\bs_1\ldots\bs_N}$$
and then extending it by zero over the remaining components.
By~\eref{AbsGWthm_e3},
\BE{AbsGWthm_e9} 
\wt\GW_{g,(A_X,A_Y);\bs_1\ldots\bs_N}^{X\cup_VY}
\big(\ka_{\#};\PD_{X,Y;\bs_1\ldots\bs_N}^{V,A;\ga^*}\De\big)\neq0\EE
for some~$\ga^*$.\\

\noindent
Suppose the homology of $\wh{V}_X$ is finitely generated.
The natural projections
$$\prod_{r=1}^N\!\wh{V}_r\lra \prod_{r=1}^{N'}\!\wh{V}_r
\qquad\hbox{and}\qquad H_X^V\lra H_X^W$$
induce covering maps
$$\wh{V}_X \lra \wh{W}_X\times \prod_{r=N'+1}^N\!\!\!\!\!\!V_r\lra \prod_{r=1}^N\!V_r\,.$$
By \cite[Theorem~1]{DF}, the $\Q$-homology of the middle space above is finitely generated;
thus, so is $H_*(\wh{W}_X;\Q)$.
From \cite[Remark~5.3]{GWrelIP}, we then conclude that $H_*(\wh{W}'_{X;\bs_1\ldots\bs_{N'}};\Q)$
is finitely generated as well.
By the Kunneth formula for cohomology \cite[Corollary~60.7]{Mu2}, this implies that 
\BE{AbsGWthm_e15} \PD_{X,Y;\bs_1\ldots\bs_N}^{V,A;\ga^*}\De
=\sum_{i=1}^m\wt\ka_{X;i}\!\otimes\!\wt\ka_{Y;i}
\in H^*(\wh{W}'_{X;\bs_1\ldots\bs_{N'}}\!\times\!\wh{W}'_{Y;\bs_1\ldots\bs_{N'}};\Q)\EE
for some $\wt\ka_{X;i}\!\in\!H^*(\wh{W}'_{X;\bs_1\ldots\bs_{N'}};\Q)$ and 
$\wt\ka_{Y;i}\!\in\!H^*(\wh{W}'_{Y;\bs_1\ldots\bs_{N'}};\Q)$.
By~\eref{AbsGWthm_e9} and~\eref{AbsGWthm_e15}, 
\BE{AbsGWthm_e17}  
\Contr_{\Ga}^{A;\ga^*}\big(\ka_{\#;X},\ka_{\#;Y}\big)
\equiv \sum_{i=1}^m \wt\GW_{\Ga}^{X,V}\!\!\big(\ka_{\#;X};\wt\ka_{X;i}\big)
\,\wt\GW_{\Ga}^{Y,V}\!\!\big(\ka_{\#;Y};\wt\ka_{Y;i}\big)\neq0\EE
for some absolute insertions $\ka_{\#;X}$ for $X$ and $\ka_{\#;Y}$ for $Y$,
with $\wt\GW_{\Ga}^{X,V}$ and $\wt\GW_{\Ga}^{Y,V}$ denoting the disconnected IP-counts
associated with the moduli spaces in~\eref{ProdEval_e}.\\

\noindent
Since $\cR_{X,Y}^V$ is infinite, $H_1(V;\Z)$ contains an infinite cyclic subgroup $G\!\approx\!\Z$
on which the quotient projection
$$H_1(V;\Z)\lra \frac{H_1(V;\Z)}{H_{X,Y}^V} \stackrel{\fR_{X,Y}^V}{\approx} \cR_{X,Y}^V$$
is injective. 
In particular, all classes
$$A\!-\!\eta\equiv A- \fR_{X,Y}^V\big([\eta]_{H_{X,Y}^V}\big)\in H_2(X\!\#_VY;\Z), 
\qquad \eta\in G,$$
are distinct.
Since the first quotient on the right-hand side in~\eref{whVsplit_e} is finite, 
$G$ contains an infinite cyclic subgroup~$G_0$ so~that the deck transformation
$$\Th_{\eta}\!:\wh{V}_{X;\bs_1\ldots\bs_N}\lra \wh{V}_{X;\bs_1\ldots\bs_N}$$
maps each topological component of $\wh{V}_{X;\bs_1\ldots\bs_N}$ to itself
whenever $\eta\!\in\!G_0$.
In light of~\eref{gAdfn_e} and~\eref{PsiThDeg_e}, this implies that 
$$\PD_{X,Y;\bs_1\ldots\bs_N}^{V,A-\eta;\ga^*}\De
=\big\{\Th_{\eta}\!\times\!\id\big\}^*
\big(\PD_{X,Y;\bs_1\ldots\bs_N}^{V,A;\ga^*}\De\big)
=\sum_{i=1}^m \big(\Th_{\eta}^*\wt\ka_{X;i}\big)\!\otimes\!\wt\ka_{Y;i}$$
and that the $(\Ga,\ga^*,\ka_{\#;X},\ka_{\#;Y})$ contribution to 
$\GW_{g,A-\eta}^{X\#_VY}(\ka_{\#})$ is given~by
\BE{AbsGWthm_e21} 
\Contr_{\Ga}^{A-\eta;\ga^*}\big(\ka_{\#;X},\ka_{\#;Y}\big)
\equiv \sum_{i=1}^m \wt\GW_{\Ga}^{X,V}\!\!\big(\ka_{\#;X};\Th_{\eta}^*\wt\ka_{X;i}\big)
\,\wt\GW_{\Ga}^{Y,V}\!\!\big(\ka_{\#;Y};\wt\ka_{Y;i}\big).\EE
Since $H^*(\wh{W}_{X;\bs_1\ldots\bs_{N'}};\Q)$ is finitely generated and
$\wt\GW_{\Ga}^{X,V}$ is linear in its (relative) inputs,
\eref{AbsGWthm_e17} implies that infinitely many of the numbers~\eref{AbsGWthm_e21}
with $\eta\!\in\!G_0$ are nonzero
(if the dimension of $H^*(\wh{W}_{X;\bs_1\ldots\bs_{N'}};\Q)$ is~$d$,
every set of~$d$ consecutive numbers~\eref{AbsGWthm_e21} contains at least one nonzero number).
However, this contradicts Gromov's Compactness, since all classes $A\!-\!\eta$ have 
the same symplectic energy.
Therefore, \eref{AbsGWthm_e1} cannot hold.
\end{proof}

\begin{rmk}\label{AbsGW_rmk}
Let $W$ be the union of the first $N'\!\le\!N$ connected components of~$V$.
The conclusion of Theorem~\ref{AbsGW_rmk} remains valid if
$$A\in H_2(X\!\#_V\!Y;\Z) - \io_{X-V*}^{X\#_VY}\big(\Eff_{\om_X}(X,W)\big)
- \io_{Y-V*}^{X\#_VY}\big(\Eff_{\om_Y}(Y,V)\big)$$
and the cokernel of the homomorphism~\eref{VirConn_e} is surjective for every $r\!=\!1,\ldots,N'$.
\end{rmk}

\begin{proof}[{\bf{\emph{Proof of Corollary~\ref{AbsGW_crl}}}}]
Let
$$A=[\T^2\!\times\!\pt]\in H_2(\T^2\!\times\!F).$$
The moduli space of  genus~1 degree~$A$ stable morphisms for a product almost complex structure
on~$\T^2\!\times\!F$ and its obstruction bundle are described~by
$$\ov\fM_1(\T^2\!\times\!F;A)\approx F \qquad\hbox{and}\qquad
\Obs\approx\cH_{\T^2}^{0,1}\!\otimes\!TF,$$
where $\cH_{\T^2}^{0,1}$ is the space of harmonic $(0,1)$-forms on~$\T^2$. 
Thus, the genus~1 degree~$A$ GW-invariant of $\T^2\!\times\!F$ is 
\BE{AbsGWcrl_e3} \GW_{1,A}^{\T^2\times F}()=\blr{TF,F}=\chi(F).\EE
On the other hand,  $\T^2\!\times\!F$ is 
the symplectic sum of $\P^1\!\times\!F$ with itself along $V\!=\!\{0,\i\}\!\times\!F$ 
with respect to the canonical isomorphism~\eref{cNpair_e} and
$$A\not\in \io_{X-V*}^{X\#_VY}\big(H_2(X\!-\!V;\Z)\big)+\io_{Y-V*}^{X\#_VY}\big(H_2(Y\!-\!V;\Z)\big),$$
where $X,Y$ are the two copies of  $\P^1\!\times\!F$.
By \cite[Example~3.6]{GWrelIP}, 
the homomorphism~\eref{VirConn_e} is surjective for $r\!=\!1,2$ and so $V\!\subset\!X$ 
is virtually connected.
By \cite[Example~6.2]{GWrelIP}, 
$$\wh{V}_X= \wh{F}\!\times\!\wh{F}/H_1(V;\Z), \qquad
\big(\wh{x}_1,\wh{x}_2\big)\sim\big(\ga\!\cdot\!\wh{x}_1,\ga\!\cdot\!\wh{x}_2\big),$$
where $\wh{F}\!\lra\!F$ is the maximal abelian cover.
By Serre's Spectral Sequence (e.g.~Theorem~9.2.1, 9.2.17, or 9.3.1 in \cite{Sp}
applied with $\Z_2$-coefficients in the last two cases) for the fiber bundle
$$\wh{F}\lra  \wh{V}_X\lra F,$$
the $\Q$-cohomology of $\wh{V}_X$ is finitely generated if this is the case for 
$\Q$-cohomology of~$\wh{F}$.
By \cite[Example~4.7]{GWrelIP},
$$\cR_{X,Y}^V\approx H_1(F;\Z);$$
thus, $\cR_{X,Y}^V$ is infinite if $H_1(F;\Q)\!\neq\!0$.
Combining these observations with Theorem~\ref{AbsGW_thm}, we conclude that
$$ \GW_{1,A}^{\T^2\times F}()=0$$
if $H_1(F;\Q)\!\neq\!0$ and  $H_*(\wh{F};\Q)$ is finitely generated.
Comparing with~\eref{AbsGWcrl_e3}, we conclude that the latter is not the case if $\chi(F)\!\neq\!0$.
\end{proof}

\noindent
We next turn to the assertions made in Section~\ref{FluxAppl_subs}.
With $X$, $Y$, $V$, and $\Phi$ as in Propositions~\ref{EqGWs_prp1} 
and~\ref{EqGWs_prp2}, let  
$$\Flux(V_r)_{X,Y}\subset \cR_{X,Y}^V\subset H_2(X\!\#_V\!Y;\Z)$$
be as in the $\fc\!=\!2$ case of~\eref{FluxXYVdfn_e}.

\begin{thm}\label{EqGWs_thm}
Let $(X,\om_X)$ and $(Y,\om_Y)$ be compact symplectic manifolds,
$V\!\subset\!X,Y$ be a common compact symplectic divisor with topological
components $V_1,\ldots,V_N$, 
and $\Phi$ be an isomorphism of complex line bundles as in~\eref{cNpair_e}.
Suppose $\cN_X{V_r}\!\approx\!V_r\!\times\!\C$ for all $r\!=\!1,\ldots,N'$
and some $N'\!\le\!N$.
If 
\BE{EqGWsthm_e}A_1,A_2 \in H_2(X\!\#_V\!Y;\Z) \qquad\hbox{and}\qquad 
A_1\!-\!A_2\in \bigoplus_{r=1}^{N'} |A_1|_{V_r}\Flux(V_r)_{X,Y},\EE
then the GW-invariants of $X\!\#_V\!Y$ of degrees~$A_1$ and~$A_2$ with $\Phi$-admissible 
inputs are the same.
\end{thm}

\begin{proof}
By~\eref{GWsumIP_e}, 
$\GW_{g,A_1}^{X\#_{\vph}Y}(\ka_{\#})$ and 
$\GW_{g,A_2}^{X\#_{\vph}Y}(\ka_{\#})$ are sums of 
the IP-counts for $X\!\cup_V\!Y$ of the~form
\BE{EqGWsthm_e1}
\wt\GW_{g,(A_X,A_Y);\bs_1\ldots\bs_N}^{X\cup_VY}\big(\ka_{\#};\PD_{X,Y;\bs_1\ldots\bs_N}^{V,A_1}\De\big)
~~\hbox{and}~~
\wt\GW_{g,(A_X,A_Y);\bs_1\ldots\bs_N}^{X\cup_VY}\big(\ka_{\#};\PD_{X,Y;\bs_1\ldots\bs_N}^{V,A_2}\De\big),\EE
respectively.
These sums are taken over all $(A_X,A_Y)$ in~\eref{HfcXTdfn_e}
and $\bs_1\!\in\!\Z^{\ell_1},\ldots,\bs_N\!\in\!\Z^{\ell_N}$ such~that
\BE{EqGWsthm_e3} A_1,A_2\in A_X\!\#_V\!A_Y, \quad
|\bs_r|=A_X\cdot_X\!V=A_Y\cdot_Y\!V=|A_1|_{V_r}=|A_2|_{V_r}~~\forall~r\!=\!1,\ldots,N'\,.\EE
By the second assumptions in~\eref{EqGWsthm_e} and~\eref{EqGWsthm_e3}, 
the second assumption in~\eref{DiagFluxcrl_e} is satisfied.
By Corollary~\ref{DiagFlux_crl}, the two numbers in~\eref{EqGWsthm_e1} are thus the same
for all relevant $(A_X,A_Y)$ and $\bs_1\!\in\!\Z^{\ell_1},\ldots,\bs_N\!\in\!\Z^{\ell_N}$.
\end{proof}

\noindent
Under the assumptions of Proposition~\ref{EqGWs_prp1}, $N,N'\!=\!1$ and
$\Flux(V_1)_{X,Y}\!=\!\cR_{X,Y}^V$.
Thus, the second condition in~\eref{EqGWsthm_e} reduces to 
the second condition in~\eref{EqGWsprp1_e} in this case.
Under the assumptions of Proposition~\ref{EqGWs_prp2}, 
$$|A_1|_{V_r}\Flux(V_r)_{X,Y}=\Flux(V_r)_{X,Y}~~\forall~r\!=\!1,\ldots,N, \qquad
 \bigoplus_{r=1}^N\Flux(V_r)_{X,Y}=\cR_{X,Y}^V.$$
Thus, the second condition in~\eref{EqGWsthm_e} reduces to 
the second condition in~\eref{EqGWsprp2_e} in this case.
Combining Theorem~\ref{EqGWs_thm} with Gromov's Compactness,
we obtain the following statement.

\begin{crl}\label{EqGWs_crl}
Let $(X,\om_X)$ and $(Y,\om_Y)$ be compact symplectic manifolds,
$V\!\subset\!X,Y$ be a common compact symplectic divisor with topological
components $V_1,\ldots,V_N$, 
and $\Phi$ be an isomorphism of complex line bundles as in~\eref{cNpair_e}.
Suppose $\cN_X{V_{r^*}}\!\approx\!V_{r^*}\!\times\!\C$ and 
$\Flux(V_{r^*})_{X,Y}$ is infinite for some $r^*\!=\!1,\ldots,N$.
If $A\!\in\!H_2(X\!\#_V\!Y;\Z)$ is such that $|A|_{V_{r^*}}\!\neq\!0$,
then all degree~$A$ GW-invariants of $X\!\#_V\!Y$ with $\Phi$-admissible inputs
vanish.
\end{crl}

\section{The convolution product on covers}
\label{ConvProd_sec}

\noindent
In Section~\ref{AbConv_subs}, we describe a convolution-like operation on abelian covers
that takes a product of covers for $V'\!\cup\!V$ and for $V\!\cup\!V''$
to a cover for $V'\!\cup\!V''$.
In the symplectic sum context, this operation takes a product of rim tori covers for
$(X,V'\!\cup\!V)$ and $(Y,V\!\cup\!V'')$ to a rim tori cover for $(X\!\#_V\!Y,V'\!\cup\!V'')$;
see Section~\ref{RTconv_subs}.
This operation is needed to make sense of \cite[(10.8)]{IPsum}, 
which expresses GW-invariants of $(X\!\#_V\!Y,V'\!\cup\!V'')$
in terms of GW-invariants of $(X,V'\!\cup\!V)$ and $(Y,V\!\cup\!V'')$.
Throughout Sections~\ref{AbConv_subs} and~\ref{RTconv_subs}, 
we let $H_1(V)\!=\!H_1(V;\Z)$ for any topological space~$V$.
We continue with the notation of Section~\ref{review_sec}.

\subsection{Abelian covers}
\label{AbConv_subs}

\noindent
Let $V,V',V''$ be topological spaces,
$$W=V'\!\sqcup\!V'', \qquad W'=V'\!\sqcup\!V, \qquad W''=V\!\sqcup\!V''\,,$$
and $H_{\De}$ be the diagonal submodule of $H_1(V)\!\oplus\!H_1(V)$.
Denote by
\begin{gather*}
\io\!: H_1(V')\!\oplus\!H_1(V'')\lra 
H_1(V')\!\oplus\!H_1(V)\!\oplus\!H_1(V)\!\oplus\!H_1(V''),\\
\pi_V\!:  H_1(V')\!\oplus\!H_1(V)\lra H_1(V), \quad
\pi_V\!:  H_1(V)\!\oplus\!H_1(V'')\lra H_1(V),\\
\pi_V\!: H_1(V')\!\oplus\!H_1(V)\!\oplus\!H_1(V)\!\oplus\!H_1(V'')\lra
H_1(V)\!\oplus\!H_1(V)
\end{gather*}
the canonical inclusion and projections.\\

\noindent
Fix submodules 
\begin{gather*}
H_1,H_2\subset H_1(V), \quad \obu{H}_{12}\subset H_1(V)\!\oplus\!H_1(V), \quad
\ori{H}_{12}\subset H_1(V')\!\oplus\!H_1(V''),\\
\wt{H}_1\subset  H_1(V')\!\oplus\!H_1(V), \quad 
\wt{H}_2\subset  H_1(V)\!\oplus\!H_1(V''),\quad
\wt{H}_{12}\subset H_1(V')\!\oplus\!H_1(V)\!\oplus\!H_1(V)\!\oplus\!H_1(V''), 
\end{gather*}
such that 
\begin{gather}
\label{H1H2cond_e1}
H_1\!\oplus\!H_2,H_{\De}\subset \obu{H}_{12}, \quad
\wt{H}_1\!\oplus\!\wt{H}_2,0\!\oplus\!H_{\De}\!\oplus\!0\subset \wt{H}_{12}, \\
\label{H1H2cond_e2}
\io(\ori{H}_{12})\supset \wt{H}_{12}\cap\ker\pi_V\,, \quad
\pi_V(\wt{H}_1)\subset H_1,\quad \pi_V(\wt{H}_2)\subset H_2,\quad 
\pi_V(\wt{H}_{12})\supset \obu{H}_{12}.
\end{gather}
Choose a collection of representatives
\BE{g12coll_e}\big\{\obu\ga_j\big\}\subset H_1(V)\!\oplus\!H_1(V) \EE
for the cosets of $\obu{H}_{12}$. \\

\noindent
Let $V_1,\ldots,V_N$, $V_1',\ldots,V_{N'}'$, and $V_1'',\ldots,V_{N''}''$
be the topological components of~$V$, $V'$, and~$V''$, respectively, and 
\BE{bstuples_e}\bs\equiv(s_{r;i})_{i\le\ell_r,r\le N}\in\prod_{r=1}^N\!\Z_{\pm}^{\ell_r}\,, ~~
\bs'\equiv(s_{r;i}')_{i\le\ell_r',r\le N'}\in\prod_{r=1}^{N'}\!\Z_{\pm}^{\ell_r'}\,, ~~
\bs''\equiv(s_{r;i}'')_{i\le\ell_r'',r\le N''}\in\prod_{r=1}^{N''}\!\Z_{\pm}^{\ell_r''}\,.\EE
Denote by 
$$\pi_{\bs}\!: W_{\bs'\bs}'\!\equiv\!V_{\bs'}'\!\times\!V_{\bs}\lra V_{\bs}
\qquad\hbox{and}\qquad
\pi_{\bs}\!: W_{\bs\bs''}''\!\equiv\!V_{\bs}\!\times\!V_{\bs''}''\lra V_{\bs}$$
the projection maps. Let
$$\wh{W}_{\wt{H}_1,\wt{H}_2;\bs'\bs\bs''} =
 \wh{W'}_{\wt{H}_1;\bs'\bs}\!\times_{V_{\bs}}\!\wh{W''}_{\wt{H}_2;\bs\bs''}
\equiv
\big\{\pi_{\bs}\!\circ\!\pi_{\wt{H}_1;\bs'\bs}\!\times\!
\pi_{\bs}\!\circ\!\pi_{\wt{H}_2;\bs\bs''}\big\}^{-1}
\big(\De_{\bs}^V\big).$$
We will describe a continuous map 
\BE{Xidfn_e0}\Xi_{\wt{H}_1,\wt{H}_2}^{H_{12},\wt{H}_{12}}\!:
\wh{W}_{\wt{H}_1,\wt{H}_2;\bs'\bs\bs''}\lra \wh{W}_{\ori{H}_{12};\bs'\bs''}\EE
so that the diagram in Figure~\ref{ConvProd_fig} commutes.\\

\begin{figure}
$$\xymatrix{\wh{V}_{\wt{H}_1,\wt{H}_2;\bs}  \ar[d]|{\pi_{H_1;\bs}\times\pi_{H_2;\bs}}&
&\ar[ll]_{\wt{q}_{\bs}^{\De}}\wh{W}_{\wt{H}_1,\wt{H}_2;\bs'\bs\bs''} 
\ar[rr]^{\Xi_{\wt{H}_1,\wt{H}_2}^{H_{12},\wt{H}_{12}}} 
\ar[d]|{\pi_{\wt{H}_1;\bs'\bs}\times\pi_{\wt{H}_2;\bs\bs''}}
&&  \wh{W}_{\ori{H}_{12};\bs'\bs''} \ar[d]|{\pi_{\ori{H}_{12};\bs'\bs''}} \\
\De_{\bs}&&\ar[ll]_{\pi_{\bs}\times\pi_{\bs}}
V_{\bs'}'\!\times\!\De_{\bs}^V\!\times\!V_{\bs''}'' \ar[rr]&& V_{\bs'}'\!\times\!V_{\bs''}''}$$
\caption{The convolution product and related morphisms.}
\label{ConvProd_fig}
\end{figure}

\noindent
Choose collections 
\begin{gather}
\label{wtgaj1j2_e1} \{(\ga_{j_1}',\ga_{1;j_1})\}\subset H_1(V')\!\oplus\!H_1(V), 
\quad \{(\ga_{2;j_2},\ga_{j_2}'')\}\subset H_1(V)\!\oplus\!H_1(V''),\\
\label{wtgaj1j2_e2}
 \big\{\ori\ga_j\big\}\subset H_1(V')\!\oplus\!H_1(V'')
\end{gather}
of representatives for the elements of 
$$\frac{\cR_{\wt{H}_1}}{\cR_{\wt{H}_1;\bs'\bs}'}\,,\qquad 
\frac{\cR_{\wt{H}_2}}{\cR_{\wt{H}_2;\bs\bs''}'}\,,  \qquad 
\frac{\cR_{\ori{H}_{12}}}{\cR_{\ori{H}_{12};\bs'\bs''}'},$$
respectively.
Suppose
$$\x\equiv \big(\big([\ga_{j_1}',\ga_{1;j_1}]_{\wt{H}_1;\bs'\bs},
[\wh{x}',\ga\!\cdot\!\wh{x}]_{\wt{H}_1}\big),
\big([\ga_{2;j_2},\ga_{j_2}'']_{\wt{H}_2;\bs\bs''},[\wh{x},\wh{x}'']_{\wt{H}_2}\big)\big)$$
for some $\ga\!\in\!H_1(V_{\bs})$.
Let $\obu{h}\!\in\!\obu{H}_{12}$ and  $\obu\ga$ be a coset representative 
from the collection in~\eref{g12coll_e} such~that 
\BE{gasplit_e}
\big(\ga_{1;j_1},\ga_{2;j_2}\big)+
\big(\Phi_{V;\bs}(\ga),0\big)=\obu\ga+\obu{h}\,.\EE
By the last assumption in~\eref{H1H2cond_e2},
\BE{wtH12_e} \big(h',\obu{h},h''\big)\in \wt{H}_{12} \EE
for some $(h',h'')$. 
Let $(\ga',\ga'')\!\in\!H_1(V_{\bs'}')\!\oplus\!H_1(V_{\bs''}'')$,
$\ori{h}\!\in\!\ori{H}_{12}$, and  $\ori\ga$ be a coset representative 
from the  collection in~\eref{wtgaj1j2_e2} such~that 
\BE{gasplit_e2}
\big(\ga_{j_1}'\!-\!h',\ga_{j_2}''\!-\!h''\big)=
\big(\Phi_{V';\bs'}(\ga'),\Phi_{V'';\bs''}(\ga'')\big)+\ori\ga+\ori{h}\,.\EE
We set 
\BE{Xidfn_e}\Xi_{\wt{H}_1,\wt{H}_2}^{\obu{H}_{12},\wt{H}_{12}}\big(\x\big)
=\big([\ori\ga]_{\ori{H}_{12};\bs'\bs''},
[\ga'\!\cdot\!\wh{x}',\ga''\!\cdot\!\wh{x}'']_{\ori{H}_{12}}\big).\EE
Below we show that the right-hand side of the above expression depends only on~$\x$.\\

\noindent
With the tuple~\eref{wtH12_e} fixed,  $(\ga',\ga'')$ is defined by~\eref{gasplit_e2}
up to an element of $(\ori{H}_{12})_{\bs'\bs''}$;
such an element has no effect on the right-hand side of~\eref{Xidfn_e}.
By the first assumption in~\eref{H1H2cond_e2},  \eref{wtH12_e} is determined 
by the left-hand side in~\eref{gasplit_e} up to an element of $\ori{H}_{12}$;
such an element has no effect on $(\ga',\ga'')$ in~\eref{gasplit_e2}.\\

\noindent
Suppose $(\al,\be)\!\in\!(\wt{H}_1)_{\bs'\bs}$ and so
$$\x=\big(\big([\ga_{1;j_1}',\ga_{1;j_1}]_{\wt{H}_1;\bs'\bs},
[\al\!\cdot\!\wh{x}',\be\ga\!\cdot\!\wh{x}]_{\wt{H}_1}\big),
\big([\ga_{2;j_2},\ga_{2;j_2}'']_{\wt{H}_2;\bs\bs''},[\wh{x},\wh{x}'']_{\wt{H}_2}\big)\big).$$
By the second assumption in~\eref{H1H2cond_e2}, $\Phi_{V;\bs}(\be)\!\in\!H_1$.
By the first inclusion in the first assumption in~\eref{H1H2cond_e1},
this change thus adds $(\Phi_{V;\bs}(\be),0)$ to $\obu{h}$ in~\eref{gasplit_e}. 
By the first inclusion in the second assumption in~\eref{H1H2cond_e1},
it adds $\Phi_{V';\bs'}(\al)$ to~$h'$ in~\eref{wtH12_e}
and subtracts $\al$ from~$\ga'$ in~\eref{gasplit_e2}.
Thus, \eref{Xidfn_e} becomes 
$$\Xi_{\wt{H}_1,\wt{H}_2}^{\obu{H}_{12},\wt{H}_{12}}\big(\x\big)
=\big([\ori\ga]_{\ori{H}_{12};\bs'\bs''},
[\ga'\al^{-1}\!\cdot\!(\al\!\cdot\!\wh{x}'),\ga''\!\cdot\!\wh{x}'']_{\ori{H}_{12}}\big),$$
which agrees with~\eref{Xidfn_e}.\\

\noindent
Suppose $(\al,\be)\!\in\!(\wt{H}_2)_{\bs\bs''}$ and so
$$\x=\big(\big([\ga_{1;j_1}',\ga_{1;j_1}]_{\wt{H}_1;\bs'\bs},
[\wh{x}',\ga\!\cdot\!\wh{x}]_{\wt{H}_1}\big),
\big([\ga_{2;j_2},\ga_{2;j_2}'']_{\wt{H}_2;\bs\bs''},
[\al\!\cdot\!\wh{x},\be\!\cdot\!\wh{x}'']_{\wt{H}_2}\big)\big).$$
By the third assumption in~\eref{H1H2cond_e2}, $\Phi_{V;\bs}(\al)\!\in\!H_2$.
By the first assumption in~\eref{H1H2cond_e1},
this change thus subtracts $(\Phi_{V;\bs}(\al),0)$ 
from $\obu{h}$ in~\eref{gasplit_e}. 
By the second assumption in~\eref{H1H2cond_e1},
it adds $\Phi_{V'';\bs''}(\be)$ to~$h''$ in~\eref{wtH12_e}
and subtracts $\be$ from~$\ga''$ in~\eref{gasplit_e2}.
Thus, \eref{Xidfn_e} becomes 
$$\Xi_{\wt{H}_1,\wt{H}_2}^{\obu{H}_{12},\wt{H}_{12}}\big(\x\big)
=\big([\ori\ga]_{\ori{H}_{12};\bs'\bs''},
[\ga'\!\cdot\!\wh{x}',\ga''\be^{-1}\!\cdot\!(\be\!\cdot\!\wh{x}'')]_{\ori{H}_{12}}\big),$$
which agrees with~\eref{Xidfn_e}.
It follows that the right-hand side of~\eref{Xidfn_e} depends only on~$\x$. \\

\noindent
Let $H_{12}\!\subset\!H_1(V)$ denote the image of $\obu{H}_{12}$ under the homomorphism
$$H_1(V)\!\oplus\!H_1(V)\lra H_1(V), \qquad (\ga_1,\ga_2)\lra\ga_1\!-\!\ga_2.$$
Thus, the homomorphism
\BE{H12wtH12_e}\De_{H_{12},\obu{H}_{12}}\!: \frac{H_1(V)\!\oplus\!H_1(V)}{\obu{H}_{12}} \lra 
\cR_{H_{12}}\!\equiv\!\frac{H_1(V)}{H_{12}}, \quad
\De_{H_{12},\obu{H}_{12}}\big([\ga_1,\ga_2]_{\obu{H}_{12}}\big)=
\big[\ga_1\!-\!\ga_2\big]_{H_{12}}\,,\EE
is well-defined and is an isomorphism.
By the middle two assumptions in~\eref{H1H2cond_e2}, there is a natural projection map
$$\wt\pi_{\bs}\!\times\!\wt\pi_{\bs}\!:\wh{W'}_{\wt{H}_1;\bs'\bs}\!\times\!\wh{W''}_{\wt{H}_2;\bs\bs''}
\lra \wh{V}_{H_1;\bs}\!\times\!\wh{V}_{H_2;\bs},$$
which restricts to a map
$$\wt{q}_{\bs}^{\De}\!:\wh{W}_{\wt{H}_1,\wt{H}_2;\bs'\bs\bs''}\lra \wh{V}_{H_1,H_2;\bs};$$
see \cite[(5.11)]{GWrelIP}.
If the (unparametrized) collections $\{\ga_{1;j_1}\},\{\ga_{2;j_2}\}\!\subset\!H_1(V)$ 
obtained from~\eref{wtgaj1j2_e1} are
the same as the collections~\eref{gaj1j2_e} used in the construction of~\eref{wtPsi_e},
then 
$$\wt\Psi_{H_1,H_2}^{H_{12}}\big(\wt{q}_{\bs}^{\De}(\x)\big)=
\De_{H_{12},\obu{H}_{12}}\big([\obu\ga]_{\obu{H}_{12}}\big)\in \cR_{H_{12}},$$
with $\obu\ga\!\equiv\!\obu\ga(\x)$ given by~\eref{gasplit_e}.\\

\noindent
The collections~\eref{wtgaj1j2_e1} can be chosen so that the induced collections
$\{\ga_{1;j_1}\}$ and $\{\ga_{2;j_2}\}$ contain precisely one representative for
each coset in $\cR_{H_1}/\cR_{H_1;\bs}'$ and $\cR_{H_2}/\cR_{H_2;\bs}'$ , respectively, 
if and only if the inclusions in the two middle conditions in~\eref{H1H2cond_e2} are equalities.
In such a case,
\BE{whWdiag_e}\wh{W}_{\wt{H}_1,\wt{H}_2;\bs'\bs\bs''}^{\wt\eta}\equiv 
\big\{\wt{q}_{\bs}^{\De}\big\}^{-1}\big(\{\wt\Psi_{H_1,H_2}^{H_{12}}\}^{-1}(\wt\eta)\big)
\subset \wh{W'}_{\wt{H}_1;\bs'\bs}\!\times\!\wh{W''}_{\wt{H}_2;\bs\bs''}\EE 
is a closed subset for every $\wt\eta\!\in\!\cR_{H_{12}}$.
If in addition $V,V',V''$ are oriented manifolds, then so is the subset~\eref{whWdiag_e}.
The cohomology class determined by this submanifold as in Section~\ref{DiagonSplit_subs}
satisfies
\BE{GenPD_e2}
\PD_{\wt{H}_1,\wt{H}_2;\bs'\bs\bs''}^{\wt\eta}\De
=\big\{\wt\pi_{\bs}\!\times\!\wt\pi_{\bs}\big\}^*
\big(\PD_{H_1,H_2;\bs}^{\wt\eta}\De\big)\in 
H^*\big(\wh{W'}_{\wt{H}_1;\bs'\bs}\!\times\!\wh{W''}_{\wt{H}_2;\bs\bs''};\Q\big),\EE
with $\PD_{H_1,H_2;\bs}^{\wt\eta}\De$ as in~\eref{GenPD_e}.

\begin{eg}\label{ConvProd_eg1}
Let $F$ be a  connected topological space, $V,V''\!=\!F$, $V'\!=\!\eset$,
\begin{gather*}
H_1,\ori{H}_{12},\wt{H}_1=0, \quad H_2\!=\!H_1(V), \quad 
\obu{H}_{12}\!=\!H_1(V)\!\oplus\!H_1(V), \quad
\wt{H}_2\!=\!H_{\De}\subset H_1(V)\!\oplus\!H_1(V''),\\
\wt{H}_{12}=\big\{(0,\al,\al\!+\!\be,\be)\!:\,\al,\be\!\in\!H_1(F)\big\}
\subset H_1(V')\!\oplus\!H_1(V)\!\oplus\!H_1(V)\!\oplus\!H_1(V'').
\end{gather*}
In this case, the collection~\eref{g12coll_e} consists of a single element,
which we can take $\obu\ga\!=\!0$.
Let $\bs_1\!\in\!\Z_{\pm}^{\ell_1}$, $\bs_2\!\in\!\Z_{\pm}^{\ell_2}$,
and
\BE{ConvProd_e1}\{\ga_{1;j_1}\},\{\ga_{2;j_2}\},\{\ga_{12;j_{12}}\}\subset H_1(F)\EE
be collections of representatives for the cosets of 
$\gcd(\bs_1)H_1(F)$, $\gcd(\bs_2)H_1(F)$, and $\gcd(\bs_1,\bs_2)H_1(F)$, respectively.
With identifications as in Example~\ref{S2T2_eg0}, the map~\eref{Xidfn_e} becomes
\begin{equation*}\begin{split}
\frac{H_1(F)}{\gcd(\bs_1)H_1(F)}\times
\frac{H_1(F)}{\gcd(\bs_1,\bs_2)H_1(F)}\times
\wh{F}_{0;\bs_1}'\!\times_{F_{\bs_1}}\!\wh{F}_{0;\bs_1(-\bs_2)}' &\lra  
\frac{H_1(F)}{\gcd(\bs_2)H_1(F)}\times\wh{F}_{0;\bs_2}', \\
\big([\ga_{1;j_1}]_{0;\gcd(\bs_1)},[\ga_{12;j_{12}}]_{0;\gcd(\bs_1\bs_2)},
[\ga\!\cdot\!\wh{x}]_0,[\wh{x},\wh{x}'']_0\big) &\lra 
\big([\ga_{2;j_2}]_{0;\gcd(\bs_2)},[\ga''\!\cdot\!\wh{x}'']_0\big),
\end{split}\end{equation*}
where $\ga''\!\in\!H_1(F_{\bs_2})$ and $\ga_{2;j_2}$ is an element of the second collection 
in~\eref{ConvProd_e1} such~that 
$$\Phi_{F;\bs_1}(\ga)\!+\!\ga_{1;j_1}\!-\!\ga_{12;j_{12}}
=\Phi_{F;\bs_2}(\ga'')\!+\!\ga_{2;j_2} \in H_1(F).$$
\end{eg}

\begin{eg}\label{ConvProd_eg2}
Suppose in addition that $F\!=\!\T^2$ and $\ell_1,\ell_2\!\ge\!1$.
With the identifications as in Examples~\ref{Tcov_eg} and \ref{ESsum_eg0} 
and in the second half of Example~\ref{S2T2_eg0}, 
\begin{equation*}\begin{split}
&\wh{F}_{0;\bs_1}\!\times_{F_{\bs_1}}\!\wh{F}_{0;\bs_1(-\bs_2)}\\
&\qquad=
\big\{\big(x,[z_i\!+\!s_{1;i}^{-1}x]_{i\le\ell_1},
y,\big([z_i\!+\!s_{1;i}^{-1}y]_{i\le\ell_1},[z_{2;i}\!-\!s_{2;i}^{-1}y]_{i\le\ell_2}\big)\big)\!\in\!
\C\!\times\!\T_{\bs_1}^{2(\ell_1-1)}\!\times\!\C\!\times\!\T_{\bs_1(-\bs_2)}^{2(\ell-1)}\big\},
\end{split}\end{equation*}
where $\ell\!=\!\ell_1\!+\!\ell_2$.
The map~\eref{Xidfn_e} becomes 
\BE{ConvEgMap_e}\begin{split}
&\big(x,[z_i\!+\!s_{1;i}^{-1}x]_{i\le\ell_1},
y,\big([z_i\!+\!s_{1;i}^{-1}y]_{i\le\ell_1},[z_{2;i}\!-\!s_{2;i}^{-1}y]_{i\le\ell_2}\big)\big)\\
&\hspace{1in}\lra
\Big(\ell_2^{-1}\big(\ell_1x\!-\!\ell y\big),
\big[z_{2;i}\!+\!\ell_2^{-1}s_{2;i}^{-1}\big(\ell_1x\!-\!\ell y\big)\big]_{i\le\ell_2}\Big) 
\in \C\!\times\!\T_{\bs_2}^{2(\ell_2-1)}\,.
\end{split}\EE
By \eref{GenPD_e2}, 
\BE{ConvProdeg_e3}
\PD_{0,H_{\De};()\bs_1\bs_2}^0\De
=\big\{\id\!\times\!\wt\pi_{\bs_1}\big\}^*
\big(\PD_{0,H;\bs_1}^0\De\big)\in 
H^*\big(\C\!\times\!\T_{\bs_1}^{2(\ell_1-1)}\!\times\!\C\!\times\!\T_{\bs_1(-\bs_2)}^{2(\ell-1)};
\Q\big),\EE
where $H\!=\!H_1(F)$,
$$\wt\pi_{\bs_1}\!:\C\!\times\!\T_{\bs_1(-\bs_2)}^{2(\ell-1)}\lra \T^{2\ell_1}, \qquad
\wt\pi_{\bs_1}\big(y,[z_{1;i}]_{i\le\ell_1},[z_{2;i}]_{i\le\ell_2}\big)=
\big[z_{1;i}\!-\!s_{1;i}^{-1}y\big]_{i\le\ell_1},$$
and
$$\PD_{0,H;\bs_1}^0\De \in 
H^*\big(\C\!\times\!\T_{\bs_1}^{2(\ell_1-1)}\!\times\!\T^{2\ell_1};\Q\big)$$
is described by~\eref{ESsum_e2}.
\end{eg}

\subsection{Rim tori covers}
\label{RTconv_subs}

\noindent
We begin this section by defining the rim tori convolution \cite[(10.8)]{IPsum}
as a special case of the continuous map~\eref{Xidfn_e0}; see~\eref{XiXYVdfn_e}.
We then show that its composition with the lifted evaluation morphisms~\eref{EvLift_e}
for the input divisors is compatible with lifted evaluation morphisms for the output divisor,
provided the relevant coset representatives are chosen consistently;
see Proposition~\ref{RimToriAct_prp3} and Figure~\ref{RimToriAct_fig2}.
We continue with the notation introduced at the beginning of Section~\ref{GlDeg_subs}.\\

\noindent
Throughout this section, $X$ and $Y$ denote compact oriented manifolds, $V,V'\!\subset\!X$
and $V,V''\!\subset\!Y$ are compact oriented disjoint submanifolds of codimension~$\fc$,
and
$$W=V'\!\sqcup\!V'', \qquad W'=V'\!\sqcup\!V, \qquad W''=V\!\sqcup\!V''\,.$$
Let \hbox{$\vph\!:S_XV\!\lra\!S_YV$} be an orientation-reversing
diffeomorphism commuting with the projections to~$V$ and
$X\!\#_{\vph}\!Y$ be the manifold
obtained by gluing the complements of tubular neighborhoods of~$V$ in~$X$ and~$Y$
by~$\vph$ along their common boundary.
Define
$$\wt\cR_{X\#_{\vph}Y}^W=\ker\big\{q_{\vph*}\!\circ\!\io_{X\#_{\vph}Y-W*}^{X\#_{\vph}Y}
\!:\,H_{\fc}(X\!\#_{\vph}\!Y\!-\!W;\Z)\!\lra\!H_{\fc}(X\!\cup_V\!Y;\Z)\big\},$$
where $q_{\vph}\!:X\!\#_{\vph}\!Y\!\lra\!X\!\cup_V\!Y$ is a smooth map obtained
by collapsing circles in the boundaries of the two complements;
see \cite[Section~2.2]{GWrelIP}.\\

\noindent
With $H_{\fc}(SV;\Z)_{X,Y}$ given by~\eref{Hfcdfn_e}, 
denote by $\obu{H}_{X,Y}^{SV}$ and $\wt{H}_{X,Y}^{SV}$ the images of
the homomorphisms 
\begin{equation*}\begin{split}
H_{\fc}(SV;\Z)_{X,Y}
&\stackrel{(\io_{SV*}^{X-V},-\io_{SV*}^{Y-V})}{\xra{2.5}}
\cR_X^V\!\oplus\!\cR_Y^V \stackrel{\approx}{\lra} 
\frac{H_1(V)}{H_X^V} \oplus \frac{H_1(V)}{H_Y^V} \qquad\hbox{and}\\
H_{\fc}(SV;\Z)_{X,Y}
&\stackrel{(\io_{SV*}^{X-W'},-\io_{SV*}^{Y-W''})}{\xra{2.5}}
\cR_X^{W'}\!\oplus\!\cR_Y^{W''} \stackrel{\approx}{\lra} 
\frac{H_1(V')\!\oplus\!H_1(V)}{H_X^{W'}} \oplus 
\frac{H_1(V)\!\oplus\!H_1(V'')}{H_Y^{W''}}\,,
\end{split}\end{equation*}
respectively; the second homomorphisms above are the isomorphisms of \cite[Corollary~3.2]{GWrelIP}.
Let $\obu{H}_{X,Y}^V$ and $\wt{H}_{X,Y}^V$ be the preimages of 
$\obu{H}_{X,Y}^{SV}$ and $\wt{H}_{X,Y}^{SV}$, respectively, under
the quotient projections
\begin{gather*}
H_1(V)\!\oplus\!H_1(V)\lra \frac{H_1(V)}{H_X^V} \oplus \frac{H_1(V)}{H_Y^V} 
\qquad\hbox{and}\\
H_1(V')\!\oplus\!H_1(V)\!\oplus\!H_1(V)\!\oplus\!H_1(V'')\lra
\frac{H_1(V')\!\oplus\!H_1(V)}{H_X^{W'}} \oplus 
\frac{H_1(V)\!\oplus\!H_1(V'')}{H_Y^{W''}}.
\end{gather*}
In particular,
\BE{H1VXdfn_e0}H_X^V\!\oplus\!H_Y^V,H_{\De}\subset \obu{H}_{X,Y}^V, \quad
H_X^{W'}\!\oplus\!H_Y^{W''},0\!\oplus\!H_{\De}\!\oplus\!0\subset \wt{H}_{X,Y}^V, \quad
\pi_V\big(\wt{H}_{X,Y}^V\big)=\obu{H}_{X,Y}^V\,.\EE
By \eref{H1VXdfn_e},
\BE{tiHvsH_e}\pi_V\big(H_X^{W'}\big)=H_X^V\subset H_1(V), \qquad
\pi_V\big(H_Y^{W''}\big)=H_Y^V\subset H_1(V).\EE
Define 
$$ \ori{H}_{X,Y}^V =\wt{H}_{X,Y}^V \cap H_1(V')\!\oplus\!0\!\oplus\!0\!\oplus\!H_1(V'') \,.$$
Thus, the modules
$$H_1=H_X^V, ~~ H_2=H_Y^V, ~~  \obu{H}_{12} =\obu{H}_{X,Y}^V, ~~
\ori{H}_{12} =\ori{H}_{X,Y}^V, ~~
\wt{H}_1=H_X^{W'}, ~~ \wt{H}_2=H_Y^{W''}, ~~
\wt{H}_{12}=\wt{H}_{X,Y}^V$$
satisfy the conditions~\eref{H1H2cond_e1} and~\eref{H1H2cond_e2};
all four inclusions in~\eref{H1H2cond_e2} are in fact equalities in this~case.\\

\noindent
By \cite[Proposition~4.9]{GWrelIP},  $\ori{H}_{X,Y}^V\!=\!H_{X\#_{\vph}Y}^W$ and
there is a natural isomorphism
$$\wt\fR_{X,Y}^V\!: \frac{H_1(W')\!\oplus\!H_1(W'')}{\wt{H}_{X,Y}^V}  
\lra \wt\cR_{X\#_{\vph}Y}^W \subset H_{\fc}(X\!\#_{\vph}Y\!-\!W;\Z)\,;$$
the $V',V''\!=\!\eset$ case of this isomorphism is the composition of 
the right vertical arrow in~\eref{cRH1diag_e} with~\eref{H12wtH12_e}.
Let 
\BE{gasum_e}\big\{\ga_{\#;j}\big\}\subset H_1(V)\!\oplus\!H_1(V)\EE
be a collection of coset representatives for 
$$\frac{H_1(V)\!\oplus\!H_1(V)}{\obu{H}_{X,Y}^V}\approx \cR_{X,Y}^V
\subset H_{\fc}(X\!\#_{\vph}Y;\Z)\,.$$\\

\noindent
With $\bs,\bs',\bs''$ as in~\eref{bstuples_e}, let 
$$\wh{W}_{X,Y;\bs'\bs\bs''}=\wh{W}_{H_X^{W'},H_Y^{W''};\bs'\bs\bs''}\,.$$
Choose collections 
\BE{wtgaj1j2_e2b}
\big\{\ga_{X;j_1}\big\},\big\{\ga_{Y;j_2}\big\}\subset H_1(V)
\qquad\hbox{and}\qquad
 \big\{\ori\ga_j\big\}\subset H_1(V')\!\oplus\!H_1(V'')\EE
of representatives for the elements of 
$$\frac{\cR_X^V}{\cR_{X;\bs}'^V}
\approx\frac{\cR_{H_X^V}}{\cR_{H_X^V;\bs}'}, \qquad
\frac{\cR_Y^V}{\cR_{Y;\bs}'^V}
\approx\frac{\cR_{H_Y^V}}{\cR_{H_Y^V;\bs}'}\,,\quad\hbox{and}\quad
\frac{\cR_{X\#_{\vph}Y}^W}{\cR_{X\#_{\vph}Y;\bs'\bs''}'^W}
\approx\frac{\cR_{\ori{H}_{X,Y}^V}}{\cR_{\ori{H}_{X,Y}^V;\bs'\bs''}'}.$$ 
By~\eref{tiHvsH_e},
\begin{equation*}\begin{split}
\pi_V\big(H_X^{W'}\!+\!\Im\,\Phi_{W';\bs'\bs}\big)&=
H_X^V\!+\!\Im\,\Phi_{V;\bs} \subset H_1(V), \\
\pi_V\big(H_Y^{W''}\!+\!\Im\,\Phi_{W'';\bs\bs''}\big)&=
H_Y^V\!+\!\Im\,\Phi_{V;\bs} \subset H_1(V).
\end{split}\end{equation*}
Thus, there exist collections
\BE{wtgaj1j2_e1b}
\big\{(\ga_{j_1}',\ga_{X;j_1})\big\}\subset H_1(V')\!\oplus\!H_1(V),  \quad 
\big\{(\ga_{Y;j_2},\ga_{j_2}'')\big\}\subset H_1(V)\!\oplus\!H_1(V'') \EE
of  representatives for the elements of 
$$\frac{\cR_X^{W'}}{\cR_{X;\bs'\bs}'^{W'}}
\approx\frac{\cR_{H_X^{W'}}}{\cR_{H_X^{W'};\bs'\bs}'}\qquad\hbox{and}\qquad 
\frac{\cR_Y^{W''}}{\cR_{Y;\bs\bs''}'^{W''}}
\approx\frac{\cR_{H_Y^{W''}}}{\cR_{H_Y^{W''};\bs\bs''}'}$$ 
so that the corresponding (unparametrized) collections 
$\{\ga_{X;j_1}\},\{\ga_{Y;j_2}\}\!\subset\!H_1(V)$ 
are the same as the first two collections in~\eref{wtgaj1j2_e2b}.
As described in Sections~\ref{GlDeg_subs} and~\ref{AbConv_subs},  
the collections~\eref{gasum_e}, \eref{wtgaj1j2_e2b}, and~\eref{wtgaj1j2_e1b} determine smooth~maps
\begin{gather}
\label{wtPsiXYV_e2}
\wt\Psi_{X,Y}^V\!\!\equiv\!\fR_{X,Y}^V\!\circ\!\wt\Psi_{H_X^V,H_Y^V}^{H_{X,Y}^V}\!\!: 
\wh{V}_{X,Y;\bs}\lra \cR_{X,Y}^V  \subset H_{\fc}(X\!\#_{\vph}\!Y;\Z),\\
\label{XiXYVdfn_e}
\Xi_{X,Y}^{V',V,V''}\!\equiv\!\Xi_{H_X^{W'},H_Y^{W''}}^{\obu{H}_{X,Y}^V,\wt{H}_{X,Y}^V}\!: 
\wh{W}_{X,Y;\bs'\bs\bs''} \lra \wh{W}_{X\#_{\vph}Y;\bs'\bs''}\,.
\end{gather}
We show below that the second map can be compatible with the lifted evaluation morphisms of
Proposition~\ref{RimToriAct_prp}.\\ 

\noindent
Let $\Si_X,\Si_Y$ be compact oriented $\fc$-dimensional manifolds,
$k_X,k_Y\!\in\!\Z^{\ge0}$, and $(A_X,A_Y)$ be an element of 
$H_{\fc}(X;\Z)\!\times_V\!H_{\fc}(Y;\Z)$.
Denote~by
\BE{liftev3_e}\begin{split}
&\wt\ev_X^{V'\cup V}\!:\fX_{\Si_X,k_X;\bs'\bs}^{V',V}(X,A_X)\lra \wh{W'}_{X;\bs'\bs}, \qquad
\wt\ev_Y^{V\cup V''}\!:\fX_{\Si_Y,k_Y;\bs\bs''}^{V,V''}(Y,A_Y) \lra 
\wh{W''}_{Y;\bs\bs''}, \\
&\qquad\hbox{and}\qquad
\wt\ev_{X\#_{\vph}Y}^{V'\cup V''}\!:
\fX_{\Si_X\#\Si_Y;\bs'\bs''}^{V',V''}\big(X\!\#_{\vph}\!Y,A_{\#}\big) \lra
 \wh{W}_{X\!\#_{\vph}\!Y;\bs'\bs''}
\end{split}\EE
the lifted relative evaluation morphisms of Proposition~\ref{RimToriAct_prp} for the relative pairs 
$$(X,V'\!\cup\!V), \qquad (Y,V\!\cup\!V''),\quad \hbox{and}  \quad (X\!\#_{\vph}\!Y,V'\!\cup\!V'')$$
compatible with the two collections in~\eref{wtgaj1j2_e1b} and
the last collection in~\eref{wtgaj1j2_e2b}.
Let
\BE{Xdiagdfn_e3}\begin{split}
\fX_{(\Si_X,\Si_Y),(k_X,k_Y);\bs'\bs\bs''}^{V',V,V''}\big((X,Y),(A_X,A_Y)\big)
&=\big\{\wt\ev_X^{V'\cup V}\!\times\!\wt\ev_Y^{V\cup V''}\big\}^{-1}
(\wh{W}_{X,Y;\bs'\bs\bs''})\\
&=\big\{\pi_{\bs}\!\circ\!\ev_X^{V'\cup V}\!\times\!\pi_{\bs}\!\circ\!\ev_Y^{V\cup V''}\big\}^{-1}
(\De_{\bs}^V)\,.
\end{split}\EE
Each element $(\bff_X,\bff_Y)$ of this space gives rise to a marked map
\BE{bffXY_e3}\bff_X\!\#_{\vph}\bff_Y\in  \bigsqcup_{A_{\#}\in A_X\#_{\vph}A_Y}\hspace{-.3in}
\fX_{\Si_X\#\Si_Y,k_X+k_Y;\bs'\bs''}^{V',V''}(X\!\#_{\vph}\!Y,A_{\#});\EE
see \cite[Section~2.2]{GWrelIP}.

\begin{prp}\label{RimToriAct_prp3}
Suppose $X$, $Y$, $V\!\subset\!X,Y$, $\vph$, and $(A_X,A_Y)$ are as in 
Proposition~\ref{RimToriAct_prp2},
$V'\!\subset\!X\!-\!V$ and $V''\!\subset\!Y\!-\!V$ are  compact oriented submanifolds 
of codimension~$\fc$,
and $\bs,\bs',\bs''$ are as in~\eref{bstuples_e}.
Let  $\wt\ev_X^{V'\cup V}$, $\wt\ev_Y^{V\cup V''}$, and $\wt\ev_{X\#_{\vph}Y}^{V'\cup V''}$
be the lifted evaluation morphisms as in~\eref{liftev3_e} compatible
with the two collections in~\eref{wtgaj1j2_e1b} and the last collection in~\eref{wtgaj1j2_e2b}.
Then there exists a collection 
$$\{\eta_{\#;j}\}\subset H_1(V')\!\oplus\!H_1(V'')$$
indexed as the collection in~\eref{gasum_e} with the following property.
If  $(\bff_X,\bff_Y)$ is an element of the fiber product~\eref{Xdiagdfn_e3} such~that
\BE{gAdfn_e2b}\wt\Psi_{X,Y}^V\big(\wt\ev_X^{V'\cup V}(\bff_X),\wt\ev_Y^{V\cup V''}(\bff_Y)\big)
=[\ga_{\#;j}]_{\obu{H}_{X,Y}^V} \in\cR_{X,Y}^V\,,\EE
then
\BE{gAdfn_e2} 
\Xi_{X,Y}^{V',V,V''}\big(\wt\ev_X^{V'\cup V}(\bff_X),\wt\ev_Y^{V\cup V''}(\bff_Y)\big)
=\Th_{\eta_{\#;j}}\big(\wt\ev_{X\#_{\vph}Y}^{V'\cup V''}(\bff_X\!\#_{\vph}\!\bff_Y)\big).\EE
\end{prp}

\begin{figure}
$$\xymatrix{
\fX_{(\Si_X,\Si_Y);\bs'\bs\bs''}^{V',V,V''}\big((X,Y),(A_X,A_Y)\big)
\ar[rr]^>>>>>>>>>>>>>>>>>>>>>>>>>>{\wt\ev_X^{V'\cup V}\times\wt\ev_Y^{V\cup V''}} \ar[ddd]_{\#_{\vph}} 
\ar[dr]|-{\ev_X^{V'\cup V}\times\ev_Y^{V\cup V''}}
&& \wh{W}_{X,Y;\bs'\bs\bs''} \ar[ddd]|-{g_{A_X,A_Y}\times \Xi_{X,Y}^{V',V,V''}} 
\ar[dl]|-{\pi_{X;\bs\bs'}^{V'\cup V}\times\pi_{Y;\bs\bs''}^{V\cup V''}}\\
& V_{\bs'}'\!\times\!\De_{\bs}^V\!\times\!V_{\bs''}'' \ar[d]\\
& V_{\bs'}'\!\times\!V_{\bs''}''\\
\bigsqcup\limits_{A_{\#}\in A_X\#_{\vph}A_Y}\!\!\!\!\!\!\!\!
\fX_{\Si_X\#\Si_Y;\bs'\bs''}^{V',V''}\big(X\!\#_{\vph}\!Y,A_{\#}\big) 
\ar[rr]^>>>>>>>>>>>>>>>>>>>{\deg\times\wt\ev_{X\#_{\vph}Y}^{V'\cup V''}} 
\ar[ru]|{\ev_{X\#_{\vph}Y}^{V'\cup V''}}&&  
H_{\fc}(X\!\#_{\vph}\!Y;\Z)\!\times\!\wh{W}_{X\#_{\vph}Y;\bs'\bs''}
\ar[lu]|{\pi_{X\#_{\vph}Y;\bs'\bs''}^{V'\cup V''}\circ\pi_2}}$$
\caption{The evaluation and gluing maps of Proposition~\ref{RimToriAct_prp3};
see Remark~\ref{RimToriAct_rmk3} regarding the commutativity of this diagram.}
\label{RimToriAct_fig2}
\end{figure}

\begin{proof} 
Fix base points 
$x_{\bs}\!\in\!V_{\bs},x_{\bs'}'\!\in\!V'_{\bs'},x_{\bs''}''\!\in\!V''_{\bs''}$.
Let 
\begin{equation*}\begin{split}
\bff_X^{\bu}&\equiv \big(z_{X;1}^{\bu},\ldots,z_{X;k_X+\ell_1+\ldots+\ell_N'+\ell_N}^{\bu},f_X^{\bu}\big)
\in  \fX_{\Si_X,k_X;\bs'\bs}^{V',V}(X,A_X),\\
\bff_Y^{\bu}&\equiv \big(z_{Y;1}^{\bu},\ldots,z_{Y;k_Y+\ell_1+\ldots+\ell_N+\ell_N''}^{\bu},f_Y^{\bu}\big)
\in \fX_{\Si_Y,k_Y;\bs\bs''}^{V,V''}(Y,A_Y)
\end{split}\end{equation*}
be such that 
\BE{ConvProd_e3b}
\ev_X^{V'\cup V}(\bff_X^{\bu})=\big(x'_{\bs'},x_{\bs}\big), \qquad
\ev_Y^{V\cup V''}(\bff_Y^{\bu})=\big(x_{\bs},x''_{\bs''}\big),\EE
and \eref{gAdfn_e2b} holds for $(\bff_X,\bff_Y)\!=\!(\bff_X^{\bu},\bff_Y^{\bu})$.
We choose $\eta_{\#;j}$ so that~\eref{gAdfn_e2} holds for 
$(\bff_X,\bff_Y)\!=\!(\bff_X^{\bu},\bff_Y^{\bu})$.
We then show that \eref{gAdfn_e2} holds for all $(\bff_X,\bff_Y)$ satisfying
\eref{ConvProd_e3b}, with  $(\bff_X^{\bu},\bff_Y^{\bu})$
replaced by  $(\bff_X,\bff_Y)$, and~\eref{gAdfn_e2b}.
Since both sides of~\eref{gAdfn_e2} are continuous lifts of the evaluation morphism
\BE{ConvProd_e5}\fX_{(\Si_X,\Si_Y);\bs'\bs\bs''}^{V',V,V''}\big((X,Y),(A_X,A_Y)\big)
\lra V_{\bs'}'\!\times\!V_{\bs''}'', ~~
(\bff_X,\bff_Y)\lra \big(\ev_X^{V'}(\bff_X),\ev_Y^{V''}(\bff_Y)\big),\EE
over the covering projection
$$\wh{W}_{X\#_{\vph}Y;\bs'\bs''} \lra V_{\bs'}'\!\times\!V_{\bs''}''$$
and this morphism is surjective on every topological component of the domain, 
it follows that  \eref{gAdfn_e2} holds for all $(\bff_X,\bff_Y)$ satisfying~\eref{gAdfn_e2b}.
If a pair $(\bff_X^{\bu},\bff_Y^{\bu})$ satisfying~\eref{ConvProd_e3b}
and~\eref{gAdfn_e2b} does not exist, then
the value of~$\eta_{\#;j}$ does not matter.\\

\noindent
By~\eref{ConvProd_e3b}
\BE{ConvProd_e5a}\begin{split}
\wt\ev_X^{V'\cup V}(\bff_X^{\bu})=
\big([\ga_j'^{\bu},\ga_{X;j}^{\bu}]_{X;\bs'\bs},[\wh{x}',\ga\!\cdot\!\wh{x}]_X\big), &\quad
\wt\ev_Y^{V\cup V''}(\bff_Y^{\bu})=
\big([\ga_{Y;j}^{\bu},\ga_j''^{\bu}]_{Y;\bs\bs''},[\wh{x},\wh{x}'']_Y\big)
\end{split}\EE
for some $\ga\!\in\!H_1(V_{\bs})$ and elements 
$(\ga_j'^{\bu},\ga_{X;j}^{\bu})$ and $(\ga_{Y;j}^{\bu},\ga_j''^{\bu})$ of
the two collections in~\eref{wtgaj1j2_e1b}.
Since \eref{gAdfn_e2b} holds for $(\bff_X,\bff_Y)\!=\!(\bff_X^{\bu},\bff_Y^{\bu})$,
there exists $(h'^{\bu},h^{\bu},h''^{\bu})\!\in\!\wt{H}_{X,Y}^V$ such~that 
\BE{ConvProd_e7}h^{\bu}\!\in\!\obu{H}_{X,Y}^V, \qquad 
\big(\ga_{X;j}^{\bu},\ga_{Y;j}^{\bu}\big)+\big(\Phi_{V;\bs}(\ga),0\big)=\ga_{\#;j}+h^{\bu}\,.\EE
Let $(\ga'^{\bu},\ga''^{\bu})\!\in\!H_1(V'_{\bs'})\!\oplus\!H_1(V''_{\bs''})$,
$\ori{h}_j\!\in\!\ori{H}_{X,Y}^V$, and $\ori\ga_j$ be a coset representative
from the last collection in~\eref{wtgaj1j2_e2b} such~that 
\BE{ConvProd_e9}
\big(\ga_j'^{\bu}\!-\!h'^{\bu},\ga_j''^{\bu}\!-\!h''^{\bu}\big)=
\big(\Phi_{V';\bs'}(\ga'^{\bu}),\Phi_{V'';\bs''}(\ga''^{\bu})\big)+\ori\ga_j+\ori{h}_j\,.\EE
By~\eref{Xidfn_e} and~\eref{ConvProd_e5a}-\eref{ConvProd_e9},
\BE{ConvProd_e11}
\Xi_{X,Y}^{V',V,V''}\big(\wt\ev_X^{V'\cup V}(\bff_X^{\bu}),\wt\ev_Y^{V\cup V''}(\bff_Y^{\bu})\big)
=\big([\ori\ga_j]_{X\#_{\vph}Y;\bs'\bs''},
[\ga'^{\bu}\!\cdot\!\wh{x}',\ga''^{\bu}\!\cdot\!\wh{x}'']_{X\#_{\vph}Y}\big).\EE  
Since both sides of~\eref{gAdfn_e2} are  lifts of~\eref{ConvProd_e5},
there exist $\al\!\in\!H(V_{\bs'}')$ and $\be\!\in\!H(V_{\bs''}'')$ such~that 
\BE{ConvProd_e15}\wt\ev_{X\#_{\vph}Y}^{V'\cup V''}(\bff_X^{\bu}\!\#_{\vph}\!\bff_Y^{\bu})
=\big([\ori\ga_j']_{X\#_{\vph}Y;\bs'\bs''},
[\al\ga'^{\bu}\!\cdot\!\wh{x}',\be\ga''^{\bu}\!\cdot\!\wh{x}'']_{X\#_{\vph}Y}\big)\EE
for some $\al\!\in\!H_1(V'_{\bs'})$, $\be\!\in\!H_1(V''_{\bs''})$, and 
a coset representative $\ori\ga_j'$
from the last collection in~\eref{wtgaj1j2_e2b}.
By~\eref{Thetadfn_e}, \eref{ConvProd_e11}, and~\eref{ConvProd_e15}, \eref{gAdfn_e2} with
\BE{etajdfn_e} \eta_{\#;j}=\ori\ga_j-\ori\ga_j'-\big(\Phi_{V';\bs'}(\al),\Phi_{V'';\bs''}(\be)\big)
\in H_1(V')\!\oplus\!H_1(V'')\EE
holds for $(\bff_X,\bff_Y)\!=\!(\bff_X^{\bu},\bff_Y^{\bu})$.\\

\noindent
Let $(\bff_X,\bff_Y)$ be any pair satisfying~\eref{ConvProd_e3b} and~\eref{gAdfn_e2b}. 
By the first assumption and~\eref{ConvProd_e5a}, 
\BE{ConvProd_e23a}\begin{split}
\wt\ev_X^{V'\cup V}(\bff_X)&=
\big([\ga_j',\ga_{X;j}]_{X;\bs'\bs},[\al'\!\cdot\!\wh{x}',\al\ga\!\cdot\!\wh{x}]_X\big), \\
\wt\ev_Y^{V\cup V''}(\bff_Y)&=
\big([\ga_{Y;j},\ga_j'']_{Y;\bs\bs''},[\be\!\cdot\!\wh{x},\be''\!\cdot\!\wh{x}'']_Y\big), 
\end{split}\EE
for some $\al,\be\!\in\!H_1(V_{\bs})$, $\al'\!\in\!H_1(V'_{\bs'})$, $\be''\!\in\!H_1(V''_{\bs''})$, 
and elements 
$(\ga_j',\ga_{X;j})$ and $(\ga_{Y;j},\ga_j'')$ of
the two collections in~\eref{wtgaj1j2_e1b}.
By~\eref{gAdfn_e2b}, there exists $(h',h,h'')\!\in\!\wt{H}_{X,Y}^V$ such~that 
\BE{ConvProd_e27}h\!\in\!\obu{H}_{X,Y}^V, \qquad 
\big(\ga_{X;j},\ga_{Y;j}\big)+\big(\Phi_{V;\bs}(\ga\!+\!\al\!-\!\be),0\big)=\ga_{\#;j}+h\,.\EE
Let $(\ga',\ga'')\!\in\!H_1(V'_{\bs'})\!\oplus\!H_1(V''_{\bs''})$,
$\ori{h}\!\in\!\ori{H}_{X,Y}^V$, and $\ori\ga$ be a coset representative
from the last collection in~\eref{wtgaj1j2_e2b} such~that 
\BE{ConvProd_e29}
\big(\ga_j'\!-\!h',\ga_j''\!-\!h''\big)=
\big(\Phi_{V';\bs'}(\ga'),\Phi_{V'';\bs''}(\ga'')\big)+\ori\ga+\ori{h}\,.\EE
By~\eref{Xidfn_e} and~\eref{ConvProd_e23a}-\eref{ConvProd_e29},
\BE{ConvProd_e31}
\Xi_{X,Y}^{V',V,V''}\big(\wt\ev_X^{V'\cup V}(\bff_X),\wt\ev_Y^{V\cup V''}(\bff_Y)\big)
=\big([\ori\ga]_{X\#_{\vph}Y;\bs'\bs''},
[\ga'\al'\!\cdot\!\wh{x}',\ga''\be''\!\cdot\!\wh{x}'']_{X\#_{\vph}Y}\big).\EE\\

\noindent
By \eref{RimToriAct_e}, \eref{ConvProd_e5a}, and~\eref{ConvProd_e23a},
\BE{ConvProd_e41}\begin{split}
\big[f_X\!\#\!(-f_X^{\bu})\big]&=
\io_{S_XW'*}^{X-W'}\big(\De_X^{W'}\big(\Phi_{V';\bs'}(\al')\!+\!\ga_j'\!-\!\ga_j'^{\bu},
\Phi_{V;\bs}(\al)\!+\!\ga_{X;j}\!-\!\ga_{X;j}^{\bu}\big)\big),\\
\big[f_Y\!\#\!(-f_Y^{\bu})\big]&=
\io_{S_YW''*}^{Y-W''}\big(\De_Y^{W''}\big(\Phi_{V;\bs}(\be)\!+\!\ga_{Y;j}\!-\!\ga_{Y;j}^{\bu},
\Phi_{V'';\bs''}(\be'')\!+\!\ga_j''\!-\!\ga_j''^{\bu}\big)\big).
\end{split}\EE
By~\eref{ConvProd_e41}, \eref{ConvProd_e7}, and~\eref{ConvProd_e27},
\begin{equation*}\begin{split}
&\big[(f_X\!\#_{\vph}\!f_Y)\!\#\!(-(f_X^{\bu}\!\#_{\vph}\!f_Y^{\bu}))\big]
= \io_{X-W'*}^{X\#_{\vph}Y-W}\big(\big[f_X\!\#\!(-f_X^{\bu})\big]\big)
+\io_{Y-W''*}^{X\#_{\vph}Y-W}\big(\big[f_Y\!\#\!(-f_Y^{\bu})\big]\big)\\
&\quad=\wt\fR_{X,Y}^V\big(\big(\Phi_{V';\bs'}(\al')\!+\!\ga_j'\!-\!\ga_j'^{\bu},
\Phi_{V;\bs}(\be), \Phi_{V;\bs}(\be),\Phi_{V'';\bs''}(\be'')\!+\!\ga_j''\!-\!\ga_j''^{\bu}\big)
+(0,h\!-\!h^{\bu},0)\big)\\
&\quad=\wt\fR_{X,Y}^V
\big(\Phi_{V';\bs'}(\al')\!+\!(\ga_j'\!-\!h')\!-\!(\ga_j'^{\bu}\!-\!h'^{\bu}),
0, 0,\Phi_{V'';\bs''}(\be'')\!+\!(\ga_j''\!-\!h'')\!-\!(\ga_j''^{\bu}\!-\!h''^{\bu})\big);
\end{split}\end{equation*}
the last equality makes use of the  second inclusion in
the middle statement in~\eref{H1VXdfn_e0}.
Combining the above with \eref{ConvProd_e9} and~\eref{ConvProd_e29}, we find~that 
\BE{ConvProd_e45}\begin{split}
&\big[(f_X\!\#_{\vph}\!f_Y)\!\#\!(-(f_X^{\bu}\!\#_{\vph}\!f_Y^{\bu}))\big]\\
&\hspace{.5in}= \io_{S_WX\#_{\vph}Y*}^{X\#_{\vph}Y-W}\big(\De_{X\#_{\vph}Y}^W\big(
\Phi_{W;\bs'\bs''}(\al'\!+\!\ga'\!-\!\ga'^{\bu},\be''\!+\!\ga''\!-\!\ga''^{\bu})
\!+\!\ori\ga\!-\!\ori\ga_j\big)\big).
\end{split}\EE
By \cite[Proposition~6.6]{GWrelIP}, Proposition~\ref{RimToriAct_prp} applied to $(X\#_{\vph}Y,W)$,
 \eref{ConvProd_e11}, and~\eref{ConvProd_e45},
$$\Th_{\eta_{\#;j}}\big(\wt\ev_{X\#_{\vph}Y}^{V'\cup V''}(\bff_X\!\#_{\vph}\!\bff_Y)\big)
=\big([\ori\ga]_{X\#_{\vph}Y;\bs'\bs''},
[\ga'\al'\!\cdot\!\wh{x}',\ga''\be''\!\cdot\!\wh{x}'']_{X\#_{\vph}Y}\big).$$
Comparing with~\eref{ConvProd_e31}, we obtain~\eref{gAdfn_e2}.
\end{proof}

\begin{rmk}\label{RimToriAct_rmk3}
In order to make the diagram in Figure~\ref{RimToriAct_fig2} fully commutative,
the second component in the bottom morphism should be twisted by 
suitable deck transformations $\Th_{\eta_{A_{\#}}}$ of $\wh{W}_{X\#_{\vph}Y;\bs'\bs''}$.
By~\cite[Proposition~6.6]{GWrelIP}, these deck transformations correspond to 
the differences between the possible lifts for $(X\!\#_{\vph}\!Y,W)$.
An intention of \cite[Section~5]{IPrel} is to take $\Th_{\eta_{A_{\#}}}\!=\!\id$
by choosing the lifts~\eref{evXVlift_e} consistently across different relative pairs.
This property, which is used in \cite[(10.8)]{IPsum}, is implicit
in the two set-theoretic descriptions of the rim tori covers;
see \cite[Remark~6.7]{GWrelIP} for more~details.
\end{rmk}

\noindent
For each $A\!\in\!A_X\!\#_{\vph}A_Y$, let 
\BE{whWXYV_e2}\wh{W}_{X,Y;\bs'\bs\bs''}^A=
\big\{\wt\pi_{\bs}\!\times\!\wt\pi_{\bs}\big\}^{-1}
\big(\wh{V}_{X,Y;\bs}^A\big)
\subset \wh{W'}_{X;\bs'\bs}\!\times\!\wh{W''}_{Y;\bs\bs''}\,,\EE
with $\wh{V}_{X,Y;\bs}^A$ as in~\eref{cHXYVsplit_e} and
$$\wt\pi_{\bs}\!:  \wh{W'}_{X;\bs'\bs}\lra \wh{V}_{X;\bs}
\qquad\hbox{and}\qquad
\wt\pi_{\bs}\!:  \wh{W''}_{Y;\bs\bs''}\lra \wh{V}_{X;\bs}$$
being the natural projection maps; see \cite[(6.4)]{GWrelIP}.
The cohomology class determined by the submanifold~\eref{whWXYV_e2} satisfies
\BE{whWXYVPD_e2}
\PD_{X,Y,W;\bs'\bs\bs''}^{V,A}\De
=\big\{\wt\pi_{\bs}\!\times\!\wt\pi_{\bs}\big\}^*
\big(\PD_{X,Y;\bs}^{V,A}\De\big)\in 
H^*\big(\wh{W'}_{X;\bs'\bs}\!\times\!\wh{W''}_{Y;\bs\bs''};\Q\big),\EE
with $\PD_{X,Y;\bs}^{V,A}\De$ as in~\eref{cHclass_e}.

\begin{eg}\label{S2T2sum_eg2}
With $X\!=\!\wh\P^2_9$, $Y\!=\!\P^1\!\times\!\T^2$, and $V\!=\!F\!\subset\!X,Y$ 
as in Example~\ref{ESsum_eg2}, we take $V'\!=\!\eset$ and $V''\!=\!F''\!\subset\!Y$
to be a fiber different from~$F$.
In this case,
\begin{gather*}
H_X^V,\ori{H}_{X,Y}^V,H_X^{W'}=0 \quad H_2\!=\!H_1(V), \quad 
\obu{H}_{X,Y}^V\!=\!H_1(V)\!\oplus\!H_1(V), \quad
H_Y^{W''}\!=\!H_{\De}\subset H_1(V)\!\oplus\!H_1(V''),\\
\wt{H}_{X,Y}^V=\big\{(0,\al,\al\!+\!\be,\be)\!:\,\al,\be\!\in\!H_1(F)\big\}
\subset H_1(V')\!\oplus\!H_1(V)\!\oplus\!H_1(V)\!\oplus\!H_1(V'').
\end{gather*}
The smooth map~\eref{XiXYVdfn_e} can be described as in Example~\ref{ConvProd_eg2}.
By~\eref{ConvProdeg_e3} and~\eref{ESsum_e9}, 
\BE{S2T2sum_e7}\PD_{X,Y,V'';()(1)(1)}^{V,A}\De= 1\!\times\!1\!\times\!\PD_{\T^2}(\pt)
\in H^*\big(\C\!\times\!\C\!\times\!\T^2;\Q\big),\EE
if the $(Y,V\!\cup\!V'')$ covering is written~as
$$\C\!\times\!\T^2 \lra \T^2\!\times\!\T^2, \qquad (z_1,[z_2])\lra \big([z_2\!-\!z_1],[z_2\!+\!z_1]\big).$$
Applying \eref{S2T2sum_e7} to the decomposition
$$(\wh\P^2_9,F)=(\wh\P^2_9,F)\#_F\big(\P^1\!\times\!\T^2,\{0,\i\}\!\times\!\T^2\big)$$
as in Example~\ref{ESsum_eg2}, we obtain
$$\wt\GW_{g,\fs_i+d\ff;(1)}^{\wh\P^2_9,F}\big(\pt^g;1\big)
=\sum_{\begin{subarray}{c}d_1,d_2\in\Z^{\ge0}\\ d_1+d_2=d\end{subarray}}\!\!\!\!\!
\wt\GW_{g,\fs_i+d_1\ff;(1)}^{\wh\P^2_9,F}\big(\pt^g;1\big)
\wt\GW_{0,\fs+d_2\ff;(1),(1)}^{\P^1\times\T^2,\{0,\i\}\times F}\big(;\PD_{\T^2}(\pt),1\big).$$
By the same considerations as in  Example~\ref{ESsum_eg2}, this identity reduces~to
$$\GW_{g,\fs_i+d\ff;(1)}^{\wh\P^2_9,F}\big(\pt^g;1\big)
=\sum_{\begin{subarray}{c}d_1,d_2\in\Z^{\ge0}\\ d_1+d_2=d\end{subarray}}\!\!\!\!\!
\GW_{g,\fs_i+d_1\ff;(1)}^{\wh\P^2_9,F}\big(\pt^g;1\big)
\GW_{0,\fs+d_2\ff;(1),(1)}^{\P^1\times\T^2,\{0,\i\}\times F}\big(;\pt,1\big).$$
This equation is consistent with
$$\GW_{0,\fs+d_2\ff;(1),(1)}^{\P^1\times\T^2,\{0,\i\}\times F}\big(;\pt,1\big)
=\begin{cases} 1,&\hbox{if}~d_2\!=\!0;\\
0,&\hbox{if}~d_2\!\neq\!0;
\end{cases}$$
the last statement can be seen directly from the moduli space consisting of the sections in the first case
and being empty in the second case.
Putting one point on the $Y$ side, we similarly obtain 
$$\GW_{g,\fs_i+d\ff;(1)}^{\wh\P^2_9,F}\big(\pt^g;1\big)
=\sum_{\begin{subarray}{c}d_1,d_2\in\Z^{\ge0}\\ d_1+d_2=d\end{subarray}}\!\!\!\!\!
\GW_{g-1,\fs_i+d_1\ff;(1)}^{\wh\P^2_9,F}\big(\pt^{g-1};1\big)
\GW_{1,\fs+d_2\ff;(1),(1)}^{\P^1\times\T^2,\{0,\i\}\times F}\big(\pt;\pt,1\big).$$
In light of Theorem~\ref{YZ_thm} and Lemma~\ref{RES_lmm}, this identity is consistent with
$$\sum_{d=0}^{\i}
\GW_{1,\fs+d\ff;(1),(1)}^{\P^1\times\T^2,\{0,\i\}\times F}\big(\pt;\pt,1\big)q^d
=qG'(q),$$
with $G(q)$ given by~\eref{Gdfn_e}.
A direct reason for the last equation, which corrects the statement in the middle case 
of the last claim in \cite[Lemma~14.5]{IPrel}, is indicated in the proof of
this lemma; see Remarks~\ref{14.2_rmk} and~\ref{14.4_rmk} for related comments.
\end{eg}

\section{The Bryan-Leung formula}
\label{RES_sec}

\noindent
One of the three applications of the symplectic sum formula appearing 
in~\cite{IPsum} is an alternative proof of \cite[Theorem~1.2]{BL},
a closed formula for the GW-invariants of the blowup~$\wh\P^2_9$ of~$\P^2$ at 9~points.
The approach of~\cite{IPsum} is significantly more efficient than the original proof
in~\cite{BL}, though less direct, as it relies heavily on the symplectic sum formula.
Unfortunately, the argument in~\cite{IPsum} contains some unnecessary statements
and several incorrect statements,
including one which results in the incorrect  main conclusion, \cite[(15.4)]{IPsum}.
It is also missing a crucial intermediate observation; 
see Lemma~\ref{P1T2_lmm} and Remark~\ref{15.3_rmk}.
Some of the incorrect claims in~\cite{IPrel} concern basic points regarding IP-counts.
The approach in~\cite{IPsum} in fact indicates an effective proof of \cite[Theorem~1.2]{BL}
via the standard symplectic sum formula.
In this section, we correct and slightly streamline the argument in~\cite{IPsum} by making
use of the vanishing result of \cite[Theorem~1.1]{GWrelIP} in a  case
when it is an obvious consequence of the existence of the lift~\eref{evXVlift_e}; see~\eref{RelGWvan_e}.
\\

\noindent
Let $\wh\P_9^2$ be a blowup of~$\P^2$ at the 9 intersection points 
of a general pair of smooth cubic curves~$C_1$ and~$C_2$
and $\pi\!:\wh\P^2_9\!\lra\!\P^1$ be the projection to the pencil parametrizing 
the cubics spanned by~$C_1$ and~$C_2$.
This fibration has 9~sections $S_1,\ldots,S_9$ corresponding to the~9 exceptional divisors.
The homology classes $\fs_1,\ldots,\fs_9$ of $S_1,\ldots,S_9$ and 
the homology class $\ff$ of a smooth fiber~$F$ form a basis for 
an index~3 sublattice of $H_2(\wh\P_9^2;\Z)\!\approx\!\Z^{10}$.
For $g\!\in\!\Z^{\ge0}$, define
\BE{Fgdfn_e} \cF_g(q)=\sum_{d=0}^{\i}\GW_{g,\fs_i+d\ff}^{\wh\P_9^2}(\pt^g)q^d,\EE
where the summand denotes the genus~$g$ degree~$\fs_i\!+\!d\ff$ GW-invariant
of~$\wh\P_9^2$ with $g$~point constraints. 
Since $\fs_i^2\!=\!-1$, there is only one holomorphic curve in the homology class~$\fs_i$ and thus
\BE{F0init_e}\cF_0(q)\in 1+q\Q[[q]], \qquad \cF_g(q)\in q\Q[[q]]~~\forall~g\!\in\!\Z^+.\EE
Let
\BE{Gdfn_e}G(q)=\sum_{d=1}^{\i}\si(d)q^d, \qquad\hbox{where}\quad 
\si(d)=\sum_{r|d}r.\EE

\begin{thm}[{\cite[Theorem~1.2]{BL}}]\label{YZ_thm}
For every $g\!\in\!\Z^{\ge0}$,
\BE{YZ_e1}
\cF_g(q)=\bigg(\prod_{d=1}^{\i}(1\!-\!q^d)\bigg)^{-12}\big(qG'(q)\big)^g\,.\EE
\end{thm}

\subsection{Genus 1 GW-invariants of $\P^1\!\times\!\T^2$}
\label{P1xT2_subs}

\noindent
We begin with some observations concerning the genus~1 GW-invariants of $\P^1\!\times\!\T^2$ and 
$(\P^1\!\times\!\T^2,F)$, where $F\!=\!p\!\times\!\T^2$
is a fiber of the projection to the first component.
We denote by $\fs$ and $\ff$ the homology classes of $\P^1\!\times\!q$ and $F$, respectively.

\begin{lmm}[{\cite[Lemma 14.4]{IPsum}}]
\label{P1xT2_lmm1}
The genus~1  GW-invariants of $\P^1\!\times\!\T^2$ satisfy
$$\sum_{d=1}^{\i}d\,\GW_{1,d\ff}^{\P^1\times\T^2}()q^d=2G(q).$$
\end{lmm}

\begin{proof}
Let $L\!=\!\cO_{\T^2}(p\!-\!q)\lra \T^2$ be a non-torsion line bundle 
($L^{\otimes k}\!\not\approx\!\cO_{\T^2}$ for all $k\!\in\!\Z^+$).
The only holomorphic maps in $\P(L\!\oplus\!\cO_{\T^2})\!\approx\!\P^1\!\times\!\T^2$ 
in the homology class $d\ff$ are then covers of 
$$F_0\equiv\P\big(0\!\oplus\!\cO_{\T^2}\big) \qquad\hbox{and}\qquad 
F_{\i}\equiv\P\big(L\!\oplus\!0\big),$$ 
and these maps are regular.
Since the number of degree~$d$ covers $\T^2\!\lra\!\T^2$
(or equivalently of subgroups of $\Z^2$ of index~$d$) is $\si(d)$,
$\ov\fM_{1,0}(\P^1\!\times\!\T^2,d\ff)$ consists of $2\si(d)$ elements.
Since the order of the automorphism group of each of these elements is~$d$,
we conclude that 
$$\GW_{1,d\ff}^{\P^1\times\T^2}()=2\si(d)/d ,$$
as claimed.
\end{proof}

\begin{lmm}[{\cite[Lemma 14.5]{IPsum}}]
\label{P1xT2_lmm2}
The genus~1  GW-invariants of 
$(\P^1\!\times\!\T^2,F)$ with two point constraints satisfy
$$\sum_{d=0}^{\i}\GW_{1,\fs+d\ff;(1)}^{\P^1\times\T^2,F}\big(\pt;\pt\big)q^d=qG'(q).$$
\end{lmm}

\begin{proof}
Suppose $\Si$ is a connected nodal genus~1 curve and $u\!:\Si\!\lra\!\P^1\!\times\!\T^2$ 
is a degree $\fs\!+\!d\ff$ stable map. 
Since $\pi_1\!\circ\!u\!:\Si\!\lra\!\P^1$ has degree~1 and every holomorphic map $\P^1\!\lra\!\T^2$
is constant, $\Si$ contains a unique irreducible component $\Si_0\!\approx\!\P^1$ such that 
$u\!:\Si_0\!\lra\!\P^1\!\times\!q_2$ is an isomorphism for some $q_2\!\in\!\T^2$.
If $\Si_i$ is another irreducible rational component of~$\Si$, then $u|_{\Si_i}$ is constant.
Since $\Si$ is of genus~1, $\Si$ contains at most one (precisely one if $d\!>\!0$)
irreducible genus~1 component~$\Si_1$; furthermore, $u|_{\Si_1}$ is a degree~$d$ 
(unbranched) cover of $q_1\!\times\!\T^2$ for some $q_1\!\in\!\P^1$.
Every such stable map is regular.\\

\noindent
Thus, the subspace
$$\big\{[u,x_1,y_1]\!\in\!\ov\fM_{1,1;(1)}^F(\P^1\!\times\!\T^2,\fs\!+\!d\ff)\!: 
 u(x_1)\!=\!\pt_1,~u(y_1)\!=\!\pt_2\}$$
consists of maps $u\!:\Si_0\!\cup\!\Si_1\!\lra\!\P^1\!\times\!\T^2$ such that $u\!:\Si_0\!\lra\!\P^1\!\times\!\pi_2(\pt_2)$ is an isomorphism and 
$u\!:\Si_1\!\lra\!\pi_1(\pt_1)\!\times\!\T^2$ is a degree~$d$ 
cover.
There are $\si(d)$ such maps~$u$, each of which has an automorphism of order~$d$.
For each choice of the map~$u$, there are $d$ choices for the preimage of~$\pt_1$
and $d$ choices for the nodes on~$\Si_1$.
Thus, 
$$\GW_{1,d\ff;(1)}^{\P^1\times\T^2,F}\big(\pt;\pt\big)=\si(d)/d\cdot d\cdot d=d\si(d);$$
this establishes the claim.
\end{proof}

\noindent
For any compact symplectic manifold $X$, we denote~by
$$\psi_1\in H^2\big(\ov\fM_{g,k}(X,A);\Q\big)$$
the first chern class of the universal cotangent line bundle for the 
first marked point.
For each $d\!\in\!\Z^{\ge0}$, let
\begin{equation*}\begin{split}
\GW_{1,\fs+d\ff}^{\P^1\times\T^2}\big(\tau_1[\ff],\pt\big)
&=\int_{[\ov\fM_{1,2}(\P^1\times\T^2,\fs+d\ff)]^{\vir}}\psi_1
\big(\ev_1^*\PD_{\P^1\times\T^2}\ff\big)\big(\ev_2^*\PD_{\P^1\times\T^2}\pt\big),\\
\GW_{1,\fs+d\ff;(1)}^{\P^1\times\T^2,F}\big(\tau_1[\ff];\pt\big)
&=\int_{[\ov\fM_{1,1;(1)}^F(\P^1\times\T^2,\fs+d\ff)]^{\vir}}\psi_1
\big(\ev_1^*\PD_{\P^1\times\T^2}\ff\big)\big(\ev_2^*\PD_{F}\pt\big).
\end{split}\end{equation*}

\begin{lmm}\label{P1T2_lmm}
For every $d\!\in\!\Z^{\ge0}$, 
\BE{P1T2lmm_e}\GW_{1,\fs+d\ff;(1)}^{\P^1\times\T^2,F}\big(\tau_1[\ff];\pt\big)=
\GW_{1,\fs+d\ff}^{\P^1\times\T^2}\big(\tau_1[\ff],\pt\big).\EE
\end{lmm}

\begin{proof}
As we explained below,
\begin{equation*}\begin{split}
&\big\{[u,x,y]\!\in\!\ov\fM_{1,1;(1)}^F(\P^1\!\times\!\T^2,\fs\!+\!d\ff)\!:~
u(x)\!\in\!\ff,~u(y)\!=\!\pt\big\}\\
&\hspace{1.5in}\approx\big\{[u,x_1,x_2]\!\in\!\ov\fM_{1,2}(\P^1\!\times\!\T^2,\fs\!+\!d\ff)\!:~
u(x_1)\!\in\!\ff,~u(x_2)\!=\!\pt\big\} ,
\end{split}\end{equation*}
where $\pt\!\in\!F$ is a fixed point.
Both spaces contain three irreducible components, which we describe below
and which  have essentially the same deformation/obstruction theory (after capping with $\psi_1$
in the third case);
see Figure~\ref{P1T2_fig}.
This implies the claim.\\ 

\noindent
One of the components common to both spaces is isomorphic to 
$$\P^1\times \big\{[u,x_1']\!\in\!\ov\fM_{1,1}(\T^2,d)\!:\,u(x_1')\!=\!\pt_2\big\},$$
if $\pt\!\equiv\!(\pt_1,\pt_2)\!\in\!\P^1\!\times\!\T^2$.
A generic element of this component is a morphism from a smooth genus~1 curve
 and a rational tail carrying the two marked points which restricts
to a degree~$d$ morphism to a fiber of~$\pi_1$ (specified by~$\P^1$)
on the genus~1 curve and an isomorphism from the tail to the section~$s_{\pt}$ through~$\pt$.\\

\begin{figure}
\begin{pspicture}(-1.5,-2.5)(10,0)
\psset{unit=.3cm}
\psline[linewidth=.02](0,-8)(0,-0)\psline[linewidth=.02](10,-8)(10,-0)
\pscircle*(0,-3){.2}\rput(-.8,-3.6){\sm{$\pt$}}
\rput(-.7,-8){\sm{$F$}}\rput(10.6,-8){\sm{$\ff$}}
\psline[linewidth=.1](-2,-3)(12,-3)\psline[linewidth=.1](5,-1)(5,-7)
\rput(5.5,-1){\sm{$d$}}\rput(.8,-2.4){\sm{$x_2$}}
\pscircle*(10,-3){.2}\rput(9.2,-2.4){\sm{$x_1$}}
\psline[linewidth=.02](18,-8)(18,-0)\psline[linewidth=.02](28,-8)(28,-0)
\pscircle*(18,-3){.2}\rput(17.2,-3.6){\sm{$\pt$}}
\rput(17.3,-8){\sm{$F$}}\rput(28.6,-8){\sm{$\ff$}}
\psline[linewidth=.1](16,-3)(30,-3)\psline[linewidth=.1](28,-1)(28,-7)
\rput(28.5,-1){\sm{$d$}}\rput(18.8,-2.4){\sm{$x_2$}}
\pscircle*(28,-5){.2}\rput(27.2,-5){\sm{$x_1$}}
\psline[linewidth=.02](36,-8)(36,-0)\psline[linewidth=.02](46,-8)(46,-0)
\pscircle*(36,-3){.2}\rput(35.1,-3.1){\sm{$\pt$}}
\rput(35.3,-8){\sm{$F$}}\rput(46.6,-8){\sm{$\ff$}}
\psline[linewidth=.1](34,-5)(48,-5)\psline[linewidth=.1](36,-1)(36,-7)
\rput(36.5,-1){\sm{$d$}}\rput(37,-3.1){\sm{$x_2$}}
\pscircle*(46,-5){.2}\rput(45.2,-4.4){\sm{$x_1$}}
\end{pspicture}
\caption{The three components of $\ov\fM_{1,2}(\P^1\!\times\!\T^2,\fs\!+\!d\ff)$}
\label{P1T2_fig}
\end{figure}

\noindent
Another component is isomorphic~to 
$$\big\{[u,x_1,x_2']\!\in\!\ov\fM_{1,2}(\ff,d)\!:\,u(x_2')\!=\!\pt_2\big\}.$$
A generic element of this component is a morphism from a smooth genus~1 curve
carrying the first marked point
and a rational tail carrying the second marked point which restricts
to a degree~$d$ morphism to the fiber~$\ff$ of~$\pi_1$ on the genus~1 curve and 
an isomorphism from the tail to~$s_{\pt}$.\\

\noindent
The last component of the absolute moduli space is isomorphic~to
$$\big\{[u,x_1',x_2]\!\in\!\ov\fM_{1,2}(F,d)\!:\,u(x_2)\!=\!\pt\big\}.$$
A generic element of this component is a morphism from a smooth genus~1 curve
carrying the second marked point and a rational tail carrying the first marked point 
which restricts to a degree~$d$ morphism to the fiber~$F$ of~$\pi_1$ on the genus~1 curve and 
an isomorphism from the tail to a section of $\pi_1$ (through the image of the first marked 
point of an element of $\ov\fM_{1,2}(F,d)$).
The last component of the relative moduli space is described in the same way, except
the morphism on the genus~1 component above is replaced by 
the $\C^*$-equivalence class of a morphism into the rubber $\P^1\!\times\!\T^2$
from a smooth genus~1 component with a rational tail carrying two marked points
which restricts to a degree~$d$ morphism into a fiber of~$\pi_1$, but not over $0,\i\!\in\!\P^1$,
and an isomorphism from the tail to a section of~$\pi_1$.
The restriction of~$\psi_1$ to this component of the moduli space is the pullback
of the first chern class of the conormal bundle to~$\ff$ by the first evaluation map.
Thus, this restriction vanishes in the absolute and relative cases.
\end{proof}

\begin{rmk}\label{14.2_rmk}
Lemma~\ref{P1xT2_lmm1} corrects the first statement of \cite[Lemma 14.4]{IPsum};
the other two, one of which is similarly~off, are never used.
Lemma~\ref{P1xT2_lmm2} is the second statement of \cite[Lemma 14.5]{IPsum};
the other two are never used.
The proof of \cite[Lemma 14.5]{IPsum} has two mutually canceling errors, 
ignoring the automorphisms of the cover and the choices of the node on the genus~1 component.
The statement at the end of the first paragraph of the proof in~\cite{IPsum} is true only 
generically or after imposing the constraints; 
otherwise, there could be maps with a component mapped into the rubber.
The third statement of \cite[Lemma 14.5]{IPsum} and its proof incorrectly describe 
the IP-counts of $(\P^1\!\times\!\T^2,V\!\equiv\!F_0\!\cup\!F_{\i})$ as being indexed by the rim tori,
suggesting that the rim tori cover $\wh{V}_{\P^1\times\T^2;(1),(1)}$
is $\Z^2\!\times\!F_0\!\times\!F_{\i}$.
As explained in \cite[Example~6.9]{GWrelIP}, 
$\wh{V}_{\P^1\times\T^2;(1),(1)}\!\approx\!\C\!\times\!\T^2$  and there is 
only one IP-count of each type appearing in the third statement of \cite[Lemma 14.5]{IPsum};
there is no indexing by the rim tori.
\end{rmk}

\subsection{GW-invariants of $\wh\P^2_9$}
\label{RES_subs}

\noindent
We next make some observations concerning absolute GW-invariants of~$\wh\P^2_9$ and
relative GW-invariants of $(\wh\P^2_9,F)$, where $F\!\approx\!\T^2$ is a fiber of
the projection $\wh\P^2_9\!\lra\!\P^1$.
From
$$(\fs_i\!+\!d\ff)\cdot\ff=1, \qquad \blr{c_1(T\wh\P^2_9),\ff}=0, \qquad\hbox{and}\quad
\blr{c_1(T\wh\P^2_9),\fs_i}=1,$$
we find that 
\begin{equation*}\begin{split}
\dim_{\C}\ov\fM_{1,0}(\wh\P^2_9,d\ff)&=0+(2\!-\!3)(1\!-\!1)=0,\\
\dim_{\C}\ov\fM_{g,0}(\wh\P^2_9,\fs_i\!+\!d\ff)
&=\dim_{\C}\ov\fM_{g,0;(1)}^F(\wh\P^2_9,\fs_i\!+\!d\ff)=1+(2\!-\!3)(1\!-\!g)=g.
\end{split}\end{equation*}
Thus, $\GW_{1,d\ff}^{\wh\P^2_9}()$, $\GW_{g,\fs_i+d\ff}^{\wh\P^2_9}(\pt^g)$, and 
$\GW_{g,\fs_i+d\ff;(1)}^{\wh\P^2_9,F}(\pt^g)$,
where $\pt^g$ denotes $g$ absolute point constraints, are rational numbers.
These invariants are independent of the choice of the complex
structure on $\wh\P^2_9$ and~$(\wh\P^2_9,F)$.

\begin{lmm}[{\cite[p1019]{IPsum}}]\label{g1RES_lmm}
The genus~1 GW-invariants of $\wh\P^2_9$ in the fiber classes are described~by
$$\sum_{d=1}^{\i} d\,\GW_{1,d\ff}^{\wh\P^2_9}()q^d=G(q).$$
\end{lmm}

\begin{proof}
If $\wh\P^2_9$ is obtained by blowing up $\P^2$ at 9 general points, 
there is only one degree~$\ff$ holomorphic curve;
this is the proper transform of the unique cubic passing through the 9~points.
In this case, $\ov\fM_{1,0}(\wh\P^2_9,d\ff)$ consists of the $\si(d)$ unbranched covers of 
this cubic, all of which are regular and have an automorphism of order~$d$.
Thus,
$$d\, \GW_{1,d\ff}^{\wh\P^2_9}=\si(d),$$
as claimed.
\end{proof}

\begin{lmm}[{\cite[Lemma~14.8]{IPsum}}]
\label{RES_lmm}
Let $d,g\!\in\!\Z^{\ge0}$.
The absolute and relative degree $\fs_i\!+\!d\ff$ genus $g$ GW-invariants of $\wh\P^2_9$ and 
$(\wh\P^2_9,F)$
with $g$ point insertions satisfy
$$\GW_{g,\fs_i+d\ff;(1)}^{\wh\P^2_9,F}(\pt^g;\ff)=\GW_{g,\fs_i+d\ff}^{\wh\P^2_9}(\pt^g).$$
\end{lmm}

\begin{proof}
Let $J$ be a generic almost complex structure on $(\wh\P^2_9,F)$.
Suppose $\Si$ is a connected nodal genus~$g$ curve and $u\!:\Si\!\lra\!\wh\P^2_9$ 
is a degree $\fs_i\!+\!d\ff$ $J$-holomorphic stable map.
If $\Si_i$ is an irreducible component of $\Si$ such that $u\!:\Si_i\!\lra\!F$
is not constant, then the genus of~$\Si_i$ is at least one
and the sum of the genera of the remaining components of~$\Si$ is at most $g\!-\!1$.
Therefore, if the $g$ points are in general position, $u(\Si)$ does not contain all 
of them.
It follows that all of the maps contributing to the absolute invariant 
with $g$ point insertions are $F$-regular and thus contribute in the same way 
to the relative invariant.
\end{proof}

\begin{rmk}\label{14.4_rmk}
Lemma~\ref{RES_lmm} corrects the statement of \cite[Lemma~14.8]{IPsum}.
The latter and its proof incorrectly describe the IP-counts of $(\wh\P^2_9,F)$ 
as being indexed by the rim tori,
suggesting that the rim tori cover $\wh{F}_{\wh\P^2_9;(1)}$
is $\Z^2\!\times\!F$.
As explained in \cite[Example~6.8]{GWrelIP}, 
$\wh{F}_{\wh\P^2_9;(1)}\!\approx\!\C$ and there is 
only one IP-count appearing in the  statement of \cite[Lemma~14.8]{IPsum};
it is the count appearing in~\eref{ESsum_e15}.
Its relative constraint is the pullback of $1\!\in\!H^0(\wh{F}_{\wh\P^2_9;(1)};\Q)$
by a lifted evaluation map~\eref{evXVlift_e}; there is no indexing by the rim tori.
\end{rmk}

\subsection{Proof of Theorem~\ref{YZ_thm}}
\label{YauZaslow_subs}

\noindent
For each $d\!\in\!\Z^{\ge0}$, let
$$\GW_{1,\fs_i+d\ff}^{\wh\P^2_9}\big(\tau_1[\ff]\big) =
\deg\big([\ov\fM_{1,1}(\wh\P^2_9,\fs_i\!+\!d\ff)]^{\vir}\!\cap\!\psi_1\!\cap\!\ev_1^*1\big)
\equiv\int_{[\ov\fM_{1,1}(\wh\P^2_9,\fs_i+d\ff)]^{\vir}}\psi_1\ev_1^*1\,.$$
The $g\!=\!0$ case of~\eref{YZ_e1} is proved in~\cite{IPsum} by obtaining two different 
expressions~for 
\BE{Hdfn_e} H(q)\equiv\sum_{d=0}^{\i}\GW_{1,\fs_i+d\ff}^{\wh\P^2_9}\big(\tau_1[\ff]\big)q^d\EE
and setting them equal.

\begin{lmm}[{\cite[Lemma 15.1]{IPsum}}]\label{g1GWs_lmm}
Let $X$ be a symplectic 4-manifold with canonical class~$K_X$.
\begin{enumerate}[label=(\alph*),leftmargin=*]

\item For every $f\!\in\!H^2(X;\Q)$,
$$\GW_{1,0}^X(f)=\frac1{24}K_X\!\cdot\!f\,.$$

\item For every $A\!\in\!H_2(X;\Z)$ with $A\!\cdot\!K_X\!<\!0$ and $f\!\in\!H^2(X;\Q)$, 
\begin{equation*}\begin{split}
&\GW_{1,A}^X\big(\tau_1(f),\pt^{-K_X\cdot A-1}\big)
=\frac{f\!\cdot\!A}{24}\big(A\!\cdot\!A+K_X\!\cdot\!A\big)
\GW_{0,A}^X\big(\pt^{-K_X\cdot A-1}\big)\\
&\hspace{.1in}+\sum_{\begin{subarray}{c}A_0,A_1\in H_2(X;\Z)-0\\ A_0+A_1=A\end{subarray}}
\!\!\!\binom{\!-\!K_X\!\cdot\!A\!-\!1}{\!-\!K_X\!\cdot\!A_0\!-\!1}
(f\!\cdot\!A_0)(A_0\!\cdot\!A_1)\GW_{0,A_0}^X\big(p^{-K_X\cdot A_0-1}\big)
\GW_{1,A_1}^X\big(p^{-K_X\cdot A_1}\big).
\end{split}\end{equation*}

\end{enumerate}
\end{lmm}

\begin{proof}
(a) If $h\!:Y\!\lra\!X$ represents the Poincare dual of $f$ (after passing to a multiple
if necessary),
$$\big\{(y,[u,x_1])\!\in\!Y\!\times\!\ov\fM_{1,1}(X,0)\!:~h(y)\!=\!u(x_1)\big\}
\approx Y\!\times\!\ov\M_{1,1}$$
and the obstruction bundle is isomorphic to $\pi_1^*h^*TX\!\otimes\!\pi_2^*\bE^*$,
where $\bE\!\lra\!\ov\M_{1,1}$ is the Hodge line bundle. Thus,
\begin{equation*}\begin{split}
\GW_{1,0}^X(f)&=\blr{e\big(\pi_1^*h^*TX\!\otimes\!\pi_2^*\bE^*\big),Y\!\times\!\ov\M_{1,1}}\\
&=-\lr{h^*c_1(TX),Y}\lr{c_1(\bE),\ov\M_{1,1}}
=\frac1{24}K_X\!\cdot\!f.
\end{split}\end{equation*}
(b) Let $k\!=\!-K_X\!\cdot\!A$ and $\{H_i\},\{\check{H}_i\}\!\subset\!H^2(X;\Q)$ be dual
bases. Choose a representative $F\!\subset\!X$ for $f$ 
and $k\!-\!1$ general points $\pt_2,\ldots,\pt_k\!\in\!X$.
By the genus~1 topological recursion relation, illustrated in Figure~\ref{g1TRR_fig}
and explained in~\cite{Liu}, 
$$\psi_1=\frac1{12}\De_0+\De_{;1}\,,$$
where $\De_0,\De_{;1}\subset\ov\fM_{1,k}(X,A)$ are the virtual divisors 
whose virtually generic elements are morphisms from the genus~1 irreducible nodal
curve and from a smooth genus~1 curve with a rational tail which carries the first marked
point.\\

\begin{figure}
\begin{pspicture}(-1.5,-2.5)(10,0)
\psset{unit=.3cm}
\rput(10,-4){$\psi_1~~~=~~~\frac1{12}$}
\psarc(16,-5){1.5}{135}{45}\pscircle*(16,-2.88){.2}
\psline[linewidth=.05](17.06,-3.94)(15.06,-1.94)
\psline[linewidth=.05](14.94,-3.94)(16.94,-1.94)
\pscircle*(16,-6.5){.2}\rput(16,-7.3){\sm{2}}
\pscircle*(17.5,-5){.2}\rput(18,-5){\sm{1}}
\rput(20,-4){$+$}
\psarc(20,-4){3}{-45}{45}\rput(22.5,-7){\sm{$g\!=\!1$}}
\psline[linewidth=.05](22,-4)(27,-4)\rput(27,-5){\sm{$g\!=\!0$}}
\pscircle*(26,-4){.2}\rput(26,-3.2){\sm{1}}
\pscircle*(22.60,-2.5){.2}\rput(23.4,-2.5){\sm{2}}
\rput(30,-4){$+$}
\psarc(30,-4){3}{-45}{45}\rput(32.5,-7){\sm{$g\!=\!1$}}
\psline[linewidth=.05](32,-4)(37,-4)\rput(37,-5){\sm{$g\!=\!0$}}
\pscircle*(36,-4){.2}\rput(36,-3.2){\sm{1}}
\pscircle*(34.5,-4){.2}\rput(34.5,-3.2){\sm{2}}
\end{pspicture}
\caption{The genus 1 TRR on $\ov\fM_{1,2}(X,A)$}
\label{g1TRR_fig}
\end{figure}

\noindent
By the Kunneth decomposition of the diagonal in $X^2$ and the divisor relation, 
the degree of the intersection~of
$$\ov\fM_{1,k}'(X,A)\equiv\big\{[u,x_1,\ldots,x_k]\!\in\!\ov\fM_{1,k}(X,A)\!:~
u(x_1)\!\in\!f,~u(x_2)\!=\!\pt_2,\ldots,u(x_k)\!=\!\pt_k\big\}$$
with $\De_0$ is 
\begin{equation*}\begin{split}
\frac12\sum_i \GW_{0,A}^X(H_i,\check{H}^i,f,\pt^{k-1})
&=\frac12\sum_i (H_i\!\cdot\!A)(\check{H}_i\!\cdot\!A)(f\!\cdot\!A)
\GW_{0,A}^X(\pt^{k-1})\\
&=\frac12(A\!\cdot\!A)(f\!\cdot\!A)\GW_{0,A}^X(\pt^{k-1}).
\end{split}\end{equation*}
This gives the first term in our formula.\\

\noindent
The intersection of $\ov\fM_{1,k}'(X,A)$ with the components of~$\De_{;1}$ whose 
generic element restricts to a morphism of degree $A_1\!=\!0$ on
the genus~1 component~is the same as with the subset of these components consisting
of morphisms from domains with no marked points on the genus~1 component
(since the virtual complex dimension of $\ov\fM_{1,1}(X,0)$ is~$1$, it contains no elements
passing through any of the points $\pt_2,\ldots,\pt_k$).
Thus, similarly to the above, the degree of this intersection~is
\begin{equation*}\begin{split}
\sum_i\GW_{0,A}^X(H_i,f,\pt^{k-1})\GW_{1,0}^X(\check{H}^i)
&=\sum_i (H_i\!\cdot\!A)(f\!\cdot\!A)
\GW_{0,A}^X(\pt^{k-1})\frac{1}{24}K_X\!\cdot\!\check{H}_i\\
&=\frac1{24}(f\!\cdot\!A)(K_X\!\cdot\!A)\GW_{0,A}^X(\pt^{k-1});
\end{split}\end{equation*}
the first equality follows from part~(a).
This gives the second term in our formula.
The intersection of $\ov\fM_{1,k}'(X,A)$ with the components of~$\De_{;1}$ whose 
generic element restricts to a morphism of degree $A_1\!=\!A$ on
the genus~1 component~is empty, since 
the domain of any morphism in the intersection would contain a union of irreducible
components on which the morphism is of degree~0 and which carries at least one
of the last $k\!-\!1$ points (for stability), but $F$ does not contain any of the points
$\pt_2,\ldots,\pt_k$.\\

\noindent
For dimensional reasons, the intersection of $\ov\fM_{1,k}'(X,A)$ with 
the components of~$\De_{;1}$ whose generic element restricts to a morphism of 
degree $A_1\!\neq\!0$ on the genus~1 component and $A_0\!\neq\!0$ on the genus~0 tail
consists of morphisms from the domains so that the rational tail carries $-K_X\!\cdot\!A_0\!-\!1$
of the last $k\!-\!1$ marked points.
Thus, similarly to the above, the degree of this intersection~is
\begin{equation*}\begin{split}
&\sum_i  \binom{\!-\!K_X\!\cdot\!A\!-\!1}{\!-\!K_X\!\cdot\!A_0\!-\!1}
\GW_{0,A_0}^X\big(H_i,f,\pt^{-K_X\cdot A_0-1}\big)
\GW_{1,A_1}^X\big(\check{H}_i,p^{-K_X\cdot A_1}\big)\\
&\qquad=\sum_i 
\binom{\!-\!K_X\!\cdot\!A\!-\!1}{\!-\!K_X\!\cdot\!A_0\!-\!1}
(H_i\!\cdot\!A_0)(f\!\cdot\!A_0)\GW_{0,A_0}^X\big(\pt^{-K_X\cdot A_0-1}\big)
(\check{H}_i\!\cdot\!A_1)\GW_{1,A_1}^X\big(\pt^{-K_X\cdot A_1}\big)\\
&\qquad=\binom{\!-\!K_X\!\cdot\!A\!-\!1}{\!-\!K_X\!\cdot\!A_0\!-\!1}
(f\!\cdot\!A_0)(A_0\!\cdot\!A_1)
\GW_{0,A_0}^X\big(\pt^{-K_X\cdot A_0-1}\big)
\GW_{1,A_1}^X\big(\pt^{-K_X\cdot A_1}\big).
\end{split}\end{equation*}
This gives the last term in our formula.
\end{proof}

\begin{crl}[{\cite[(15.7)]{IPsum}}]\label{g1GWs_crl1}
The genus~0 and~1 GW-invariants of~$\wh\P^2_9$ satisfy
\BE{H_e1}H(q)=\frac{1}{12}\big(q\cF_0'(q)-\cF_0(q)\big)+\cF_0(q)\cdot G(q)\,. \EE
\end{crl}

\begin{proof}
We apply Lemma~\ref{g1GWs_lmm}(b) with $X\!=\!\wh\P^2_9$ and $A\!=\!\fs_i\!+\!d\ff$.
In this case,
$$-K_X\!\cdot\!(\fs_i\!+\!d\ff)=-1, \quad (\fs_i\!+\!d\ff)^2=2d\!-\!1,\quad
\ov\fM_{0,k}(X,d\ff)=\eset \qquad \forall~d\!\in\!\Z^+\,.$$
Thus,
$$H(q)=\sum_{d=0}^{\i}\frac{d\!-\!1}{12}\GW_{0,\fs+d\ff}^{\wh\P^2_9}()q^d
+\sum_{\begin{subarray}{c}d_0\in\Z^{\ge0},d_1\in\Z^+\\ d_0+d_1=d\end{subarray}}
\!\!\!\!\!\!\!\!\!\!\!\!\GW_{0,\fs+d_0\ff}^{\wh\P^2_9}()q^{d_0}\cdot d_1\GW_{1,d_1\ff}^{\wh\P^2_9}()q^{d_1}.$$
The claim now follows from the $g\!=\!0$ case of~\eref{Fgdfn_e} and
Lemma~\ref{g1RES_lmm}.
\end{proof}

\begin{crl}\label{g1GWs_crl2}
The genus~1 relative GW-invariants  of $(\P^1\!\times\!\T^2,F)$ satisfy
$$\GW_{1,\fs+d\ff;(1)}^{\P^1\times\T^2,F}\big(\tau_1[\ff];\pt\big)=\begin{cases}
-\frac{1}{12},&\hbox{if}~d=0;\\
d\,\GW_{1,d\ff}^{\P^1\times\T^2}(),&\hbox{if}~d\!>\!0.
\end{cases}$$
\end{crl}

\begin{proof}
We apply Lemma~\ref{g1GWs_lmm}(b) with $X\!=\!\P^1\!\times\!\T^2$
and $A\!=\!\fs\!+\!d\ff$ to the right-hand side of~\eref{P1T2lmm_e}.
In this case, 
$$-K_X\!\cdot\!(\fs\!+\!d\ff)=-2,\quad (\fs\!+\!d\ff)^2=2d, \qquad\forall~d\!\in\!\Z.$$
Thus,  
\BE{g1GWscrl2_e1}\GW_{1,\fs+d\ff}^{\P^1\times\T^2}\big(\tau_1[\ff],\pt\big)
=\frac{d\!-\!1}{12}\GW_{0,\fs+d\ff}^{\P^1\times\T^2}(\pt)
+\sum_{\begin{subarray}{c}d_0\in\Z^{\ge0},d_1\in\Z^+\\ d_0+d_1=d\end{subarray}}
\!\!\!\!\!\!\!\!\!\!
\GW_{0,\fs+d_0\ff}^{\P^1\times\T^2}(\pt)\cdot d_1\GW_{1,d_1\ff}^{\P^1\times\T^2}().\EE
Since the composition of a degree $\fs\!+\!d\ff$ morphism to $\P^1\!\times\!\T^2$
with the projection to the second factor is a degree~$d$ morphism to~$\T^2$
and there are no such morphisms from~$\P^1$ if $d\!\in\!\Z^+$, 
\BE{ovMempt_e}
\ov\fM_{0,1}(\P^1\!\times\!\T^2,\fs\!+\!d\ff),
\ov\fM_{0,1;(1)}^F(\P^1\!\times\!\T^2,\fs\!+\!d\ff)=\eset\qquad\forall\,d\!\in\!\Z^+.\EE
Thus, the first genus~0 term on the right-hand side of~\eref{g1GWscrl2_e1} vanishes unless $d\!=\!0$ 
and the second unless $d_0\!=\!0$; in the exceptional cases, they equal~1.
The claim now follows from Lemma~\ref{P1T2_lmm}. 
\end{proof}

\noindent
We next obtain a second expression for $H(q)$ by applying 
the symplectic sum formula to the decomposition
\BE{Esplit_e1} 
\wh\P^2_9=\wh\P^2_9\#_F(\P^1\!\times\!\T^2)\EE
and moving the fiber constraint to the $\P^1\!\times\!\T^2$ side.
Since $\cR_{\P^1\times \T^2}^F\!=\!0$, the homomorphism
$$\#\!: H_2(\wh\P^2_9;\Z)\!\times_F\!H_2(\P^1\!\times\!\T^2;\Z)\lra H_2(\wh\P^2_9;\Z)$$
is well-defined; see \cite[Corollary~4.2(2)]{GWrelIP}.
Since 
$$(a_1\fs_1\!+\!\ldots\!+\!a_9\fs_9+d'\ff)\cdot F
=(a\fs+d''\ff)\cdot F$$ 
if and only if $a_1\!+\!\ldots\!+\!a_9=\!a$ and $\fs_i\!\#\!\fs\!=\!\fs_i$,
the symplectic sum formula gives 
\BE{RESsplit_e1}\begin{split}
\GW_{1,\fs_i+d\ff}^{\wh\P^2_9}\big(\tau_1[\ff]\big)
&=\sum_{\begin{subarray}{c}d',d''\in\Z^{\ge0}\\ d'+d''=d\end{subarray}}\!\!\!\!\!
\GW_{0,\fs_i+d'\ff;(1)}^{\wh\P^2_9,F}(;\ff)\,
\GW_{1,\fs+d''\ff;(1)}^{\P^1\times\T^2,F}\big(\tau_1[\ff];\pt\big)\\
&\qquad+
\sum_{\begin{subarray}{c}d',d''\in\Z^{\ge0}\\ d'+d''=d\end{subarray}}\!\!\!\!\!
\GW_{1,\fs_i+d'\ff;(1)}^{\wh\P^2_9,F}(;\pt)\,
\GW_{0,\fs+d''\ff;(1)}^{\P^1\times\T^2,F}\big(\tau_1[\ff];\ff\big),
\end{split}\EE
where the relative constraints are listed after the semi-columns.\\

\noindent
By~\eref{ovMempt_e},
$$\GW_{0,\fs+d''\ff;(1)}^{\P^1\times\T^2,F}\big(\tau_1[\ff];\ff\big)=0
\qquad\forall\,d''\!\in\!\Z^+\,.$$
On the other hand, the morphism
$$\big\{[u,x,y]\!\in\!\ov\fM_{0,1;(1)}^F(\P^1\!\times\!\T^2,\fs)\!:~u(x)\!\in\!\ff\big\}
\lra \ff, \qquad [u,x,y]\lra u(x),$$
is an isomorphism and the restriction of $\psi_1$ under this isomorphism is 
the first chern class of the conormal bundle to a fiber $\T^2$ in $\P^1\!\times\!\T^2$,
i.e.~$0$.
Thus, the second sum in~\eref{RESsplit_e1} vanishes.\\

\noindent
Combining~\eref{RESsplit_e1} with the above conclusion,
the $g\!=\!0$ case of Lemma~\ref{RES_lmm}, 
and Corollary~\ref{g1GWs_crl2}, we find~that
$$\GW_{1,\fs_i+d\ff}^{\wh\P^2_9}\big(\tau_1[\ff]\big)
=-\frac1{12}\GW_{0,\fs_i+d\ff}^{\wh\P^2_9}()
+\sum_{\begin{subarray}{c}d'\in\Z^{\ge0},d''\in\Z^+\\ d'+d''=d\end{subarray}}
\!\!\!\!\!\!\!\!\!\!\!\!
\GW_{0,\fs_i+d'\ff}^{\wh\P^2_9}()\,d''\GW_{1,d''\ff}^{\P^1\times\T^2}().$$
Along with~\eref{Hdfn_e}, the $g\!=\!0$ case of~\eref{Fgdfn_e}, and Lemma~\ref{P1xT2_lmm1}, 
this identity gives
\BE{H_e2} H(q)=-\frac{1}{12}\cF_0(q)+\cF_0(q)\cdot 2G(q). \EE
By \eref{F0init_e}, \eref{H_e1}, and~\eref{H_e2},
\BE{F0ODE_e} \cF_0(0)=1, \qquad q\frac{\nd }{\nd q}\log \cF_0(q)=12G(q).\EE
Since
$$\frac{1}{12}q\frac{\nd }{\nd q}\log\bigg(\prod_{d=1}^{\i}(1\!-\!q^d)\bigg)^{-12}
=\sum_{d=1}^{\i}\frac{dq^d}{1-q^d}=\sum_{d=1}^{\i}\si(d)q^d=G(q),$$
\eref{F0ODE_e} implies the $g\!=\!0$ case of~\eref{YZ_e1}.
The full statement of~\eref{YZ_e1} follows from this case and the next lemma.

\begin{lmm}\label{RES_lmm3}
For every $g\!\in\!\Z^+$,
\BE{RES3_e}\cF_g(q)=\cF_{g-1}(q)\cdot qG'(q).\EE
\end{lmm}

\begin{proof}
In light of Lemmas~\ref{P1xT2_lmm2} and~\ref{RES_lmm},
this statement is equivalent to~\eref{ESsum_e15b}, which was obtained based on
the approach to the symplectic sum formula in~\cite{IPsum}.
We now give a proof by
 applying the usual symplectic sum theorem to the splitting~\eref{Esplit_e1} and moving
one point to the $\P^1\!\times\!\T^2$ side.
Since $g\!-\!1$ points stay on the $\wh\P^2_9$ side, the genus on the $\wh\P^2_9$ side
in the symplectic sum formula
must be at least~$g\!-\!1$ for the invariants not to vanish.
Thus, similarly to~\eref{RESsplit_e1}, we~obtain
\BE{RESsplit_e2}\begin{split}
\GW_{g,\fs_i+d\ff}^{\wh\P^2_9}\big(\pt^g\big)
&=\sum_{\begin{subarray}{c}d',d''\in\Z^{\ge0}\\ d'+d''=d\end{subarray}}\!\!\!\!\!
\GW_{g-1,\fs_i+d'\ff;(1)}^{\wh\P^2_9,F}\big(\pt^{g-1};\ff\big)
\GW_{1,\fs+d''\ff;(1)}^{\P^1\times\T^2,F}(\pt;\pt)\\
&\qquad+
\sum_{\begin{subarray}{c}d',d''\in\Z^{\ge0}\\ d'+d''=d\end{subarray}}\!\!\!\!\!
\GW_{g,\fs_i+d'\ff;(1)}^{\wh\P^2_9,F}\big(\pt^{g-1};\pt\big)
\GW_{0,\fs+d''\ff;(1)}^{\P^1\times\T^2,F}(\pt;\ff).
\end{split}\EE
By \cite[Theorem~1.1]{GWrelIP},
\BE{RelGWvan_e}  \GW_{g,\fs_i+d'\ff;(1)}^{\wh\P^2_9,F}\big(\pt^{g-1};\pt\big)=0\,,\EE
This particular statement holds because the relative evaluation morphism~\eref{evdfn_e2}
factors through the lift to $\wh{F}_{\wh\P^2_9;(1)}\!\approx\!\C$;
see \cite[Example~6.8]{GWrelIP}.
Combining~\eref{RESsplit_e2} with~\eref{RelGWvan_e}
and Lemma~\ref{RES_lmm}, we find~that 
$$\sum_{d=0}^{\i}\GW_{g,\fs_i+d\ff}^{\wh\P^2_9}\big(\pt^g\big)q^d
=\sum_{d',d''\in\Z^{\ge0}}\!\!\!\!\!\GW_{g-1,\fs_i+d'\ff}^{\wh\P^2_9}\big(\pt^{g-1}\big)q^{d'}
\cdot\GW_{1,\fs+d''\ff;(1)}^{\P^1\times\T^2,F}(\pt;\pt)q^{d''}.$$
The claim now follows from~\eref{Fgdfn_e} and Lemma~\ref{P1xT2_lmm2}.
\end{proof}

\begin{rmk}\label{15.3_rmk}
Lemma~\ref{g1GWs_lmm}(b) extends \cite[Lemma~15.1(b)]{IPsum} from 
the $K_X\!\cdot\!A\!=\!-1$ case, using the same argument;
the $K_X\!\cdot\!A\!=\!-2$ case is needed to obtain the crucial identity \cite[(15.8)]{IPsum},
i.e.~\eref{H_e2} above.
Our use of~\eref{RelGWvan_e} avoids the need for \cite[Lemma~15.2(c)]{IPsum}
and directly establishes the last equation in \cite[Section~15.3]{IPsum}.
The statement of the symplectic sum formula in the middle of \cite[p1020]{IPsum}
is wrong, as it should involve relative GW-invariants;
as stated, the last factor is not even zero-dimensional.
The next displayed expression in~\cite{IPsum} has the same problem
and does not lead to \cite[(15.8)]{IPsum}.
Because of problems with these formulas, Lemma~\ref{P1T2_lmm} never even arises
in~\cite{IPsum}.
\end{rmk}

\vspace{.2in}

\noindent
{\it Simons Center for Geometry and Physics, SUNY Stony Brook, NY 11794\\
mtehrani@scgp.stonybrook.edu}\\

\noindent
{\it Department of Mathematics, SUNY Stony Brook, Stony Brook, NY 11794\\
azinger@math.sunysb.edu}\\

\end{document}